\documentclass[a4paper]{article}

\usepackage{hyperref}

\usepackage{amsmath,amssymb,amsthm}

\numberwithin{equation}{section}

\usepackage{shortvrb, makeidx}
\newtheorem{definition}{\hspace{2em}Definition}[section]

\newtheorem{theorem}[definition]{\hspace{2em}Theorem}

\newtheorem{lemma}[definition]{\hspace{2em}Lemma}
\newtheorem{proposition}[definition]{\hspace{2em}Proposition}
\newtheorem{corollary}[definition]{\hspace{2em}Corollary}

\newtheorem{remark}{\hspace{2em}Remark}[section]

\newcommand{\curl}{\operatorname{curl}}
\newcommand{\Div}{\operatorname{div}}
\renewcommand{\epsilon}{\varepsilon}

\makeindex

\title{Local Well-posedness of the three dimensional compressible Euler--Poisson equations with physical vacuum}
\author{Xumin Gu\footnote{School of Mathematical Sciences, Fudan University, Shanghai 200433, P.R.China.
{\it Email:xumingu11@fudan.edu.cn}}\ \  and Zhen Lei\footnote{School of Mathematical Sciences; LMNS and Shanghai
Key Laboratory for Contemporary Applied Mathematics, Fudan University, Shanghai 200433, P. R.China. {\it Email: zlei@outlook.com, zlei@fudan.edu.cn}}}

\date{}

\begin{document}
\maketitle

\begin{abstract}
This paper is concerned with the three dimensional compressible Euler--Poisson
equations with moving physical vacuum boundary condition. This fluid system is
usually used to describe the motion of a self-gravitating inviscid gaseous star. The local
existence of classical solutions for initial data in certain weighted Sobolev spaces is established
in the case that the adiabatic index satisfies $1 < \gamma < 3$. 
\end{abstract}

\section{Introduction}
The motion of a self-gravitating inviscid gaseous star in the universe can be described by the
following free boundary problem for the compressible Euler equations coupled with the Poisson equation:
\begin{subequations}
\label{ep3d}
\begin{align}
  \label{euler-poisson-3d}
  \rho_t + \nabla_{\eta} \cdot (\rho u) &= 0&\text{in}\ \ \Omega(t),\\
  \rho [u_t + u \cdot \nabla_{\eta} u] + \nabla P &= \rho \nabla_{\eta} \phi &\text{in}\ \  \Omega(t),\label{f3d}\\
  - \Delta_{\eta\eta}\phi &= 4 \pi  g\rho  &\text{in}\ \  \Omega(t),\label{9999999999999}\\
  \nu (\Gamma(t)) &= u\cdot n(t) &\text{on}\ \  \Gamma(t),\\
  (\rho, u) &= (\rho_0, u_0) &\text{in}\ \  \Omega(0).\label{5353}
\end{align}
\end{subequations}
The open, bounded domain $\Omega(t) \subset \mathbb{R}^3$ denotes the changing domain
occupied by the gas. $\Gamma(t) := \partial \Omega(t)$ denotes the moving vacuum boundary, $\nu (\Gamma(t))$
denotes the velocity of $\Gamma(t)$, and $n(t)$ denotes the exterior unit normal vector to $\Gamma(t)$.
The density of gas $\rho > 0$ in $\Omega(t)$ and $\rho = 0$ in $\mathbb{R}^3\setminus \Omega(t)$. $u$
denotes the Eulerian velocity field, $P$ denotes the scalar pressure, $\phi$ is the potential function
of the self-gravitational force, and $g$ is the gravitational constant. We consider a polytropic gas star,
then the equation of state is given by:
\begin{equation}
P= C_{\gamma}\rho^{\gamma} \ \ \ \text{for}\ \ \gamma > 1,
\label{poly}
\end{equation}
where $C_{\gamma}$ is the adiabatic constant. We set both $g$ and $C_{\gamma}$ to be unity. We refer the readers
 to \cite{C,Cox_68} for more details of the related background on this system.

The sound speed of equations \eqref{ep3d} is given by $c:=\sqrt{\partial P / \partial \rho}$, and $N$ denotes the outward unit normal to the initial boundary $\Gamma := \partial \Omega(0)$, then the condition
\begin{equation}
 - \infty < \dfrac{\partial c_0^2}{\partial N} <0 \ \ \text{on}\ \ \  \Gamma
 \label{vacuum}
\end{equation}
defines a ``physical vacuum'' boundary, where $c_0(\cdot) =c(\cdot,0)$. This definition of physical vacuum was motivated by the case
of the Euler equations with damping studied in \cite{L_1996, LT_2000}. For more
details and the physical background of this concept, please see \cite{MJ_11,MJ_12, L_1996, LT_1997,LT_2000,Yang_06}.

The physical vacuum condition \eqref{vacuum} is equivalent to the requirement that
\begin{equation}
\label{va}
\dfrac{\partial \rho_0^{\gamma-1}}{\partial N} <0 \,\,\text{on}\,\,\Gamma.
\end{equation}
This condition is necessary for the gas particles on the boundary to accelerate.
Since $\rho_0>0$ in $\Omega$, \eqref{va} implies that for some positive constant $C$, when $x\in \Omega$ is close enough to the vacuum boundary $\Gamma$, then
\begin{equation}
\label{vacuum2}
\rho_0^{\gamma-1}(x)\geq C \text{dist}(x,\Gamma).
\end{equation}

When the physical vacuum boundary condition is assumed, the compressible Euler equations become a degenerate and
characteristic hyperbolic system, then the classical theory of hyperbolic systems can not be
directly applied. The local existence theory of classical solutions featuring the physical vacuum boundary was only established recently
(cf. \cite{DS_2009, DS_2010, MJ_2009, MJ_2010}). In \cite{DS_2010}, Coutand and Shkoller constructed $H^2$-type solutions in the Lagrangian coordinates based on Hardy's inequalities and the degenerate parabolic approximation. Their solutions can be regarded as degenerate viscosity solutions in some sense. Independently, in \cite{MJ_2010},
Jang and Masmoudi proved a local well-posed result by using a different approach. They treated the Euler equations in the Lagrangian coordinates as a hyperbolic type equation and obtained an energy type a prior estimate in weighted Sobolev spaces in a natural and clear way. The key
observation of these works is that, in the Lagrangian coordinates, the Euler equations have a natural weight and this weight can be used to overcome the possible singularity near the boundary. But this weight will be lost after applying standard spatial derivatives on the equations.
However, after applying the tangential derivatives (with respect to the boundary) on the equations, the the same weight will be kept. This factor is guaranteed by the Hardy inequality. Consequently, one may perform an $L^2$ type high order energy estimate with respect to the tangential derivatives and time derivative. Another observation is that the higher-order energy estimates show that the highest derivative terms contain three main parts: the weighted $L^2$ norm of the full Eulerian gradient, the Eulerian divergence and the Eulerian vorticity (see \cite[Section 4]{MJ_2010}). The first two terms have a positive sign and the Eulerian vorticity has a good estimate due to the transport-type structure of the vorticity equation. Combining these observations and the weighted Sobolev embedding inequality, the estimates for high order time and tangential derivatives can be established. Then by using the Hodge-type elliptic estimates (cf. \cite{DS_2010}) or adding weight on the normal derivatives (cf. \cite{
MJ_2010}), the estimates of the normal derivatives and hence a closed a prior estimates can be obtained.

Vacuum states also arise in coupled systems, such as Euler--Poisson system, the magnetohydrodynamics (MHD), etc. In fact, the physical
vacuum behavior appears naturally in the stationary solutions or the spherically symmetric stationary solutions for
the Euler--Poisson system \eqref{ep3d} (cf.  \cite{Xin_2013,J_09}), and the vacuum states are even richer in the magnetohydrodynamics due to the interplay between the scalar pressure and the anisotropic magnetic stress (cf.
\cite{MJ_2010, MJ_11}). However, the rigorous
study on those coupled system seems to be widely open. In this paper, we study the local well-posedness of Euler--Poisson system in multidimensional case, a
coupled system describing astrophysical flows where gravity plays an important role. This is a continuation of our earlier work on the one-dimensional case (see \cite{Gu_2011}).
Below we will focus on the three-dimensional case. The two-dimensional case is similar.
Our main result is the following theorem, where the definition of $E$ and $M_0$ can be found at the beginning of Section \ref{ee}:
\begin{theorem}\label{theorem}(Local well-posedness)
Let $1<\gamma<3$ and $M_0$ be two positive constants. Assume that $\rho_0 >0$ in $\Omega$ and the physical vacuum condition \eqref{vacuum} holds. Then
there exists a unique solution to \eqref{eq:ep} (and hence \eqref{ep3d}) on $[0,T]$ for some sufficiently small $T >0$  such that
\begin{equation}
\sup_{t \in [0,T]} E(t) \leq 2M_0.
\end{equation}
\end{theorem}
We explain the main steps and the main difficulties to prove the local well-posedness result slightly below. Firstly,  motivated by Coutand,
Shkoller \cite{DS_2010} and Jang, Masmoudi \cite{MJ_2010}, we use the Lagrangian coordinates to reduce the original free boundary problem to the system in a fixed domain.
Since we deal with the coupled system now, the main difference arising in our system from the Euler equations is the presence of the potential force term $\nabla_{\eta}\phi$ in \eqref{f3d}.
In order to handle this term, we will give an explicit formula for it
in a similarly way as we did in \cite{Gu_2011} for 1D case, and then we use this formula to reformulate the system as the Euler equations with external forcing term.
However, in the three dimensional case, the potential force term is non-local and brings many technical difficulties, in particular,
for finding solutions to the parabolic approximations and
the estimates of the approximated solutions.

More precisely, the force term can be written as
\begin{equation*}
\int_{\Omega}\dfrac{[{F^{*}}_i^k\partial_k(\rho_0J^{-1})](z,t)\,dz}{|\eta(x,t)-\eta(z,t)|}
\end{equation*}
in the Lagrangian coordinates.
When we act derivative (for instance, a time derivative $\partial_t$) on it, the singularity of the kernel may increase:
\begin{equation*}
\int_{\Omega}\dfrac{(\eta(x,t)-\eta(z,t))\cdot(v(x,t)-v(z,t))}{|\eta(x,t)-\eta(z,t)|^3}[{F^{*}}_i^k\partial_k(\rho_0J^{-1})](z,t)\,dz+\cdots
\end{equation*}
 However, after a careful calculation, we can find that the kernel can still be integrable in $\Omega$. In fact,
 we will make use of the Sobolev embedding inequality ($C^{\alpha}$ estimate) and Taylor's formula to reduce the
 increased singularity and overcome the difficulty. For instance,
 \begin{align*}
 &|\dfrac{(\eta(x,t)-\eta(z,t))\cdot(v(x,t)-v(z,t))}{|\eta(x,t)-\eta(z,t)|^3}|\\\lesssim&|\dfrac{(\eta(x,t)-\eta(z,t))\cdot(Dv(x^*)(x-z))}{|x-z|^3}|\\\lesssim&\|Dv\|_{L^{\infty}(\Omega)}\dfrac{1}{|x-z|}.
 \end{align*}
The detail of this calculation is carried out in Section \ref{sg1} and Section \ref{sg2}. In this way, we can get a good estimate for $\nabla_{\eta}\phi$.

Secondly, motivated by \cite{DS_2010}, we introduce a degenerate parabolic approximation for the construction of approximated solutions by adding an artificial viscosity, which is different from the one in \cite{DS_2010} due to the presence of the potential force. The viscosity parameter is $\kappa$ and we call the parabolic approximation as $\kappa$-equations for convenience. We will use a fixed-point scheme to get solutions for $\kappa$-equations and construct the contract map in two steps. Suppose we have an approximate solution $v^{(n)}$, then in the first step, we derive a degenerate heat-type
equation by acting the Eulerian divergence on the $\kappa$-equations and introducing an intermediate variable $X$, which is corresponding to $\rho_0\Div_{\eta}v$ in some sense. We call this equation as $X$-equation and use the Galerkin scheme to find the solution $X^{(n)}$ to the linearized problem (with respect to $v^{(n)}$) of $X$-equation. We can also get an energy type estimate of $X^{(n)}$, which will be used to construct contract map in the second step. In this process, we will make the
fundamental use of the higher-order Hardy-type inequality introduced in \cite{DS_2010}. We mention that our intermediate variable $X$ is different from the one used in \cite{DS_2010}. By using our intermediate variable, the improvement
of the space regularity for $X^{(n)}$ will be easier and clearer with less computations. However,
due to our choice of the intermediate variable, the symmetric structure of $X$-equation is different from that in \cite{DS_2010}, we will need to make use of $\kappa$ to construct the contract map. In the second step, we derive a new approximate velocity field $v^{(n+1)}$ by a linear elliptic system of equations, which is constructed by defining the divergence, curl and vertical component on the boundary of $v^{(n+1)}$ by using $v^{(n)}$, $X^{(n)}$ and their derivatives directly. The idea is that we will add a proper small perturbation to $v^{(n)}$'s divergence, curl and vertical component on the boundary to define $v^{(n+1)}$'s divergence, curl and vertical component on the boundary. To achieve this goal, we will make use of the evolution equations for $\Div_{\eta}v$, $\curl_{\eta}v$ in the domain and the evolution equation for normal component $v^3$ on the boundary. The first equation is derived from our definition of intermediate variable $X$, the other two equations are derived from $\kappa$-equations. For $v^{(n)}$ and $X^{(n)}$, these equations do not hold, but they can be regarded as small perturbations from zero. Then combining with the energy type estimate we get in the first step, we can construct the contract map $\Theta:v^{(n)}\mapsto v^{(n+1)}$ and use the fixed-point scheme to get solutions to the $\kappa$-equations.

Lastly, we will derive $\kappa$-independent a prior estimates for the approximated solutions. This kind a prior estimates were introduced in \cite{DS_09, DS_2010} for the Euler equations. We combine some ideas from \cite{MJ_2010} with \cite{DS_2010} to
 carry out a little simpler proof in Section \ref{ee} than that in \cite[Section  9]{DS_2010}. The main strategy of a prior estimates is that: we first get the curl estimate and then perform the $L^2$ type high order energy estimates with respect to time derivative and tangential derivatives. In this way, we can control the $L_t^{\infty}L^2_x$ norms of $\sqrt{\rho_0}\partial_t^{2s}\bar\partial^{4-s} v$ and $\rho_0\partial_t^{2s}\bar\partial^{4-s} D\eta$, and hence the $L_t^{\infty}L^2_x$ norms of $\partial_t^{2s}\bar\partial^{4-s} \eta$ by using the fundamental theorem of calculus and the weighted Sobolev embedding inequalities.  Next, we use Hodge-type elliptic estimates to control the normal derivative of $\eta$ and use bootstrap arguments to get a closed a prior estimates. Then the local existence of the Euler--Poisson equations is followed by taking weak limit $\kappa\rightarrow0$.
Below we will first focus on the case for $\gamma=2$. The general case for $1<\gamma<3$ is treated in Section \ref{ga3}.

Now we briefly review some related theories and results from various aspects.
For the Euler--Poisson equations, the linearized stability of spherically symmetric stationary solutions was studied in \cite{Lin_97}, and the existence theory for the stationary solutions has been proved by Deng,
Liu, Yang, and Yao in \cite{TY_2002}. Jang studied nonlinear instability of the spherically symmetric Euler--Poisson system for polytropic gases
with adiabatic exponents $\gamma=6/5$ and $6/5 < \gamma< 4/3$ around the Lane--Emden equilibrium star configurations in \cite{J_08, J_13}. Especially, in \cite{J_13}, the
boundary behavior of compactly supported Lane--Emden solutions was characterized
by the physical vacuum condition \eqref{vacuum}. Luo, Xin and Zeng considered the spherically symmetric motions with physical vacuum in \cite{Xin_2013}
and gave a much simpler proof for its local existence without imposing the compatibility condition of the first derivative being zero at
the center of symmetry. A novelty uniqueness result is also obtained for $1 < \gamma \leq 2$. For the Navier--Stokes--Poisson equations,
Jang \cite{J_09} studied spherically symmetric and isentropic motion and captured the physical boundary behavior, and in \cite{JI_2013}, Jang and Tice also established both linear and nonlinear instability for this spherically symmetric system.  Li, Matsumura and
Zhang \cite{HL_2010} studied optimal decay rate for the Navier--Stokes--Poisson equations.

For compressible fluids, some of the early developments in the theory of vacuum states for compressible gas dynamics can be found in \cite{Lin_87,Liu_80}. In \cite{Makino_1986}, Makino proved the local-in-time existence of solutions with boundary condition $\rho=0$ for some non-physical restrictions on the initial data. He regarded this problem as a Cauchy problem with compact support initial data in the whole space. Similar methodology was used by DiPerna \cite{Di_83} and Lions, Perthame, Souganidis \cite{Lions_96}. However, this kind method can not track the position of vacuum boundary. Lindblad \cite{Lind_2005} considered the compressible liquid for general case of initial data. He transformed the equations in the Lagrangian coordinates and proved the local-in-time existence with the vacuum boundary condition $P=0$ and Taylor sign condition $\nabla P_0\cdot n |_{\Gamma}<0$. Trakhinin \cite{Tra_09} proved Lindblad's result in a different way. He transformed the equations into a  symmetric hyperbolic system in a fixed domain while he still used the Eulerian coordinates. Then he used Alinac's good unknown to derive a prior estimates with the manifest loss of regularity for the linearized problem, and by using Nash-Moser iteration, he proved the local well-posedness. In \cite{Xu_05}, Xu, Yang proved a local-in-time existence result for the small perturbation of a planar wave to the one dimensional Euler equations with damping. We also refer the reader to H. Li, J. Li,
Xin \cite{HL_2008}, Luo, Xin, Yang \cite{XinL}, and Xin \cite{Xin} for extended results, such as the compressible Navier--Stokes equations with vacuum.

For incompressible flows, Wu solved the local well-posedness for the irrotational problem, with no surface tension in all dimensions in \cite{Wu_1997} and \cite{Wu_1999}. Wu also proved the global well-posedness in three dimension and almost global well-posedness in two dimension in \cite{Wu_11, Wu_09}. In \cite{Lindblad_2005}, assuming the Rayleigh-Taylor sign condition,
Lindblad proved the local existence of solutions for rotational flows with no surface tension. In \cite{Masmoudi_2012}, Masmoudi and Rousset studied the inviscid limit of the free boundary Navier--Stokes equations and hence obtained a local well-posedness for both the Navier--Stokes equations and the Euler equations. For the problem with surface tension, Schweizer proved the existence for the general three-dimensional irrotational problem in \cite{Schweizer_2005}. And we also mention the works by Ambrose and Masmoudi \cite{AM}, Coutand and Shkoller \cite{DS}, Lannes \cite{Lannes} and P. Zhang and Z. Zhang \cite{ZZ}.
Very recently, the global well-posedness results of two dimensional gravity water waves system for small localized data were independently obtained by Alazard, Delort \cite{AD_13}, and Ionescu, Pusateri \cite{Ion_13}. We refer the reader to their papers and the references therein for a more extensive bibliography.

The rest of the paper is organized as follows. First, we focus on the case for $\gamma=2$. In Section \ref{s2}, we formulate the problem in the Lagrangian coordinates.
In Section \ref{s3}, we introduce our notations used in the paper.
In Section \ref{s4}, we recall some preliminary analysis lemmas which will be used frequently throughout the paper. In Section \ref{sec5}, we recall some useful identities which will be used for constructing approximated solutions and getting a prior estimates. In Section \ref{sec6}--Section \ref{s5}, we introduce a degenerate parabolic approximation and solve it by a fixed-point method. In Section \ref{ee}--\ref{s7}, we derive the uniform a priori estimates for approximated solutions and prove the local well-posedness. Then, in Section \ref{ga3}, we discuss the general case for $1 < \gamma <3$. 

\section{Lagrangian formulation}\label{s2}
In this section, motivated by \cite{DS_2010,MJ_2010}, we use Lagrangian coordinates to transform the free boundary problem to a fixed domain problem.
We denote $\eta$ as Eulerian coordinates and denote $x$ as Lagrangian coordinates, which means that $\eta(x,t)$ denotes the ``position'' of the gas particle $x$ at time $t$. And we have that
\begin{equation}
  \partial_t \eta = u \circ \eta \ \ \text{for}\ \  t > 0 \ \ \text{and}\ \  \eta(x,0)=x,
  \label{hhhhh}
\end{equation}
where $\circ$ denotes the composition $[u\circ\eta](x,t)=u(\eta(x,t),t)$. We also use the following notations:
\begin{equation}
  \begin{split}
    v &= u\circ \eta\quad(\text{Lagrangian velocity}),\\
    f &= \rho \circ \eta\quad(\text{Lagrangian density}),\\
    \Phi &= \phi \circ \eta\quad(\text{Lagrangian potential field}),\\
    F &= [D\eta]\quad(\text{Deformation tensor}),\\
    J &= \text{det} F\,\,\,\,(\text{Jacobian determinant}),\\
    F^{*} &= JF^{-1} \,\,\,(\text{Cofactor matrix}).
  \end{split}
  \label{lagrang}
\end{equation}
Then equation \eqref{hhhhh} can be rewritten in the Lagragian coordinates as
\begin{equation*}
\label{etav}
v(x,t)=\eta_t(x,t).
\end{equation*}
This equation holds throughout the paper.

\subsection{The Lagrangian version of the system}
Noticing (\ref{hhhhh}) and (\ref{lagrang}), the Lagrangian version of system (\ref{euler-poisson-3d})--(\ref{5353}) can be written in the fixed reference domain $\Omega$ as
\begin{subequations}
\label{eq:ep}
\begin{align}
  \label{euler-poisson-lag}
  f_t + f {F^{-1}}_i^k\partial_k (v^i) &= 0 &\text{in}&\ \ \Omega \times (0,T],\\
  f v_t^i + {F^{-1}}_i^k\partial_k (f^{2}) &= f {F^{-1}}_i^k\partial_k (\Phi) &\text{in}&\ \  \Omega \times (0,T],\\
  - ({F^{-1}}_i^k\partial_k)^2\Phi &=  f &\text{in}&\ \  \Omega \times (0,T],\\
  (f, v, \eta) &= (\rho_0, u_0, e) &\text{in}&\ \ \Omega \times \{t=0\},\label{5r}
\end{align}
\end{subequations}
where $e$ denotes the identity map on $\Omega: e(x)=x$. We also used Einstein's summation convention to simplify the description.

To avoid the use of local coordinate charts which are necessary for arbitrary geometries, and to simplify our expression, we will assume that the initial domain at time $t=0$ is given by
\begin{equation}
\label{domain}
\Omega=\mathbb{T}^2\times(0,1),
\end{equation}
where $\mathbb{T}^2$ denotes the 2-torus. Then, the reference vacuum boundary is comprised of the bottom and top of the domain $\Omega$ so that
\begin{equation*}
\Gamma =\mathbb{T}^2 \times ({0}\cup {1}).
\end{equation*}

By the conservation law of mass \eqref{euler-poisson-lag}, we have
\begin{equation}
  f = \rho \circ \eta = \rho_0 J^{-1}.
  \label{lag_mass}
\end{equation}
Then we can rewrite the compressible Euler--Poisson system as
\begin{subequations}
\label{eq:lag}
\begin{align}
  \label{lag}
  \rho_0 v_t + {F^{*}}_i^k\partial_k(\rho_0^{2}J^{-2})  &= \rho_0 {F^{-1}}_i^k\partial_k(\Phi) &\text{in}&\ \  \Omega\times (0,T],\\ \label{cffffff}
  - {F^{*}}_i^j\partial_j({F^{-1}}_i^k\partial_k(\Phi))  &=  \rho_0 &\text{in}&\ \  \Omega \times (0,T], \\ \label{field}
(v, \eta) &= (u_0, e) &\text{in}&\ \ \Omega\times \{t=0\},\\
  \rho_0 &= 0 &\text{on}&\ \  \Gamma,
\end{align}
\end{subequations}
where $\rho_0(x)$ also satisfies that $\rho_0(x) \geq C \text{dist}(x,\Gamma)$ for $x\in\Omega$ close to $\Gamma$.

\subsection{The formula for the potential force}
\label{gfo}
In this subsection, we will give an explicit formula for the potential force $\nabla_{\eta}\phi$ in (\ref{f3d}) and the corresponding term ${F^{-1}}_i^k\partial_k \Phi$ in (\ref{lag}).

Firstly, since the density of gas $\rho > 0$ in $\Omega(t)$ and $\rho = 0$ in $\mathbb{R}^3\setminus \Omega(t)$, then from equation (\ref{9999999999999}), by the Newtonian potential, we have
\begin{equation}
  \phi(\eta,t) = \int_{\mathbb{R}^3}\dfrac{ \rho(y,t)}{|\eta-y|}\,dy=\int_{\mathbb{R}^3}\dfrac{\rho(\eta-y,t)}{|y|}\,dy,
  \label{fieldfor}
\end{equation}
and
\begin{equation}
\label{id:phi}
\Phi(x,t)=\int_{\mathbb{R}^3}\dfrac{1}{|y|}\rho(\eta(x,t)-y,t)\,dy,
\end{equation}
under the assumption that $|\phi(x)|\rightarrow 0$ when $|x|\rightarrow +\infty$. Therefore,
\begin{equation}
\label{g}
g=\nabla_{\eta}\phi=\int_{\mathbb{R}^3}\dfrac{1}{|y|}\nabla_{\eta}\rho(\eta-y,t)\,dy
=\int_{\mathbb{R}^3}\dfrac{1}{|\eta-y|}\nabla_y\rho(y,t)dy.
\end{equation}
Due to $\rho = 0$ in $\mathbb{R}^3\setminus \Omega(t)$, we have that $\nabla_y\rho(y,t)=0, y\in\mathbb{R}^3\setminus \bar\Omega(t)$ and
$$g=\int_{\Omega(t)}\dfrac{1}{|\eta-y|}\nabla_y\rho(y,t)dy.$$
Then we can transform \eqref{g} to the Lagrangian version: $\forall y\in\Omega(t)$, there exists $z$ such that $y=\eta(z,t)$, and we have the following formula
\begin{equation}
\begin{split}
  G^i(x,t):={F^{-1}}_i^k\partial_k \Phi= \int_{\Omega}\dfrac{[{F^{*}}_i^k\partial_k(\rho_0J^{-1})](z,t)\,dz}{|\eta(x,t)-\eta(z,t)|} .
\end{split}
  \label{forceF}
\end{equation}
Particularly, when $t=0$, we have
\begin{align*}
\phi_0(x)&:=\phi(x,0)=\int_{\Omega}\dfrac{\rho_0(z)\,dz}{|x-z|},
\\
G(x,0)&=\nabla\phi_0=\int_{\Omega}\dfrac{\nabla{\rho_0}(z)\,dz}{|x-z|} .
\end{align*}
Now, we can rewrite the system \eqref{lag}--\eqref{cffffff} as the following Euler equation with external forcing term:
\begin{align}
\label{polp}
  \rho_0 v_t^i + {F^{*}}_i^k\partial_k(\rho_0^{2}J^{-2}) = \rho_0 G^i \quad\text{in}\,\,\ \  \Omega\times (0,T],
\end{align}

\begin{remark}
If $\rho \in C^{\alpha}$, then $g \in C^{1,\alpha}$, we will see that this regularity is important for the case $\gamma \neq 2$ in Section \ref{ga3}.
\end{remark}

\section{Notation}
\label{s3}
In this section, we introduce the notations which will be used throughout the paper. We will use Einstein's summation convention and use the notation $w_{,k}=\frac{\partial w}{\partial x_k}$ to denote $k$-th partial derivative of $w$.
\subsection{Derivatives and operators}
The reference domain $\Omega$ is defined in \eqref{domain}. We use $D$ to denote the three-dimensional gradient vector
\begin{equation*}
D=(\dfrac{\partial}{\partial x_1},\dfrac{\partial}{\partial x_2},\dfrac{\partial}{\partial x_3}),
\end{equation*}
and use $\bar{\partial}$ to denote the tangential derivatives $\bar{\partial}=(\dfrac{\partial}{\partial x_1},\dfrac{\partial}{\partial x_2})$. We also denote the gradient vector in the Eulerian coordinates as $D_{\eta^i}={F^{-1}}_i^k\partial_k$, $i=1,2,3$.

We use $\Div V$ to denote the divergence of a vector field $V$ on $\Omega$ as $\Div V={V^k}_{,k}$, and use $\curl V$ to denote the curl of a vector $V$ on $\Omega$ as $[\curl V]^i=\epsilon_{ijk}{V^k}_{,j}$. Here we make use of the permutation symbol
\begin{equation}
\epsilon_{ijk}=\left\{\begin{aligned}
1 &,\,\,\text{even permutation of}\,\, \{1,2,3\},\\
-1 &,\,\,\text{odd permutation of}\,\, \{1,2,3\},\\
0 &,\,\,\text{otherwise.} \\
\end{aligned}\right.
\end{equation}

Finally, we define the two-dimensional divergence operator $\Div_{\Gamma}$ for vector fields $\mathfrak{V}$ on the two-dimensional boundary $\Gamma$ as $\Div_{\Gamma}\mathfrak{V}={\mathfrak{V}^1}_{,1}+{\mathfrak{V}^2}_{,2}$, and we define the divergence, curl operators and curl matrix in the Eulerian coordinates as follows
\begin{align*}
\Div_{\eta} U&={F^{-1}}_i^j {U^i}_{,j},\\
[\curl_{\eta} U]^i &= \epsilon_{ijk}{F^{-1}}_j^r {U^k}_{,r},\\
[\text{Curl}_{\eta} U]_j^i &= {F^{-1}}_j^s {U^i}_{,s}-{F^{-1}}_i^s {U^j}_{,s}.
\end{align*}

\subsection{Sobolev spaces}
For integers $k\geq0$, we define the Sobolev space $H^k(\Omega)$ to be the completion of the functions in $C^{\infty}(\bar{\Omega}):=C^{\infty}(\mathbb{T}^2\times[0,1])$ in the norm
\begin{equation*}
\|u\|_k:=(\sum_{|\alpha|\leq k}\int_{\Omega}|D^{\alpha}u(x)|^2\,dx)^{1/2}
\end{equation*}
for a multi-index $\alpha\in \mathbb{Z}_{+}^3$. For real numbers $s\geq0$, the Sobolev spaces $H^s(\Omega)$ are defined by interpolation. We will also make use of the following subspace of $H^1(\Omega)$:
\begin{equation*}
H^1_0=\{u\in H^1(\Omega): u=0\,\,\text{ on}\,\, \Gamma\},
\end{equation*}
where the vanishing of $u$ on $\Gamma$ is understood in the sense of trace.

\noindent
On the boundary $\Gamma$, for functions $w\in H^k(\Gamma)$, $k\geq0$, we set
\begin{equation*}
|w|_k:=(\sum_{|\beta|\leq k}\int_{\Gamma}|\bar{\partial}^{\beta}w(y)|^2\,dy)^{1/2}
\end{equation*}
for a multi-index $\beta\in\mathbb{Z}^2_+$. The real number $s\geq0$ Sobolev space $H^s(\Gamma)$ is defined by interpolation. The negative-order Sobolev spaces $H^{-s}(\Gamma)$ are defined via duality: for real $s\geq0$, $H^{-s}(\Gamma):=[H^s(\Gamma)]'$.

\section{Preliminary}\label{s4}
In this section, we will recall some basic tools of analysis, which will be used in the construction of approximated solutions and the a prior estimates.
\subsection{Hardy's inequality in high-order form}
We will make the fundamental use of Hardy's inequality in the context of higher-order derivatives for the construction of approximated solutions in Section \ref{s11}. This type inequality was introduced in \cite{DS_2010}.
\begin{lemma}[Hardy's inequality in higher-order form]
\label{hardy}
Let $s\geq1$ be a given integer, and suppose that
\begin{equation*}
u \in H^s(\Omega)\cap H_0^1(\Omega).
\end{equation*}
If $\tilde d(x)$ is defined on $\bar\Omega$, $\tilde d(x)>0$ for $x\in\Omega$, $\tilde d\in H^r(\Omega), r=\max(s-1,3),$ and $\tilde d(x)$ is the distance function to $\Gamma$ when $dist(x,\Gamma)$ is small enough, then we have $\dfrac{u}{\tilde d}\in H^{s-1}(\Omega)$ and
\begin{equation}
\|\dfrac{u}{\tilde d}\|_{s-1}\leq C\|u\|_s.
\end{equation}
\end{lemma}
The proof of this lemma can be found in \cite{DS_2010}.
\subsection{$\kappa$-independent elliptic estimates}
In order to obtain a uniform a prior estimates ($\kappa$-independent) for the degenerate parabolic approximate equations in Section \ref{ee}, we will use the following lemma, whose proof can be found in \cite[Section 6]{DS_06}.
\begin{lemma}
\label{elli}
Given a constant $\kappa>0$ and a function $g\in L^{\infty}(0,T;H^s(\Omega))$, $f \in H^1(0,T;H^s(\Omega))$ is a solution to the equation
\begin{equation*}
f+\kappa f_t =g\quad \text{in}\quad (0,T)\times \Omega.
\end{equation*}
Then, we have
\begin{equation*}
\|f\|_{L^{\infty}(0,T;H^s(\Omega))}\leq C\max\{\|f(0)\|_s,\|g\|_{L^{\infty}(0,T;H^s(\Omega))}\},
\end{equation*}
where constant $C$ is independent on $\kappa$.
\end{lemma}

\subsection{Weighted Sobolev spaces and embedding inequalities}
As we mentioned before, the equation \eqref{eq:lag} has a weight $\rho_0$ and $\rho_0$ is corresponding to the distance function in some sense due to the physical vacuum condition \eqref{vacuum2}, so it is natural for us to recall the following kind weighted Sobolev spaces as in \cite{DS_2010, MJ_2010}. Using $d$ to denote the distance function to the boundary $\Gamma$, and letting $p=1,2$, the weighted Sobolev space $H^1_{d^p}(\Omega)$ is defined by
$$H^1_{d^p}(\Omega):=\big\{f\big|[\int_{\Omega}d(x)^p(|f(x)|^2+|Df(x)|^2)\,dx]^{1/2}<\infty\big\},$$ endowed with the norm $[\int_{\Omega}d(x)^p(|f(x)|^2+|Df(x)|^2)\,dx]^{1/2}$. Then we have the following embedding:
\begin{equation*}
H^1_{d^p}(\Omega) \hookrightarrow H^{1-\frac{p}{2}}(\Omega).
\end{equation*}
Therefore, there is a constant $C>0$ depending only on $\Omega$ and $p$, such that
\begin{equation}
\label{embd}
\|f\|_{1-\frac{p}{2}}^2\leq C \int_{\Omega}d(x)^p(|f(x)|^2+|Df(x)|^2)\,dx.
\end{equation}
More details about the weighted Sobolev spaces and embedding inequalities can be found in, for instance, \cite[Section 8.8]{Kufner_1985}.
\subsection{Trace estimates and the Hodge decomposition}
Since $\Omega=\mathbb{T}^2\times(0,1)$, the outward normal vector to $\Gamma$ will be $N=(0,0,1)$ or $(0,0,-1)$. Then the normal trace theorem provides the existence of the normal trace $\omega\cdot N=\pm\omega^3$ of a velocity field $\omega\in L^2(\Omega)$ with $\Div \omega \in L^2(\Omega)$. Motivated by \cite{DS_2010}, we recall the following trace estimates: if $\bar{\partial}\omega \in L^2(\Omega)$ with $\Div \omega\in L^2(\Omega)$, then $\bar{\partial}\omega^3$ exists in $H^{-\frac 1 2}(\Gamma)$ and
\begin{equation}
\|\bar{\partial}\omega^3\|_{H^{-\frac 1 2}(\Gamma)}\leq C[\|\bar{\partial}\omega\|_{L^2(\Omega)}^2+\|\Div\omega\|^2_{L^2(\Omega)}]
\label{gga}
\end{equation}
for some constant $C$ independent of $\omega$. In addition to the normal trace theorem, we also recall the following lemma:
\begin{lemma}
Let $\bar{\partial}\omega\in L^2(\Omega)$ such that $\curl \omega \in L^2(\Omega)$. 
\begin{equation}
\label{trineq}
\|\bar{\partial}\omega^{\alpha}\|_{H^{-\frac 1 2}(\Gamma)}\leq C[\|\bar{\partial}\omega\|_{L^2(\Omega)}^2+\|\curl\omega\|^2_{L^2(\Omega)}], \quad\alpha=1,2
\end{equation}
for some constant $C$ independent of $\omega$.
\end{lemma}
Combining \eqref{gga} and \eqref{trineq}, we have
\begin{equation}
\|\bar{\partial}\omega\|_{H^{-\frac 1 2}(\Gamma)}\leq C[\|\bar{\partial}\omega\|_{L^2(\Omega)}^2+\|\Div\omega\|_{L^2(\Omega)}^2+\|\curl\omega\|^2_{L^2(\Omega)}]
\end{equation}
for some constant $C$ independent of $\omega$. The proof of inequalities \eqref{gga} and \eqref{trineq} on a general $H^r$ domain can be found in \cite{cheng_08}.

The construction of our higher-order energy function is based on the following Hodge-type elliptic estimate:
\begin{proposition}
\label{curllemma}
For the domain $\Omega=\mathbb{T}^2\times(0,1)$, $r\geq3$, $1\leq s\leq r$, if $\omega\in L^2(\Omega;\mathbb{R}^3)$ with $\curl \omega\in H^{s-1}(\Omega;\mathbb{R}^3), \Div \omega\in H^{s-1}(\Omega;\mathbb{R}^3),$ and $\omega^3|_{\Gamma} \in H^{s-\frac{1}{2}}(\Gamma)$, then there exists a constant $\bar{C} >0$ depending only on $\Omega$ such that
\begin{align}
\nonumber &\|\omega\|_s\leq \bar{C}(\|\omega\|_0+\|\curl \omega\|_{s-1}+\|\Div \omega\|_{s-1}+|\bar{\partial}\omega^3|_{s-\frac{3}{2}}),\\
&\|\omega\|_s\leq \bar{C}(\|\omega\|_0+\|\curl \omega\|_{s-1}+\|\Div \omega\|_{s-1}+\sum_{\alpha=1}^2|\bar{\partial}\omega^{\alpha}|_{s-\frac{3}{2}}).
\label{curlineq}
\end{align}
\end{proposition}
These estimates are well-known and follow from the identity $-\Delta\omega=\curl \curl \omega-D\Div \omega$. The reader can see \cite{Taylor} for more details.
\section{Some useful identities}
\label{sec5}
In this section, we recall some useful identities, which show the properties of $J$, $F^{-1}$, $F^{\ast}$ and their derivatives.
\subsection{Differentiating the $F^{-1}$ and Jacobian determinant $J$}
In this subsection, we recall some identities which can be checked directly and will be crucial for our high order energy estimate. The details can also be found in \cite{MJ_2010,L_14,LLZ_08,LSZ_12,LZ_05}.

First, we have the following identities,
\begin{align}
\label{dJ}
\partial J=\dfrac{\partial J}{\partial F^r_s}\partial F^r_s &= [J F^{-T}]^r_s\partial F^r_s= {F^{*}}_r^s\partial F^r_s,\\
\partial {F^{-1}}_i^k & = -{F^{-1}}_r^k{F^{-1}}_i^s \partial F^r_{s},
\label{dF}
\end{align}
where $\partial$ can be $D$, $\bar{\partial}$ and $\partial_t$ operators.
\noindent
As a result, we compute the derivatives of ${F^{-1}}_i^kJ^{-1}$ as following:
\begin{align*}
\partial({F^{-1}}_i^kJ^{-1})=&J^{-1}\partial({F^{-1}}_i^k)-J^{-2}{F^{-1}}_i^k\partial J\\=&-J^{-1}{F^{-1}}_r^k{F^{-1}}_i^s\partial\eta^r_{,s}-J^{-1}{F^{-1}}_i^k{F^{-1}}_r^s\partial\eta^r_{,s}
\\=&-J^{-1}{F^{-1}}_r^k{F^{-1}}_r^s\partial\eta^i_{,s}-J^{-1}{F^{-1}}_i^k{F^{-1}}_r^s\partial\eta^r_{,s}\\
&-J^{-1}{F^{-1}}_r^k[{F^{-1}}_i^s\partial\eta^r_{,s}-{F^{-1}}_r^s\partial\eta^i_{,s}]
\end{align*}
Thus, we have
\begin{align}
\nonumber\bar{\partial}({F^{-1}}_i^kJ^{-1})=&-J^{-1}{F^{-1}}_r^k[D_{\eta}\bar{\partial}\eta]_r^i-J^{-1}{F^{-1}}_i^k \Div_{\eta}\bar{\partial}\eta\\&-J^{-1}{F^{-1}}_r^k[\text{Curl}_{\eta}\bar{\partial}\eta]_i^r\label{pn}\\
\nonumber\partial_t({F^{-1}}_i^kJ^{-1})=&-J^{-1}{F^{-1}}_r^k[D_{\eta}v]_r^i-J^{-1}{F^{-1}}_i^k \Div_{\eta}v\\&-J^{-1}{F^{-1}}_r^k[\text{Curl}_{\eta}v]_i^r\label{pt}
\end{align}
We will make use of these two equalities in our high order energy estimate.

Last, we recall the Piola identity
\begin{equation}
\label{piola}
{{F^{*}}_i^k}_{,k}=0,
\end{equation}
which will play a vital role in our energy estimates.
\subsection{Geometric identities for the surface $\eta(\Gamma,t)$}
In this subsection, we recall some geometric identities for the moving surface $\Gamma(t)=\eta(\Gamma,t)\subset \mathbb{R}^3$ as in \cite{DS_2010}, and include them below for completeness. These identities will be used in the estimates of $v$ on the boundary and the estimates of the normal derivative.

For the tangent plane to $\eta(\Gamma,t)$, $(\eta_{,1},\eta_{,2})$ is a basis of this plane, and
\begin{equation*}
\tau_1:=\dfrac{\eta_{,1}}{|\eta_{,1}|},\quad \tau_2:=\dfrac{\eta_{,2}}{|\eta_{,2}|},\quad n:=\dfrac{\eta_{,1}\times\eta_{,2}}{|\eta_{,1}\times\eta_{,2}|}
\end{equation*}
are the unit tangent and normal vectors, respectively, to $\Gamma(t)$.
Let $g_{\alpha\beta}=\eta_{,\alpha}\cdot\eta_{,\beta}$ denote the induced metric on the surface $\Gamma(t)$, then we have $\text{det} g=|\eta_{,1}\times\eta_{,2}|^2$ and
\begin{equation*}
\sqrt{g}n:=\eta_{,1}\times\eta_{,2},
\end{equation*}
where we use $\sqrt{g}$ to denote $\sqrt{\text{det} g}$.

By definition of the cofactor matrix, we have
\begin{equation}
\label{ai3}
{F^{*}}_i^3=\begin{bmatrix}
\eta^2_{,1}\eta^3_{,2}-\eta^3_{,1}\eta^2_{,2}\\
\eta^3_{,1}\eta^1_{,2}-\eta^1_{,1}\eta^3_{,2}\\
\eta^1_{,1}\eta^2_{,2}-\eta^2_{,1}\eta^1_{,2}
\end{bmatrix},
\,\,\text{and}\,\, \sqrt{g}=|{F^{*}}_i^3|.
\end{equation}
It follows that
\begin{equation}
n= {F^{*}}_i^3/\sqrt{g}.
\end{equation}

\section{An asymptotically consistent degenerate parabolic $\kappa$-approximation of the compressible Euler--Poisson equations in vacuum}
\label{sec6}
In this section, motivated by \cite{DS_2010}, we will introduce a $\kappa$-approximation of the equations \eqref{eq:lag} by adding an artificial viscosity. Our artificial viscosity is different from the one in \cite{DS_2010} due to the presence of the potential force term. In this way, we can construct the approximated solutions with higher regularity which is required by our a prior estimates.

\subsection{Smoothing the initial data}
\label{sm}
For the purpose of constructing solutions, we will smooth the initial velocity
field $u_0$ and density field $\rho_0$ while preserving the conditions that $\rho_0 > 0 $ in $\Omega$, $\rho_0=0$ on $\Gamma$ and $\rho_0$ satisfies (\ref{vacuum2}) for $x\in\Omega$ very close to $\Gamma$. This way of smoothing the initial data was developed by Coutand and Shkoller \cite{DS_2010}. For completeness and a self-contained presentation, we will still include them below.

Let $\alpha(x)\in C_0^{\infty}(\mathbb{R}^3)$ be a standard mollifier such that $\text{spt}(\alpha)=\{x\big||x|\leq 1\}$. For $\epsilon>0$, we define $\alpha_{\epsilon}(x)=\frac{1}{\epsilon^3}\alpha(\frac{x}{\epsilon})$, then we  have $\alpha_{\epsilon}(x)\in C_0^{\infty}(\mathbb{R}^3)$,
$\text{spt}(\alpha_{\epsilon}) = \{x\big| |x| \leq \epsilon\}$ and $\int_{\mathbb{R}^3}\alpha_{\epsilon}(x)\,dx=1$. We also denote $E_{\Omega}$ as a Sobolev extension operator mapping $H^s(\Omega)$ to $H^s(\mathbb{T}^2\times\mathbb{R})$ for $s \geq 0$.

We set the smoothed initial velocity filed $u_0^{\kappa}$ as:
\begin{align}
\label{smu}
u_0^{\kappa} &= \alpha_{1/|\ln\kappa|} \ast E_{\Omega}(u_0),
\end{align}

In order to smooth the initial density function, we firstly introduce the boundary convolution operator $\Lambda_{\theta}$ on $\Gamma$.
Let $0\leq \chi \in C_0^{\infty}(\mathbb{R}^2)$ with $\text{spt}(\chi)=\bar{B}(0,1)$ denote a standard  mollifier on $\mathbb{R}^2$. For $\theta>0$, we define $\chi_{\theta}(y)=\frac{1}{\theta^2}\chi(
\frac{y}{\theta})$. Then we  have $\chi_{\theta}(y)\in C_0^{\infty}(\mathbb{R}^2)$,
$\text{spt}(\chi_{\theta}) = \{y\big| |y| \leq \theta\}$ and $\int_{\mathbb{R}^2}\chi_{\theta}(y)\,dy=1$. With $x_h=(x_1,x_2)$, we define the operation of convolution on the boundary as follow
\begin{equation}
\label{boundaryop}
\Lambda_{\theta}f(x_h)=\int_{\mathbb{R}^2}\chi_{\theta}(x_h-y_h)f(y_h)\,dy_h \,\, \text{for}\,\, f\in L_{loc}^1(\mathbb{R}^2).
\end{equation}
By standard properties of convolution, there exists a constant $C$ independent of $\theta$, such that for $s\geq 0$
\begin{equation*}
|\Lambda_{\theta}f|_s\leq C|f|_s,\,\,\forall f\in H^s(\Gamma).
\end{equation*}
Furthermore,
\begin{equation}
\theta|\bar\partial\Lambda_{\theta}f|_0\leq C|f|_0,\,\,\forall f\in L^2(\Gamma).
\label{molli}
\end{equation}

Now, the smoothed initial density function $\rho_0^{\kappa}$ is defined as the solution of the elliptic equation:
\begin{subequations}
\label{smoothrho}
\begin{align}
\label{smp}
\Delta^2 \rho_0^{\kappa} &= \alpha_{1/|\ln\kappa|}\ast E_{\Omega}(\Delta^2\rho_0) &\text{in}&\ \  \Omega,\\
\rho_0^{\kappa} &= 0 &\text{on}&\ \  \Gamma,\\
\dfrac{\partial\rho_0^{\kappa}}{\partial N} &=\Lambda_{1/|\ln\kappa|}\dfrac{\partial\rho_0}{\partial N}&\text{on}&\ \  \Gamma.\\
\nonumber (x_1,x_2)&\mapsto \rho_0(x_1,x_2,x_3)\,\, \text{is 1-periodic}
\end{align}
\end{subequations}
$\Lambda_{1/|\ln\kappa|}$ is the boundary convolution operator defined by \eqref{boundaryop}. So for sufficiently small $\kappa > 0$, $u_0^{\kappa},\rho_0^{\kappa} \in C^{\infty}(\bar{\Omega})$, $\rho_0^{\kappa} >0$ in $\Omega$, and vacuum condition (\ref{vacuum}) is preserved. Details can be found in \cite[Section 7.1]{DS_2010}.

Until Section \ref{s6}, for notational convenience, we will denote $u_0^{\kappa}$ by $u_0$ and $\rho_0^{\kappa}$ by $\rho_0$. In Section \ref{s6}, we will show that Theorem \ref{theorem} holds with the optimal regularity.

\subsection{The degenerate parabolic approximation to the compressible Euler--Poisson equations: the $\kappa$-problem}
In this subsection, we introduce a regularized approximation system for \eqref{polp}. In \cite{DS_2010}, the authors introduced a degenerate parabolic approximation $\kappa$-problem for the compressible Euler equations. Motivated by their results, we introduce the following degenerate parabolic approximation $\kappa$-problem for $\kappa >0$ for our system:
\begin{subequations}
\label{approeq}
\begin{align}
\label{lagappro}
  \rho_0 v_t + {F^{*}}_i^k\partial_k(\rho_0^{2}J^{-2})+\kappa\partial_t[{F^{*}}_i^k\partial_k(\rho_0^{2}J^{-2}) ]&= \rho_0 G^i+\kappa\rho_0\partial_tG^i &\text{in}&\,\,\Omega\times (0,T],\\ \label{ffffff}
(v, \eta) &= (u_0, e) &\text{in}&\,\,\Omega\times \{0\},\\
  \rho_0 &= 0 &\text{on}&\,\,\Gamma,
\end{align}
\end{subequations}
Solutions to \eqref{polp} will be found in the limit as $\kappa\rightarrow 0.$

Moreover, the equation \eqref{lagappro} can be equivalently written as
\begin{equation}
\label{lagc}
v_t^i+{F^{-1}}_i^k(2\rho_0J^{-1}-\Phi)_{,k}+\kappa\partial_t[{F^{-1}}_i^k(2\rho_0J^{-1}-\Phi)_{,k}]=0.
\end{equation}
We will use this form for the curl estimates.
\begin{remark}
In \cite{DS_2010}, the authors introduced $\kappa\partial_t[{F^{*}}_i^k\partial_k(\rho_0^{2}J^{-2}) ]$ as the artificial viscosity for the Euler equations. Motivated by this result, we introduce a similar artificial viscosity but with an extra term $\kappa\rho_0\partial_t G$. This extra viscosity term is necessary to be added. Because of presence of the potential force, we need this extra term to preserve the structure of the curl estimate for our approximate system. More precisely, \eqref{lagc} has a similar structure to the $\kappa$-approximate equation for the Euler equations in \cite{DS_2010}.
\end{remark}
\subsection{Assumption on the initial data}
Recall the fact that $\eta(x,0)=x$ and \eqref{forceF}, the quantity $v_t|_{t=0}$ for the degenerate parabolic $\kappa$-problem can be computed by using (\ref{lagc}):
\begin{equation}
  \begin{split}
u_1:=v_t\bigg|_{t=0}&=\bigg(-{F^{-1}}_i^k(2\rho_0J^{-1})_{,k}+G^i-\kappa\partial_t[{F^{-1}}_i^k(2\rho_0J^{-1})_{,k}-G^i]\bigg)\bigg|_{t=0}\\
  &=\bigg(-2{\rho_0}_{,i}-\int_{\Omega}\dfrac{{\rho_0}_{,i}}{|x-z|}\,dz +2(\kappa\rho_0\Div u_0)_{,i}+2\kappa {u_0^k}_{,i}{\rho_0}_{,k}
  \\&\quad+\kappa\int_{\Omega}\dfrac{(x-z)\cdot(u_0(x)-u_0(z)){\rho_0}_{,i}}{|x-z|^3}\,dz
  -\kappa\int_{\Omega}\dfrac{(\rho_0{u_0^k}_{,i})_{,k}}{|x-z|}\,dz\bigg).
\end{split}
  \label{initialdata}
\end{equation}
Similarly, for all $k \geq 1,\ k \in \mathbb{N}$:
\begin{equation}
\begin{split}
 u_k:&= \partial_t^k v\bigg|_{t=0}=\partial_t^{k-1}\bigg(-{F^{-1}}_i^k(2\rho_0J^{-1})_{,k}+G^i-\kappa\partial_t[{F^{-1}}_i^k(2\rho_0J^{-1})_{,k}-G^i]\bigg)\bigg|_{t=0}.
\end{split}
  \label{vt}
\end{equation}
These formulae make it clear that each $\partial_t^k v|_{t=0}$ is a function of space-derivatives of $u_0$ and $\rho_0$.

\subsection{Introduction of the $X$ variable and the heat-type equation for $X$}
Motivated by \cite{DS_2010}, we derive a heat-type equation by acting $J\Div_{\eta}={F^{*}}_i^j\partial_{x_j}$ on \eqref{lagappro} and using the Piola identity \eqref{piola}:
\begin{align}
\nonumber
J\Div_{\eta}v_t-2\kappa[{F^{*}}_i^j{F^{-1}}_i^k(\rho_0J^{-2}J_t)_{,k}]_{,j}=&-2\kappa[{F^{*}}_i^j\partial_t{F^{-1}}_i^k(\rho_0J^{-1})_{,k}]_{,j}\\
\nonumber &-2[{F^{*}}_i^j{F^{-1}}_i^k(\rho_0J^{-1})_{,k}]_{,j}+{F^{*}}_i^j{G^i}_{,j}\\&+\kappa{F^{*}}_i^j\partial_t{G^i}_{,j}
\label{tmp1}
\end{align}
Now we set
\begin{equation}
\label{xdef}
X=\rho_0J^{-2}J_t=\rho_0J^{-1}\Div_{\eta}v.
\end{equation}
Recall \eqref{field}, \eqref{g} and \eqref{lag_mass}, we have
\begin{align*}
{F^{*}}_i^j{G^i}_{,j}&=\rho_0,\\
\kappa{F^{*}}_i^j\partial_t{G^i}_{,j}&=-\kappa \rho_0\Div_{\eta}v -\kappa J\partial_t({F^{-1}}_i^j) G^i_{,j},
\end{align*}
then we can rewrite \eqref{tmp1} as the following nonlinear heat-type equation for $X$:
\begin{align}
\nonumber
\dfrac{J^{2}X_t}{\rho_0}-2\kappa[{F^{*}}_i^j{F^{-1}}_i^k X_{,k}]_{,j}+\kappa J X=&-2\kappa[{F^{*}}_i^j\partial_t{F^{-1}}_i^k(\rho_0J^{-1})_{,k}]_{,j}-2J^{-1}(J_t)^2\\
\nonumber &+\partial_t{F^{*}}_i^j{v^i}_{,j}-2[{F^{*}}_i^j{F^{-1}}_i^k(\rho_0J^{-1})_{,k}]_{,j}+\rho_0\\&-\kappa J\partial_t({F^{-1}}_i^j)G^i_{,j}.
\label{appro1}
\end{align}
\begin{remark}
Notice that our definition for $X$ is different from the one used in \cite{DS_2010}. The advantage of our definition is that we can improve the space regularity of solutions to the linearized problem of \eqref{appro1} in an easy and clear way with less computation, see Section \ref{regforx2t} to Section \ref{regforx}. However,
the symmetric structure of \eqref{appro1} becomes different from that in \cite{DS_2010}, we will need to make use of $\kappa$ to construct the contract map later.
\end{remark}

\subsection{Identities for the $\kappa$-problem}
In this subsection, we give some identities which will be used in constructing fixed-point iteration scheme in a similar way to \cite{DS_2010}. Some identities are slightly different from those in \cite{DS_2010} due to the presence of the potential force term.

From \eqref{xdef}, we can get the following identity by time-differentiating:
\begin{equation}
\label{lagdiv}
\Div_{\eta}v_t=\dfrac{(XJ)_t}{\rho_0}-\partial_t{F^{-1}}_i^j{v^i}_{,j}.
\end{equation}

We also make use of the following nonlinear vorticity equation by acting $\curl_{\eta}$  on \eqref{lagappro},
\begin{equation}
\label{c01}
(\curl_{\eta} v_t)^k = -\kappa \epsilon_{kji}v^r_{,s}{F^{-1}}_j^s[(2\rho_0J^{-1}-\Phi)_{,l}{F^{-1}}_r^l]_{,m}{F^{-1}}_i^m.
\end{equation}

For the purpose of constructing solutions to \eqref{approeq}, we also need the formula for the normal component of $v_t$ on $\Gamma$:
\begin{equation}
\label{vt3}
v_t^3=-2J^{-2}{F^{*}}_3^3{\rho_0}_{,3}-2\kappa\partial_t[J^{-2}{F^{*}}_3^3]{\rho_0}_{,3}+G^3+\kappa\partial_tG^3
\end{equation}
where
\begin{align}
{F^{*}}_3^3&=(\eta_{,1}\times\eta_{,2})\cdot e_3,\\
\partial_t{F^{*}}_3^3&=(v_{,1}\times\eta_{,2}+\eta_{,1}\times v_{,2})\cdot e_3.
\label{tmp2}
\end{align}
If we linearize \eqref{tmp2} around $\eta=e$, \eqref{tmp2} becomes $\Div_{\Gamma}v$. Thus, we can regard $\Div_{\Gamma}v$ as the linearized analogue of $\partial_t{F^{*}}_3^3$.

\section{Solving the parabolic $\kappa$-problem by a fixed-point method}\label{s5}
In this section, we use fixed-point method to find approximated solutions to the parabolic $\kappa$-problem. The main arguments are quite similar to that in \cite{DS_2010}. But there are some differences in this paper. First, we need to estimate the potential force $G$ and we will encounter the main obstacle that the possible increasing singularity of the kernel when we estimate the derivatives of $G$. Second, since our intermediate variable $X$ defined by \eqref{xdef} is different from the one in \cite{DS_2010}, then the improvement of the space regularity for $X$ will be easier and clearer with less computations. Furthermore, due to this choice of the intermediate variable, the symmetric structure of equations for $X$ is also different from that in \cite{DS_2010}, we will need to make use of $\kappa$ to construct the contract map.
\subsection{Functional framework for the fixed-point scheme}
The functional framework for the fixed-point scheme is just the same as the one in \cite{DS_2010}, we recall them below for completeness.

For $T>0$, we denote the following Hilbert spaces by $X_T$, $Y_T$ and $Z_T$:
\begin{align*}
&X_T=\bigg\{v\in L^2(0,T;H^4(\Omega))|\partial_t^a v \in L^2(0,T;H^{4-a}(\Omega)), a=1,2,3\bigg\},\\
&Y_T=\bigg\{y\in L^2(0,T;H^3(\Omega))|\partial_t^a y \in L^2(0,T;H^{3-a}(\Omega)), a=1,2,3\bigg\},\\
&Z_T=\bigg\{v\in X_T|\rho_0Dv \in X_T\bigg\},
\end{align*}
endowed with their natural Hilbert norms:
\begin{multline}
\label{norm}
\|v\|_{X_T}^2=\sum_{a=0}^3\|\partial_t^a v\|_{L^2(0,T;H^{4-a}(\Omega))}^2, \|y\|_{Y_T}^2=\sum_{a=0}^3\|\partial_t^a y\|_{L^2(0,T;H^{3-a}(\Omega))}^2,\\ \|v\|_{Z_T}^2=\|v\|_{X_T}^2+\|\rho_0Dv\|_{X_T}^2.
\end{multline}
For $M>0$, we define the following closed, bounded, convex subset of $X_T$:
\begin{equation}
C_T(M)=\{w\in Z_T: \|w\|_{Z_T}^2 \leq M, w(0)=u_0, \partial_t^k w(0)=\partial_t^k v|_{t=0} (k=1,2)\},
\end{equation}
We also define the polynomial function $N_0$ of norms of the initial data as
\begin{equation}
\label{n0}
N_0=P(\|u_0\|_{100},\|\rho_0\|_{100}).
\end{equation} Since we have smoothed the initial data, we can use the artificially high $H^{100}(\Omega)$-norm and later, in Section \ref{s6}, we carry out the optimal regularity for the initial data.

Then we assume that $T>0$ is given independent of the choice of $v\in C_T(M)$, such that $\eta(x,t)=x+\int_0^tv(x,s)\,ds$ is injective for $t\in[0,T]$, and
\begin{equation}
\label{Jnorm}
\frac{7}{8}\leq J(x,t)\leq \frac{9}{8}
\end{equation} for $t\in[0,T]$ and $x\in\bar{\Omega}$.
This can be achieved by taking $T>0$ sufficiently small because we have that
\begin{equation*}
\|J(\cdot,t)-1\|_{L^{\infty}(\Omega)}\leq C \|J(\cdot,t)-1\|_{2}=\|\int_0^t{F^{*}}_r^sv^r_{,s}\|_2\leq C\sqrt{T}M.
\end{equation*}
In the same fashion, we can take $T>0$ small enough to ensure that on $[0,T]$,
\begin{equation}
\label{eta0}
\dfrac{7}{8}|x-y|\leq|\eta(x,t)-\eta(y,t)|\leq \dfrac{9}{8}|x-y|,
\end{equation}
and 
\begin{equation}
\label{positive}
\frac{7}{8}|\xi|^2\leq {F^{*}}_i^j(x,t){F^{-1}}_i^k(x,t)\xi_j\xi_k, \,\, \forall \xi\in\mathbb{R}^3, x\in \Omega.
\end{equation}
The space $Z_T$ will be appropriate for our fixed-point methodology to prove existence of a solution to our degenerate parabolic $\kappa$-problem \eqref{approeq}.

\begin{theorem}[Solutions to the $\kappa$-problem]
\label{th2}
Given smooth initial data with $\rho_0$ satisfying $\rho_0(x)>0$ for $x\in\Omega$ and verifying the physical vacuum condition \eqref{vacuum2} near $\Gamma$, then there exists a positive constant $\kappa_0$ only depending on $\|\rho_0\|_4$ and the domain $\Omega$. When $\kappa\leq \kappa_0$, and $T_{\kappa}>0$ sufficiently small, there exists a unique solution $v\in Z_{T_{\kappa}}$ to the degenerate parabolic $\kappa$-problem \eqref{approeq}.
\end{theorem}

The remainder of this section will be devoted to the proof of Theorem \ref{th2}.
\subsection{Fixed-point scheme for the $\kappa$-problem \eqref{approeq}}
Motivated by \cite{DS_2010}, we build the fixed-point scheme by two steps. In the first step, with $\bar{v}\in C_T(M)$, we find a solution $X$ to the linearizion (with respect to $\bar v$) of \eqref{appro1}. Next, in the second step, with $\bar v$ and $X$,  we define $v$ by a linear elliptic system of equations for $v$, which can be viewed as the linear analogue (with respect to $\bar{v}$, and $X$) of equations \eqref{lagdiv}, \eqref{c01} and \eqref{vt3}.

\noindent\textbf{Step 1.}
Given $\bar{v}\in C_T(M)$, we define $\bar{\eta}(t)=e+\int_0^t\bar{v}(t')\,dt'$, and set
$$
\bar{F}^{-1}=[D\bar{\eta}]^{-1}, \bar{J}=\text{det}D\bar{\eta},\,\,  \bar{F}^{\ast}=\bar{J}\bar{F}^{-1},$$
$$
\bar{B}^{jk}=\bar{F}^{*j}_{\,\,\,i}{\bar F}^{-1k}_{\quad i},\,\, \bar{G}=\int_{\Omega}\dfrac{[\bar{F}^{*k}_{\,\,\,i}\partial_k(\rho_0\bar{J}^{-1})]}{|\bar{\eta}(x,t)-\bar{\eta}(z,t)|}\,dz.
$$
Moreover, we define $\bar{\Phi}(x,t)$ as following:
\begin{equation*}
\bar{\Phi}(x,t)=\int_{\mathbb{R}^3}\dfrac{1}{|y|}\bar\rho(\bar{\eta}(x,t)-y,t)\,dy,\,\, x\in\Omega,
\end{equation*}
where
\begin{equation*}
\bar{\rho}(y,t)=\left\{\begin{aligned}\rho_0(x)\bar{J}^{-1}(x,t), \,\, &\text{when}\,\, y=\bar\eta(x,t)\\0,\,\,&\text{otherwise}
\end{aligned}\right., \,\,y\in\mathbb{R}^3
\end{equation*}
Since $\bar v\in C_T(M)$, then $\bar{\eta}$ is injective on $t\in [0,T]$. Thus, $\bar{\rho}$, and $\bar{\Phi}$ are well-defined, and
\begin{equation}
\label{phiG}
D_{\bar{\eta}}\bar{\Phi}=\bar{G}.
\end{equation}
\noindent
Now we define $X$ as the solution of the following linear and degenerate parabolic problem which is the linearization of \eqref{appro1}:
\begin{subequations}
\label{lxversion}
\begin{align}
\label{lx}
\dfrac{\bar{J}^2X_t}{\rho_0}-2\kappa[\bar{B}^{jk}X_{,k}]_{,j}+\kappa\bar JX&=\bar{W}\quad\quad\quad\quad\quad\quad\,\,\,\, \text{in}\,\, \Omega\times(0,T_{\kappa}],\\
X&=0\quad\quad\quad\quad\quad\quad\quad\,\, \text{on}\,\, \Gamma\times(0,T_{\kappa}],\\
(x_1,x_2)&\mapsto X(x_1,x_2,x_3,t) \quad\,\,\text{is 1-periodic},\,\,\\
X&=X_0:=\rho_0\Div u_0\quad\,\,\text{on}\,\,\Omega\times\{0\},
\end{align}
\end{subequations}
where the forcing function $\bar{W}$ is defined as
\begin{multline}
\label{zio}
\bar{W}=-2\kappa[\bar{F}^{*j}_{\,\,\,i}\partial_t\bar{F}^{-1k}_{\quad i}(\rho_0\bar{J}^{-1})_{,k}]_{,j}-2\bar{J}^{-1}(\bar{J}_t)^2
+\partial_t\bar{F}^{*j}_{\,\,\,i}{\bar{v}^i}_{,j}\\-2[\bar{F}^{*j}_{\,\,\,i}\bar{F}^{-1k}_{\quad i}(\rho_0\bar{J}^{-1})_{,k}]_{,j}+\rho_0-\kappa \bar J\partial_t(\bar{F}^{-1j}_{\quad i})\bar G^i_{,j}.
\end{multline}
As a result of \textbf{Step 1}, we will establish the following proposition
\begin{proposition}
\label{p8}
For $0<\mu\ll1$, there exists a positive constant $\kappa_0$ depending on the domain $\Omega$, initial data $N_0$ and $\mu$. When $\kappa\leq\kappa_0$, then for $T>0$ taken sufficiently small, there exists a unique solution to \eqref{lxversion} satisfying
\begin{equation*}
\|X\|_{X_T}^2\leq N_0+TP(\|\bar{v}\|^2_{Z_T})+\mu\|\bar v\|^2_{Z_T}.
\end{equation*}
The norms $X_T, Y_T, $ and $Z_T$ are defined in \eqref{norm}, and $P$ denotes a generic polynomial function of its arguments.
\end{proposition}
The proof of this proposition will be given in Sections \ref{s11}.

\noindent\textbf{Step 2.} In this step, we define a linear elliptic system of equations for $v$ as we mentioned before. The main idea of the construction of the linear elliptic system of equations was developed by \cite{DS_2010} for the Euler equations. Given $\bar{v}\in C_T(M)$, $\bar v$ can be regarded as an approximated solution to \eqref{approeq}, and $X$ obtained by solving the linear problem \eqref{lxversion} can be regarded as an approximation to the weighted  divergence in the Eulerian coordinates of the solution, then another approximated solution $v$ can be constructed by adding a proper small perturbation. On the other hand, the identities \eqref{lagdiv}, \eqref{c01} and \eqref{vt3} will not hold with $\bar v$ and $X$, but they can be regarded as small perturbations from zero.
Then by this observation,  we will show in Section \ref{s89} that we can define $v(t)$ on $[0,T_{\kappa}]$, in a similar way to that in \cite{DS_2010}, by specifying the divergence and curl of its time derivative in $\Omega$, as well as the trace of its normal component on the boundary $\Gamma$ in the following way:
\begin{subequations}
\label{eq:vdef}
\begin{align}
\label{8.7}
v(0)&=u_0 \quad\quad\quad\quad\quad\quad\quad\quad\quad\quad\quad\quad\quad\quad\quad\quad\quad\,\,\text{in}\,\, \Omega&\\\label{diveq}
\Div v_t&=\Div \bar{v}_t-\Div_{\bar{\eta}}\bar{v}_t+\dfrac{[\bar{X}\bar{J}]_t}{\rho_0}-\partial_t\bar{F}^{-1j}_{\quad i}\bar{v}^i_{,j}\quad\quad\,\,\,\text{in}\,\,\Omega&
\\
\curl v_t&=\curl \bar{v}_t-\curl_{\bar\eta}\bar v_t+\kappa\epsilon_{\cdot ji}\bar v_{,s}^r\bar{F}^{-1s}_{\quad j}D_{\bar\eta^r}\bar{\mathfrak{F}}^i+\bar{\mathfrak{C}}\,\,\text{in}\,\,\Omega&\label{cppp}\\\nonumber
v_t^3+2\kappa\rho_{0,3}\Div_{\Gamma} v&= 2\kappa\rho_{0,3}\Div_{\Gamma}\bar v -2\bar{J}^{-2}\bar{F}^{*3}_{\,\,\,\,3}\rho_{0,3}-2\kappa\partial_t[\bar{J}^{-2}\bar{F}^{*3}_{\,\,\,\,3}]\rho_{0,3}\\&\quad+\bar{G}^3+\kappa\partial_t\bar G^3+\bar c(t)N^3 \,\,\quad\quad\quad\quad\quad\quad\quad\,\,\,\text{on}\,\,\Gamma&\label{xcvj}\\
\nonumber\int_{\Omega}v_t^{\alpha}\,dx&=-2\int_{\Omega}\bar{F}^{-1k}_{\quad\alpha}(\dfrac{\rho_0}{\bar J})_{,k}\, dx -2\kappa\int_{\Omega}\partial_t[\bar{F}^{-1k}_{\quad\alpha}(\dfrac{\rho_0}{\bar J})_{,k}]\,dx\\&\quad+\int_{\Omega}\bar G^{\alpha}+\kappa\partial_t\bar G^{\alpha}\,dx.&\label{jiaf}
\end{align}
\end{subequations}
The presence of $\Div_{\Gamma}\bar v$ in \eqref{xcvj} presents the linearization of $\partial_t\bar{F}^{*3}_{\,\,\,\,3}$ around $\bar\eta= e$, where
\begin{align}
\nonumber \bar{F^{*}}_3^3&=(\bar\eta_{,1}\times\bar\eta_{,2})\cdot e_3,\\
\partial_t\bar{F^{*}}_3^3&=(\bar v_{,1}\times\bar\eta_{,2}+\bar\eta_{,1}\times \bar v_{,2})\cdot e_3.
\label{tmp3}
\end{align}
The function $\bar c(t)$ on the right-hand side of \eqref{xcvj} is defined by
\begin{align}
\nonumber \bar c(t)=& \dfrac{1}{2}\int_{\Omega}(\Div \bar v_t-\Div_{\bar\eta}\bar v_t)\,dx+\dfrac{1}{2}\int_{\Omega}\dfrac{[\bar X\bar J]_t}{\rho_0}\,dx-\dfrac{1}{2}\int_{\Omega}\partial_t\bar{F}^{-1j}_{\quad i}\bar v^i_{,j}\,dx\\\nonumber
&+\int_{\Gamma}\bar{J}^{-2}\bar{F}^{*3}_{\,\,\,\,3}\rho_{0,3} N^3 \,dS+\kappa\int_{\Gamma}\partial_t[\bar{J}^{-2}\bar{F}^{*3}_{\,\,\,\,3}]\rho_{0,3} N^3\,dS\\
&+\kappa\int_{\Gamma}\Div_{\Gamma}(v-\bar v)\rho_{0,3} N^3\,dS+\int_{\Gamma}(\bar G^3+\kappa\partial_t\bar G^3) N^3\,dS,
\end{align}
and the vector field $\bar{\mathfrak{F}}$ on the right-hand side of \eqref{cppp} is defined on $[0,T]\times \Omega$ as the solution of the ODE
\begin{align}
\label{Eeq}
\bar{\mathfrak{F}}+\kappa\bar{\mathfrak{F}}_t &= -\bar v_t,\\
\bar{\mathfrak{F}}(0)&= 2D\rho_0-D\phi_0=2D\rho_0+\int_{\Omega}\dfrac{D\rho_0}{|x-z|}\,dz.
\end{align}
The vector field $\bar{\mathfrak{C}}$ on the right-hand side of \eqref{cppp} is then defined on $[0,T]\times\Omega$ by
\begin{equation}
\label{Cdef}
\bar{\mathfrak{C}}^i=\bar{F}^{-1j}_{\quad i}\psi_{,j}+\kappa\partial_t[\bar{F}^{-1j}_{\quad i}\psi_{,j}],
\end{equation}
where $\psi$ is solution of the following time-dependent elliptic-type problem for $t\in[0,T]$:
\begin{subequations}
\label{eq:b}
\begin{align}
[\bar{F}^{-1j}_{\quad i}\psi_{,j}]_{,i}+\kappa\partial_t[\bar{F}^{-1j}_{\quad i}\psi_{,j}]_{,i}&=\Div(\curl_{\bar\eta}\bar v_t-\kappa\epsilon_{\cdot ji}\bar v^r_{,s}\bar{F}^{-1s}_{\quad j} D_{\bar{\eta}^r}\bar{\mathfrak{F}}^i)\,\,\text{in}\,\,\Omega,\\
\psi &=0 \quad\quad\quad\quad\quad\quad\quad\quad\quad\quad\quad\quad\quad\quad\,\text{on}\,\,\Gamma,\label{bn}\\
\psi|_{t=0}&=0 \quad\quad\quad\quad\quad\quad\quad\quad\quad\quad\quad\quad\quad\quad\,\,\text{in}\,\, \Omega,
\end{align}
\end{subequations}
so that the compatibility condition for \eqref{cppp}
\begin{equation}
\Div(-\curl_{\bar\eta}\bar v_t+\kappa\epsilon_{\cdot ji}\bar v^r_{,s}\bar{F}^{-1s}_{\quad j} D_{\bar\eta^r}\bar{\mathfrak{F}}^i+\bar{\mathfrak{C}})=0
\end{equation}
holds.
By integrating factor method, we have a closed form solution to the ODE \eqref{Eeq}, and then by integrating by parts in time, we can find that
\begin{align}
\nonumber \bar{\mathfrak{F}}(t,x)&=e^{-\frac{t}{\kappa}}(2D\rho_0(x)+\int_{\Omega}\dfrac{D\rho_0(z)}{|x - z|}\,dz)-\int_0^t\dfrac{e^{\frac{t'-t}{\kappa}}}{\kappa}\bar v_t(t',x)\,dt'\\\nonumber
&=e^{-\frac{t}{\kappa}}(2D\rho_0(x)+\int_{\Omega}\dfrac{D\rho_0(z)}{|x - z|}\,dz)+\int_0^t\dfrac{e^{\frac{t'-t}{\kappa}}}{\kappa^2}\bar v(t',x)\,dt'\\&\quad-\dfrac{1}{\kappa}\bar v(t,x)+\dfrac{e^{-\frac{t}{\kappa}}}{\kappa}u_0(x).
\label{f0}
\end{align}
Thus, the formula \eqref{f0} shows that $\bar{\mathfrak{F}}$ has the same regularity as $\bar v$. This gain in regularity is crucial and should be viewed as one of the key reasons that allow us to construct solutions to \eqref{approeq} using linearization \eqref{lxversion} with a fixed-point argument.

Similarly, we have that
\begin{align}
\nonumber [\bar{F}^{-1j}_{\quad i}\psi_{,j}]_{,i}&=\int_0^t \dfrac{e^{\frac{t'-t}{\kappa}}}{\kappa}
\Div(\curl_{\bar\eta}\bar v_t-\kappa\epsilon_{\cdot ji}\bar v^r_{,s}\bar{F}^{-1s}_{\quad j} D_{\bar{\eta}^r}\bar{\mathfrak{F}}^i)(t',x)\,dt'\\
\nonumber &=\int_0^t\dfrac{e^{\frac{t'-t}{\kappa}}}{\kappa}\epsilon_{kji}[\bar v^i_{,r}\bar{F}^{-1r}_{\quad s}\bar v_{,l}^s \bar{F}^{-1j}_{\quad l}-\dfrac{1}{\kappa}\bar v^i_{,r}\bar{F}^{-1j}_{\quad r}]_{,k}(t',x)\,dt'\\
\nonumber &\quad-\int_0^t\dfrac{e^{\frac{t'-t}{\kappa}}}{\kappa}\epsilon_{kji}[\kappa \bar v^r_{,s}\bar{F}^{-1j}_{\quad s}D_{\bar{\eta}^r}\bar{\mathfrak{F}}^i]_{,k}(t',x)\,dt'\\
&\quad +\dfrac{e^{\frac{-t}{\kappa}}}{\kappa}\Div\curl_{\bar\eta}\bar v(t,x).
\label{lo}\end{align}
Noticing that the left-hand side of \eqref{lo} can be rewritten as $\Delta\psi+[(\bar{F}^{-1j}_{\quad i}-\delta_i^j)\psi_{,j}]_{,i}$, then for $\bar v\in C_T(M)$, the elliptic problem \eqref{lo} is well-defined and together with the boundary condition \eqref{bn} provides the following estimates for any $t \in [0,T]$:
\begin{align*}
\|\psi\|_4&\leq N_0+C(T\|\bar v\|_4+\int_0^t\|\bar v\|_4)\\
\|\psi_t\|_3&\leq N_0+C(T\|\bar v_t\|_3+\|\bar v\|_3+\int_0^t\|\bar v\|_3)
\end{align*}
Recall that \eqref{Cdef}, we have
\begin{align}
\label{ccp}
\|\bar{\mathfrak{C}}\|_2&\leq N_0+C(T\|\bar v_t\|_3+\|\bar v\|_3+\int_0^t\|\bar v\|_3) \\
\|\int_0^t\bar{\mathfrak{C}}\|_3&\leq N_0+C(T\|\bar v\|_4+\int_0^t\|\bar v\|_4)
\end{align}
\begin{remark}
Condition \eqref{jiaf} is necessary only because of the periodicity of our
domain in the directions $e_1$ and $e_2$.
\end{remark}
\subsection{Construction of solutions and regularity theory for $\bar X$ and its time derivatives}
\label{s11}
This subsection will be devoted to the proof of Proposition \ref{p8}. Motivated by \cite{DS_2010}, we will proceed with a two stage process. First, we smooth $\bar{v}$ and use the Galerkin scheme to obtain strong solutions to the linear equation \eqref{lxversion} in the case that the forcing function $\bar W$ and the coefficient matrix $\bar{B}^{jk}$ are $C^{\infty}(\bar{\Omega})$-functions. Second, we use interpolation estimates and the Sobolev embedding theorem to carry out a prior estimate independent on the smoothing parameter and then conclude the proof of Proposition \ref{p8}.
\subsubsection{Smoothing the $\bar v$}
The smoothness of $v$ is quite standard, using the notation of Section \ref{sm}, for each $t\in[0,T_{\kappa}]$ and for $\nu>0$, we define
\begin{equation*}
\bar v^{\nu}(\cdot,t)=\alpha_{\nu}\ast E_{\Omega}(\bar v(\cdot,t)),
\end{equation*}
so that for each $\nu>0$, $\bar{v}^{\nu}(\cdot,t)\in C^{\infty}(\bar{\Omega})$. We define $\bar{W}^{\nu}$ by replacing $\bar{F}^{-1}, \bar{F}^*, \bar J, \bar v$ and $\bar G$ in \eqref{zio} with $\bar{F}^{-1\nu}, \bar{F}^{*\nu}, \bar J^{\nu}, \bar v^{\nu}$ and $\bar G^{\nu}$, respectively. The quantities $\bar{F}^{-1\nu}, \bar{F}^{*\nu}, \bar J^{\nu}, \bar G^{\nu}$ are defined from the map $\bar\eta^{\nu}=x+\int_0^t\bar v^{\nu}$. We also define $[\bar B^{\nu}]^{jk}=(\bar{F}^{*\nu})_i^j(\bar{F}^{-1\nu})_i^k$. According to \eqref{positive}, we can choose $\nu>0$ sufficiently small, so that for $t\in[0,T_{\kappa}]$,
\begin{equation}
\frac{7}{8}|\xi|^2\leq [\bar B^{\nu}]^{jk}(x,t)\xi_j\xi_k,\,\,\forall \xi \in \mathbb{R}^3,\,\, x\in\Omega.
\end{equation}

Until Section \ref{s9}, we will use $\bar B^{\nu}$ and $\bar W^{\nu}$ as the coefficient matrix and forcing function, respectively, but for notational convenience we will not explicitly write the superscript $\nu$.
\subsubsection{Regularity for $\bar G$}
\label{sg1}Before we study the existence and regularity for $X$, we first check the regularity for $\bar G$. In fact, we claim that
\begin{equation}
\label{est:G}
\int_0^t\|D^3\bar G\|_0^2+\|\partial_t^3D\bar{G}\|_0^2+\|\partial_t^2D^2\bar{G}\|_0^2+\|\partial_tD^3\bar G\|_0^2\leq P(\|\bar v\|_{Z_T}^2).
\end{equation}
Recall the definition of $\bar{\Phi}$ and \eqref{phiG} , we have
\begin{align}
\nonumber
D^3 \bar G^i&=D^3D_{\bar{\eta}^i}\bar{\Phi}\\\nonumber&=\int_{\mathbb{R}^3}\dfrac{1}{|y|}D_x^3D_{\bar\eta^i}\bar\rho(\bar\eta(x,t)-y,t)\,dy
\\\nonumber&=\int_{\mathbb{R}^3}\dfrac{1}{|y|} D_x^2(\bar F^{T}(x)D_{\bar\eta})D_{\bar\eta^i}\bar\rho(\bar\eta(x,t)-y,t))\,dy\\
\nonumber&=\int_{\mathbb{R}^3}\dfrac{1}{|y|} D_x^2\bar F^{T}(x) D_{\bar\eta}D_{\bar\eta^i}\bar\rho(\bar\eta(x,t)-y,t))\,dy\\\nonumber&\quad+\int_{\mathbb{R}^3}\dfrac{1}{|y|} D_x(\bar F^{T}(x))\bar F^{T}(x)(D_{\bar\eta})^2D_{\bar\eta^i}\bar\rho(\bar\eta(x,t)-y,t))\,dy\\&\quad+\int_{\mathbb{R}^3}\dfrac{1}{|y|} \bar F^{T}(x)\bar F^{T}(x)\bar F^{T}(x)(D_{\bar\eta})^3D_{\bar\eta^i}\bar\rho(\bar\eta(x,t)-y,t))\,dy.
\label{ineq:est00}
\end{align}
Due to the definition of $\bar\rho$, when $y \in \mathbb{R}^3\backslash\bar\eta(\bar\Omega,t)$, $D_{y}^{a}\bar\rho(y,t)=0$ for $0\leq a\leq 4$, then we change variables to get that
\begin{align}
\nonumber
\int_{\mathbb{R}^3}\dfrac{1}{|y|}(D_{\bar\eta})^a\bar\rho(\bar\eta(x,t)-y,t)\,dy
=&\int_{\mathbb{R}^3}\dfrac{1}{|\bar\eta(x,t)-y|}(D_{y})^a\bar\rho(y,t)\,dy\\
=&\int_{\Omega}\dfrac{1}{|\bar\eta(x,t)-\bar\eta(z,t)|}\bar JD_{\bar{\eta}}^a(\rho_0\bar J^{-1})(z,t)\,dz.
\label{ineq:est0}
\end{align}
Thus by \eqref{eta0} and Young's inequality for convolution, we can have that
\begin{align}
\nonumber \|\int_{\mathbb{R}^3}\dfrac{1}{|y|}D_{\bar\eta}^a\bar\rho(\bar\eta(x,t)-y,t)\,dy\|_0^2&\leq  C\|\dfrac{1}{|x|}\|_{L^1(\Omega)}^2\|\bar JD_{\bar{\eta}}^a(\rho_0\bar{J}^{-1})\|_0^2
\\&\leq P(\|\bar\eta\|_4,\|\rho_0D\bar{\eta}\|_4),
\label{ineq:est1}
\end{align}
and
\begin{align*}
\int_0^t\|D^3\bar G\|_0^2\leq P(\|\bar v\|_{Z_T}^2).
\end{align*}

Now we estimate the $L^2(0,T;L^2(\Omega))$-norm of $\partial_t^3D\bar G$. Similarly, as we showed from \eqref{ineq:est00} to \eqref{ineq:est0}, we can get that
\begin{align*}
\partial_t^3D\bar G=&\partial_t^3(\bar{F}^{T}(x,t) \int_{\Omega}\dfrac{1}{|\bar\eta(x,t)-\bar\eta(z,t)|}\bar J(D_{\bar{\eta}})^2(\rho_0\bar J^{-1})(z,t)\,dz)\\=&
 \bar{F}^{T}(x,t)\sum_{p=0}^3C_p\int_{\Omega}\partial_t^p[\bar J(D_{\bar{\eta}})^2(\rho_0\bar J^{-1})](z,t)\partial_t^{3-p}\dfrac{1}{|\bar\eta(x,t)-\bar\eta(z,t)|}\,dz\\&+
  \sum_{b=1}^3\partial_t^b\bar{F}^{T}(x,t)\partial_t^{3-b}\int_{\Omega}[\bar J(D_{\bar{\eta}})^2(\rho_0\bar J^{-1})](z,t)\dfrac{1}{|\bar\eta(x,t)-\bar\eta(z,t)|}\,dz
\end{align*}
We denote $K(\bar\eta,\bar v)=(\bar\eta(x,t)-\bar\eta(z,t))\cdot ((\bar v(x,t)-\bar v(z,t))$. When $p=0$, we have
\begin{align*}
\partial_t^3\dfrac{1}{|\bar\eta(x,t)-\bar\eta(z,t)|}=&C_1\dfrac{K^3(\bar\eta,\bar v)}{|\bar\eta(x,t)-\bar\eta(z,t)|^{7}}+C_2\dfrac{\partial_tK(\bar{\eta},\bar v)K(\bar\eta,\bar v)}{|\bar\eta(x,t)-\bar\eta(z,t)|^5}\\&+C_3\dfrac{\partial_t^2K(\bar{\eta},\bar v)}{|\bar\eta(x,t)-\bar\eta(z,t)|^3},
\end{align*}
and then
\begin{align}
\nonumber
&\int_{\Omega}|\int_{\Omega}[\bar{J}(D_{\bar{\eta}})^2(\rho_0\bar{J}^{-1})]\partial_t^{3}\dfrac{1}{|\bar\eta(x,t)-\bar\eta(z,t)|}\,dz|^2\,dx\\\nonumber\leq& C\int_{\Omega}|\int_{\Omega}[\bar{J}(D_{\bar{\eta}})^2(\rho_0\bar{J}^{-1})]\dfrac{\partial_t^2K(\bar{\eta},\bar v)}{|\bar\eta(x,t)-\bar\eta(z,t)|^3}\,dz|^2\,dx+R
\\\nonumber
\leq &C\int_{\Omega}|\int_{\Omega}[\bar{J}(D_{\bar{\eta}})^2(\rho_0\bar{J}^{-1})]\dfrac{\partial_t^2(\bar{\eta}(x,t)-\bar{\eta}(z,t))\cdot(\bar v(x,t)-\bar v(z,t))}{|\bar\eta(x,t)-\bar\eta(z,t)|^3}\,dz|^2\,dx\\&+C\int_{\Omega}|\int_{\Omega}[\bar{J}(D_{\bar{\eta}})^2(\rho_0\bar{J}^{-1})]\dfrac{(\bar{\eta}(x,t)-\bar{\eta}(z,t))\cdot\partial_t^2(\bar v(x,t)-\bar v(z,t))}{|\bar\eta(x,t)-\bar\eta(z,t)|^3}\,dz|^2\,dx
+R.
\label{ineq:est2}
\end{align}
By Taylor's formula and \eqref{eta0}, the first term on the right-hand side of \eqref{ineq:est2} can be bounded by the following integral
\begin{multline*}
C\int_{\Omega}|\int_{\Omega}[\bar{J}(D_{\bar{\eta}})^2(\rho_0\bar{J}^{-1})]\dfrac{|D\bar v(\tilde x(x,z),t)||\partial_t(\bar v(x,t)-\bar v(z,t))|}{|x-z|^2}\,dz|^2\,dx\\\leq C\|\bar{J}(D_{\bar{\eta}})^2(\rho_0\bar{J}^{-1})\|_{L^{\infty}(\Omega)}^2\|D\bar v\|_{L^{\infty}(\Omega)}^2\|\partial_t\bar v\|_0^2\|\dfrac{1}{|x|^2}\|_{L^1(\Omega)}^2,
\end{multline*}
where we also used Young's inequality for convolution. Similarly, the other terms can bounded by $C\|\bar{J}(D_{\bar{\eta}})^2(\rho_0\bar{J}^{-1})\|_{L^{\infty}(\Omega)}^2\|\partial_t^2\bar v\|_0^2\|\dfrac{1}{|x|^2}\|_{L^1(\Omega)}^2$.
Thus, by repeating this process for $p=1,2,3$, we can get that
\begin{align*}
\int_0^t\|\partial_t^3D\bar G\|_0^2\leq P(\|\bar v\|_{Z_T}^2).
\end{align*}
We can estimate the $L^2(0,T;L^2(\Omega))$-norm of $\partial_t^2D^2\bar G$ and $\partial_tD^3\bar G$ in a similar fashion, and get that
\begin{equation*}
\int_0^t \|\partial_tD^3\bar G\|_0^2+\|\partial_t^2D^2\bar G\|_0^2\leq P(\|\bar v\|_{Z_T}^2).
\end{equation*}
\begin{remark}
This estimate is independent on $\nu$.
\end{remark}
\subsubsection{$L^2(0,T;{H}^1_0(\Omega))$ regularity for ${X}_{ttt}$}

\begin{definition}[Weak solution of \eqref{lxversion}]
$ X \in L^2(0,T;{H}^1_0(\Omega))$ with $\dfrac{ X_t}{\rho_0} \in L^2(0,T;H^{-1}(\Omega))$ is a weak solution of \eqref{lxversion} if

\noindent
(i) $\bar W\in L^2(0,T;H^{-1}(\Omega))$, for all $V\in{H}_0^1(\Omega)$,
\begin{equation}
\label{weakform}
\langle\dfrac{\bar J^2 X_t}{\rho_0}, V\rangle+2\kappa\int_{\Omega}\bar B^{jk} X_{,k}V_{,j}\,dx =\langle\bar W, V\rangle \,\, \text{a.e.}\,\,[0,T],
\end{equation}
(ii) $X=X_0$.

\noindent $\langle\cdot,\cdot\rangle$ denotes the duality pairing between $ H_0^1(\Omega)$ and $H^{-1}(\Omega)$.
\end{definition}
Recall that if $\bar W\in H^{-1}(\Omega)$, then $$\|\bar W\|_{H^{-1}(\Omega)} =\sup \{\langle\bar W, V\rangle| V\in  H^1_0, \|V\|_{ H_0^1(\Omega)}=1\}.$$ Furthermore, there exists functions $\bar W_0, \bar W_1, \bar W_2, \bar W_3$ in $L^2(\Omega)$ such that $\langle\bar W,V\rangle=\int_{\Omega}\bar W_0 V+\bar W_i V_{,i}\,dx$, so that $\|\bar W\|^2_{H^{-1}(\Omega)}=\inf\sum_{a=0}^3\|\bar W_a\|_0^2$, the infimum being taken over all such functions $\bar W_a$.

\begin{lemma}
\label{weakso}
If $\bar W \in L^2(0,T;H^{-1}(\Omega))$ and $\dfrac{X_0}{\sqrt{\rho_0}}\in L^2(\Omega)$, then for $T>0$ taken sufficiently small so that \eqref{positive} holds, there exists a unique weak solution to \eqref{lxversion} such that for constants $C_p>0$ and $C_{\kappa}>0$,
\begin{align*}
\|\dfrac{ X_t}{\rho_0}\|^2_{L^2(0,T;H^{-1}(\Omega))}+&\sup_{t\in[0,T]}C\|\dfrac{ X(t)}{\sqrt{\rho_0}}\|_0^2+C_p\kappa\| X\|^2_{L^2(0,T; H^1_0(\Omega))}\\
&\leq \|\dfrac{X_0}{\sqrt{\rho_0}}\|_0^2+C_{\kappa}\|\bar W\|^2_{L^2(0,T;H^{-1}(\Omega))}.
\end{align*}
\end{lemma}
\begin{proof}
Let $(e_n)_{n\in\mathbb{N}}$ denote a Hilbert basis of $ H^1_0(\Omega)$, and each $e_n$ is smooth, for instance, the eigenfunctions of the Laplace operator on $\Omega$ with vanishing Dirichlet boundary conditions on $\Gamma$. We then define the Galerkin approximation at order $n\geq 1$ as $X_n=\sum_{i=0}^n\lambda_i^n(t)e_i$ such that $\forall l \in \{0,\dots,n\}$, $X_n$ satisfies the following equation:
\begin{subequations}
\label{eq:gal}
\begin{align}
\nonumber (\bar J^2\dfrac{X_{nt}}{\rho_0},e_l)_{L^2(\Omega)}&+2\kappa(\bar B^{jk}{X_n}_{,k},{e_l}_{,j})_{L^2(\Omega)}+\kappa(\bar J {X_n}, e_l)_{L^2(\Omega)}\\\label{gal}
&=(\bar W_0,e_l)_{L^2(\Omega)}-(\bar W_i,\dfrac{\partial e_l}{\partial x_i})_{L^2(\Omega)}\,\,\text{in}\,\,[0,T],
\\
\lambda_l^n(0)&=(X_0,e_l)_{L^2(\Omega)}.
\end{align}
\end{subequations}
Since each $e_l$ is in $H^{k+1}(\Omega)\cap H^1_0(\Omega)$ for every $k\geq 1$, by Hardy's inequality \eqref{hardy}, we have that
\begin{equation*}
\dfrac{e_l}{\rho_0}\in H^{k}(\Omega) \,\, \text{for}\,\, k\geq 1,
\end{equation*}
therefore, each integral in \eqref{eq:gal} is well-defined.

Furthermore, since the $e_l$ are linearly independent, then $\dfrac{e_l}{\sqrt{\rho_0}}$ are linearly independent and therefore the determinant of the matrix
\begin{equation*}
[(\dfrac{e_i}{\sqrt{\rho_0}},\dfrac{e_j}{\sqrt{\rho_0}})_{L^2(\Omega)}]_{(i,j)\in \mathbb{N}_n}
\end{equation*}
is nonzero. This implies that our finite-dimensional Galerkin approximation \eqref{eq:gal} is a well-defined first-order differential system of order $n+1$, which therefore has a solution on a time interval $[0,T_n]$, where $T_n$ a prior depends on the rank $n$ of the Galerkin approximation. In order to prove that $T_n=T$, with $T$ independent of $n$, we notice that since $X_n$ is a linear combination of the $e_l$, we have that on $[0,T_n]$,
\begin{align*}
 (\bar J^2\dfrac{X_{nt}}{\rho_0},X_n)_{L^2(\Omega)}&+2\kappa(\bar B^{jk}{X_n}_{,k},{X_n}_{,j})_{L^2(\Omega)}+\kappa(\bar{J}X_n,X_n)_{L^2(\Omega)}\\
&=(\bar W_0,X_n)_{L^2(\Omega)}-(\bar W_i,\dfrac{\partial X_n}{\partial x_i})_{L^2(\Omega)}
\end{align*}
Then it follows that on $[0,T_n]$
\begin{align*}
\dfrac{1}{2}\dfrac{d}{dt}&\int_{\Omega}\bar J^2\dfrac{|X_n|^2}{\rho_0}\,dx+2\kappa\int_{\Omega}\bar B^{jk}X_{n,k}X_{n,j}\,dx+\kappa\int_{\Omega}\bar J|X_n|^2\,dx\\
&=\dfrac{1}{2}\int_{\Omega}(\bar J^2)_t\dfrac{|X_n|^2}{\rho_0}\,dx +\int_{\Omega}\bar W_0 X_n\,dx-\int_{\Omega}\bar W_i X_{n,i}\,dx,
\end{align*}
Using \eqref{positive}, we see that
\begin{align}
\nonumber \dfrac{1}{2}\dfrac{d}{dt}&\int_{\Omega}\bar J^2\dfrac{|X_n|^2}{\rho_0}\,dx+\frac{7}{4}\kappa\int_{\Omega}|DX_{n}|^2\,dx+\kappa\int_{\Omega}\bar{J}|\bar X_n|^2\,dx\\
&\leq\|\dfrac{1}{2}(\bar J^2)_t\|_{L^{\infty}(\Omega)}\int_{\Omega}\dfrac{1}{\rho_0}|X_n|^2\,dx+C\|\bar W\|_{H^{-1}(\Omega)}\|DX_n\|_0.
\label{8.21}
\end{align}
Then by the Sobolev embedding theorem and the Cauchy--Young inequality, we get that
\begin{align*}
\dfrac{1}{2}\dfrac{d}{dt}&\int_{\Omega}\bar J^2\dfrac{|X_n|^2}{\rho_0}\,dx+\frac{7}{8}\kappa\int_{\Omega}|DX_{n}|^2\,dx+\kappa\int_{\Omega}\bar{J}| X_n|^2\,dx\\
&\leq\|\bar J_t\|_{2}\|\bar{J}\|_{L^{\infty}}\int_{\Omega}\dfrac{1}{\rho_0}|X_n|^2\,dx+C_{\kappa}\|\bar W\|_{H^{-1}(\Omega)}^2.
\end{align*}
where the constant $C_{\kappa}$ depends inversely on $\kappa$. Since $\bar v \in C_T(M)$, then with \eqref{Jnorm}, we have that on $[0,T]$,
\begin{equation*}
\int_0^t \|\bar J_t\|_{2}\|\bar{J}\|_{L^{\infty}}\,dt\leq C_M\sqrt{t}
\end{equation*}
for a constant $C_M$ depending on $M$, so that Gronwall's inequality shows that $T_n=T$ independent of $n$, and with $\bar J\geq\frac{7}{8}$ for all $\bar v\in C_T(M)$, we see that
\begin{equation*}
\sup_{t\in[0,T]}C\|\dfrac{X_n(t)}{\sqrt{\rho_0}}\|_0^2+\frac{7}{8}\kappa\int_0^T\|DX_n(t)\|_0^2 \leq \|\dfrac{X_0}{\sqrt{\rho_0}}\|_0^2+C_{\kappa}\int_0^T\|\bar W(t)\|_{H^{-1}(\Omega)}^2.
\end{equation*}
By the Poincar\'e inequality, we have that
\begin{equation*}
\sup_{t\in[0,T]}C\|\dfrac{X_n(t)}{\sqrt{\rho_0}}\|_0^2+C_p\kappa\int_0^T\|X_n(t)\|_1^2 \leq \|\dfrac{X_0}{\sqrt{\rho_0}}\|_0^2+C_{\kappa}\int_0^T\|\bar W(t)\|_{H^{-1}(\Omega)}^2.
\end{equation*}
Thus, there exists a subsequence $\{{X_n}_m\}\subset\{X_n\}$ which converges weakly to some $X$ in $L^2(0,T; H^1_0(\Omega))$, which satisfies
\begin{equation*}
\sup_{t\in[0,T]}C\|\dfrac{ X(t)}{\sqrt{\rho_0}}\|_0^2+C_p\kappa\int_0^T\|X(t)\|_0^2 \leq \|\dfrac{X_0}{\sqrt{\rho_0}}\|_0^2+C_{k}\int_0^T\|\bar W(t)\|_{H^{-1}(\Omega)}^2.
\end{equation*}
Furthermore, it can be shown from the previous estimates, by using standard arguments for weak solutions of linear parabolic systems, that
\begin{equation*}
\dfrac{X_t}{\rho_0}\in L^2(0,T;H^{-1}(\Omega)),
\end{equation*}
and that $X(0)=X_0$ and that this $ X$ verifies the identity \eqref{weakform}. Uniqueness follows by letting $V=X$ in \eqref{weakform}.
\end{proof}

By \eqref{est:G}, it is easy to check that
\begin{equation}
\label{ineq:W}
\|\bar W\|_{L^2(0,T;L^2(\Omega))}^2 \leq P(\|\bar v\|_{Z_T}^2),
\end{equation}
thus it follows from Lemma \ref{weakso} and \eqref{n0} that
\begin{equation}
\|\dfrac{ X_t}{\rho_0}\|_{L^2(0,T;H^{-1}(\Omega))}^2+\sup_{t\in[0,T]}\|\dfrac{ X(t)}{\rho_0}\|_0^2+C_p\kappa\| X\|_{L^2(0,T; H^1_0(\Omega))}\leq C.
\label{8.22}
\end{equation}
In order to improve the regularity for $X$, we construct weak solutions for the time-differentiated problems of \eqref{lxversion}. It is convenient to proceed from the first to third time-differentiated problems. We start with the first time-differentiated version of \eqref{lxversion}:
\begin{subequations}
\label{lxtversion}
\begin{align}
\label{lxt}
\dfrac{\bar{J}^2X_{tt}}{\rho_0}-2\kappa[\bar{B}^{jk}( X_t)_{,k}]_{,j}+\kappa\bar J  X_t&=\bar{W}_t+\mathfrak{W}_1\,\, \quad\quad\,\,\,\text{in}\,\, \Omega\times(0,T_{\kappa}],\\
X_t&=0\,\, \quad\quad\quad\quad\quad\quad\text{on}\,\, \Gamma\times(0,T_{\kappa}],\\
X_t&=X_1\,\,\quad\quad\,\quad\quad\quad\text{on}\,\,\Omega\times\{0\},
\end{align}
\end{subequations}
where the initial condition $X_1$ is given as
\begin{equation}
X_1=2\kappa\rho_0\Delta  X_0-\kappa\rho_0X_0+\rho_0\bar W(0),
\end{equation}
the additional forcing term $\mathfrak{W}_1$ is defined by
\begin{equation}
\mathfrak{W}_1=2\kappa[\bar B^{jk}_t X_{,k}]_{,j}-\dfrac{(\bar J^2)_t X_t}{\rho_0}-\kappa\bar{J}_tX,
\label{8.25}
\end{equation}
and $\bar{W}(0)$ is defined by
\begin{align*}
\bar W(0)=& \rho_0-2\kappa\curl\curl u_0\cdot D\rho_0-2\kappa\Div u_0\Delta\rho_0+2\kappa {u_0^j}_{,i}{\rho_0}_{,ij}-2\Delta\rho_0\\&-(\Div u_0)^2-{u_0^i}_{,j}{u_0^j}_{,i}+\kappa {u_0}^r_{,s}\partial_r\partial_s\phi_0.
\end{align*}
According to the estimate \eqref{8.22}, $\|\mathfrak{W}_1\|^2_{L^2(0,T;H^{-1}(\Omega))}\leq C$.

\noindent From \eqref{zio},  $\|\bar{W}_t\|_{L^2(0,T;L^2(\Omega))}^2\leq P(\|\bar v\|_{X_T}^2)+\|\kappa \partial_t[\bar J\partial_t(\bar{F}^{-1j}_{\quad i})\partial_j \bar G^i]\|_{L^2(0,T;L^2(\Omega))}^2$, and
\begin{align*}
\| \partial_t[\bar J\partial_t(\bar{F}^{-1j}_{\quad i})\partial_j \bar G^i]\|_{0}^2\leq&\|\partial_t(\bar{J}\partial_t\bar {F^{-1}}_i^j)\|_{L^{\infty}(\Omega)}\|\bar{G}^i_{,j}\|_0^2\\&+\|\bar{J}\partial_t\bar {F^{-1}}_i^j\|_{L^{\infty}(\Omega)}\|\partial_t\bar G^i_{,j}\|_0^2.
\end{align*}

\noindent
Hence, by \eqref{est:G},  $\|\bar{W}_t+\mathfrak{W}_1\|^2_{L^2(0,T;H^{-1}(\Omega))}\leq C$ and by Lemma \ref{weakso} (with $ X_t, \bar W_t+\mathfrak{W}_1, X_1$ replacing $ X, \bar W, X_0$, respectively),
\begin{equation}
\|\dfrac{ X_{tt}}{\rho_0}\|_{L^2(0,T;H^{-1}(\Omega))}^2+\sup_{t\in[0,T]}\|\dfrac{ X_t(t)}{\rho_0}\|_0^2+C_p\kappa\| X_t\|_{L^2(0,T; H^1_0(\Omega))}\leq C.
\label{8.26}
\end{equation}
Next, we consider the second time-differentiated version of \eqref{lxversion}:
\begin{subequations}
\label{lx2tversion}
\begin{align}
\label{lx2t}\dfrac{\bar{J}^2{X}_{ttt}}{\rho_0}-2\kappa[\bar{B}^{jk}({X}_{tt})_{,k}]_{,j}+\kappa\bar{J}{X}_{tt}&=\bar{W}_{tt}+\mathfrak{W}_2\,\, \quad\quad\,\,\text{in}\,\, \Omega\times(0,T_{\kappa}],\\
{X}_{tt}&=0\,\, \quad\quad\quad\quad\quad\quad\text{on}\,\, \Gamma\times(0,T_{\kappa}],\\
{X}_{tt}&=X_2\,\,\quad\quad\quad\quad\quad\,\,\text{on}\,\,\Omega\times\{0\},
\end{align}
\end{subequations}
where the initial condition $X_2$ is given as
\begin{equation}
X_2=2\kappa\rho_0\Delta X_1-\kappa\rho_0 X_1+\rho_0\mathfrak{W}_1(0)+\rho_0\bar W_t(0),
\end{equation}
and the forcing function $\mathfrak{W}_2$ is defined by
\begin{equation}
\mathfrak{W}_2=\partial_t \mathfrak{W}_1+2\kappa[\bar B^{jk}_t( X_t)_{,k}]_{,j}-\dfrac{(\bar J^2)_t X_{tt}}{\rho_0}-\kappa\bar{J}_t{X}_t.
\label{8.29}
\end{equation}
The highest-order terms of $\mathfrak{W}_1(0)$ scale like $D^3u_0$ or $\rho_0D^4u_0$ or $D^3\rho_0$, so that
 $\|\sqrt{\rho_0}\mathfrak{W}_1(0)\|_0^2\leq N_0$. Using the estimate \eqref{8.26} and \eqref{est:G}, we see that
 \begin{equation}
\| \mathfrak{W}_2\|^2_{L^2(0,T;H^{-1}(\Omega))}\leq C.
\label{ineq:W_2}
 \end{equation}
Then it follows from Lemma \ref{weakso} again that
\begin{equation}
\|\dfrac{ X_{ttt}}{\rho_0}\|_{L^2(0,T;H^{-1}(\Omega))}^2+\sup_{t\in[0,T]}\|\dfrac{ X_{tt}(t)}{\rho_0}\|_0^2+C_p\| X_{tt}\|_{L^2(0,T; H^1_0(\Omega))}\leq C.
\label{8.30}
\end{equation}
Finally, we consider the third time-differentiated version of \eqref{lxversion}
\begin{subequations}
\label{eq:lx3t}
\begin{align}
\label{lx3t}
\dfrac{\bar{J}^2{X}_{tttt}}{\rho_0}-2\kappa[\bar{B}^{jk}({X}_{ttt})_{,k}]_{,j}+\kappa\bar{J}{X}_{ttt}&=\bar{W}_{ttt}+\mathfrak{W}_3\,\, \quad\,\text{in}\, \Omega\times(0,T_{\kappa}],\\
{X}_{ttt}&=0\,\, \quad\quad\quad\quad\quad\text{on}\,\, \Gamma\times(0,T_{\kappa}],\\
{X}_{ttt}&=X_3\,\,\quad\quad\quad\quad\,\text{on}\,\,\Omega\times\{0\},
\end{align}
\end{subequations}
where the initial condition $X_3$ is given as
\begin{equation}
X_3=2\kappa\rho_0\Delta X_2-\kappa\rho_0 X_2+\rho_0\mathfrak{W}_2(0)+\rho_0\bar W_{tt}(0),
\end{equation}
and the forcing function $\mathfrak{W}_3$ is defined by
\begin{equation}
\mathfrak{W}_3=\partial_t \mathfrak{W}_2+2\kappa[\bar B^{jk}_t( X_{tt})_{,k}]_{,j}-\dfrac{(\bar J^2)_t X_{ttt}}{\rho_0}-\kappa\bar{J}_t{X}_{tt}.
\end{equation}
Once again, the highest-order terms of $\mathfrak{W}_2(0)$ scale like $D^4u_0$ or $\rho_0D^5u_0$ or $D^4\rho_0$, so that $\|\sqrt{\rho_0}\mathfrak{W}_2(0)\|_0^2\leq N_0$. Using the estimate \eqref{8.30} and \eqref{est:G}, we see that $\|\bar W_{ttt}+\mathfrak{W}_3\|^2_{L^2(0,T;H^{-1}(\Omega))}\leq C$. (Note that this constant $C$ crucially depends on $\nu>0$.) Thus Lemma \ref{weakso} yields
\begin{equation}
\|\dfrac{ X_{tttt}}{\rho_0}\|_{L^2(0,T;H^{-1}(\Omega))}^2+\sup_{t\in[0,T]}\|\dfrac{ X_{ttt}(t)}{\rho_0}\|_0^2+C_p\| X_{ttt}\|_{L^2(0,T; H^1_0(\Omega))}\leq C.
\label{8.34}
\end{equation}
\subsubsection{$L^2(0,T;H^2(\Omega))$ regularity for $ X_{tt}$}
\label{regforx2t}
From \eqref{8.26} and the higher-order Hardy inequality, we have
\begin{equation*}
\|\dfrac{\bar J^2 X_t}{\rho_0}\|_{L^2(0,T;L^2(\Omega))}^2 \leq C.
\end{equation*}
Then recall \eqref{lx}, \eqref{8.22}, and \eqref{ineq:W}, by the elliptic estimates, we have
\begin{equation}
\|{X}\|_{L^2(0,T;H^2(\Omega))}^2\leq C.
\label{est:X}
\end{equation}
From \eqref{est:X}, $\|\kappa(\bar B^{jk}_t\partial_k{X})_{,j}\|_{L^2(0,T;L^2(\Omega))}^2\leq C$, and thus $\|\mathfrak{W}_1\|_{L^2(0,T;L^2(\Omega))}^2\leq C$.

\noindent Then, similarly, with \eqref{lxt}, \eqref{8.30}, and the higher-order Hardy inequality, we obtain that
\begin{equation}
\| X_t\|_{L^2(0,T;H^2(\Omega))}^2\leq C.
\label{8.58}
\end{equation}
Once again we repeat the estimates for $ X_{tt}$ which we just explained for $ X$, then we can obtain the desired result
\begin{equation}
\| X_{tt}\|_{L^2(0,T;H^2(\Omega))}^2\leq C.
\label{xtt}
\end{equation}
It follows that \eqref{lxt} holds almost everywhere.
\subsubsection{$L^2(0,T;H^3(\Omega))$ regularity for $ X_t$}
\label{imp}
First, by \eqref{est:G} and \eqref{8.58}, we have $\|\bar{W}\|_{L^2(0,T;H^1(\Omega))}\leq C$, then we can show that
\begin{equation}
\label{est:X_2}
\| X\|_{L^2(0,T;H^3(\Omega))}^2\leq C
\end{equation}
by the higher-order Hardy inequality and the elliptic estimates for \eqref{lx}.

\noindent Then it follows that in \eqref{lxt}, $\|\mathfrak{W}_1\|_{L^2(0,T;H^1(\Omega))}^2\leq C$, $\|\bar{W}_t\|_{L^2(0,T;H^1(\Omega))}^2\leq C$. Thus, by the elliptic estimates again, we have
\begin{equation}
\| X_t\|_{L^2(0,T;H^3(\Omega))}^2\leq C.
\end{equation}
\subsubsection{$L^2(0,T;H^4(\Omega))$ regularity for $ X$.}\label{regforx}
Repeating the argument in Section \ref{imp}, we can get that
\begin{equation}
\| X\|_{L^2(0,T;H^4(\Omega))}^2\leq C.
\end{equation}

Now we establish the existence and regularity of our solution $X$, however, the bounds and time interval of existence depend on $\nu>0$. Next, we use the Sobolev-type estimates to establish bounds for $X$ and time interval of existence independent on $\nu>0$. These estimates are also useful for our fixed-point scheme. The main strategy of the estimates is similar to that used in \cite{DS_2010}, but since our definition of $X$ and the corresponding symmetric structure of equations for $X$ are different, we can not get curl structure and use curl estimates as the authors did in \cite{DS_2010}. Instead, we will make use of $\kappa$ to construct the contract map.
\subsubsection{Estimate for $\|X\|^2_{X_T}$ independent of $\nu$}
\label{s9}
\textbf{Step 1}. We begin this subsection by getting $\nu$ independent energy estimates for the third time differential problem \eqref{eq:lx3t}.
\begin{lemma}
For $\delta>0$, there exists a positive constant $\kappa_0$ depends on the domain $\Omega$, initial data $N_0$ and $\delta$. When $\kappa\leq\kappa_0$, then for $T>0$ taken sufficiently small,
\begin{align}
\nonumber\|\dfrac{ X_{tttt}}{\rho_0}\|_{L^2(0,T;H^{-1}(\Omega))}^2+\sup_{t\in[0,T]}\|\dfrac{ X_{ttt}}{\sqrt{\rho_0}}\|_0^2+C_p\kappa\| X_{ttt}\|_{L^2(0,T; H^1_0(\Omega))}^2\\
\leq N_0+TP(\| X\|_{X_T}^2)+C\delta\|\bar X\|_{X_T}^2+TP(\|\bar v\|_{Z_T}^2)+\delta\|\bar v\|_{Z_T}^2
\label{lx3test}.
\end{align}
\end{lemma}
\begin{proof}
We write the forcing term $\bar{W}_{ttt}+\mathfrak{W}_3$ as
\begin{equation}
\label{8.66}
\bar{W}_{ttt}+\mathfrak{W}_3=\underbrace{\bar{W}_{ttt}+\partial_t\mathfrak{W}_2}_{I_1}+\underbrace{2\kappa[\bar B^{jk}_t( X_{tt})_{,k}]_{,j}}_{I_2}-\underbrace{\dfrac{(\bar J^2)_t X_{ttt}}{\rho_0}}_{I_3}
\end{equation}
We test \eqref{lx3t} with $ X_{ttt}$, then in a same fashion that we obtained \eqref{8.21}, we can get
\begin{align*}
\dfrac{1}{2}\dfrac{d}{dt}&\int_{\Omega}\dfrac{| X_{ttt}|^2}{\rho_0}\,dx+\frac{7}{4}\kappa\int_{\Omega}|D X_{ttt}|^2\,dx+\kappa\int_{\Omega}\bar{J}| X_{ttt}|^2\,dx\\
\leq &\|\dfrac{1}{2}(\bar J^2)_t\|_{L^{\infty}(\Omega)}\int_{\Omega}\dfrac{1}{\rho_0}| X_{ttt}|^2\,dx+\langle\bar{W}_{ttt}+\mathfrak{W}_3,  X_{ttt}\rangle.
\end{align*}
Integrating this inequality from $0$ to $t\in(0,T]$, we see that
\begin{align*}
&\dfrac{1}{2}\sup_{t\in[0,T]}\int_{\Omega}\dfrac{| X_{ttt}|^2}{\rho_0}\,dx+\frac{7}{4}\kappa\int_0^T\int_{\Omega}|D X_{ttt}|^2\,dx\,dt\\
&\leq N_0+T\sup_{t\in[0,T]}\|\dfrac{1}{2}(\bar J^2)_t\|_{L^{\infty}(\Omega)}\int_{\Omega}\dfrac{1}{\rho_0}| X_{ttt}|^2\,dx+ \int_0^T\langle\bar{W}_{ttt}+\mathfrak{W}_3,  X_{ttt}\rangle\,dt.
\end{align*}
By the Sobolev embedding theorem, $\|\frac{1}{2}(\bar J^2)_t\|_{L^{\infty}(\Omega)}\leq C\|\frac{1}{2}(\bar J^2)_t\|_2$. The highest-order derivative in the term $\frac{1}{2}(\bar J^2)_t$ scales like $D\bar v$, and $\sup\limits_{t\in[0,T]}\|\bar v\|_3\leq N_0+C\sqrt{t}\|\bar v\|_{X_T}$. Therefore, by choosing $T$ sufficiently small and using the Poincar\'e inequality, we see that
\begin{align*}
&\sup_{t\in[0,T]}\int_{\Omega}\dfrac{| X_{ttt}|^2}{\rho_0}\,dx+C_p{\kappa}\int_0^T\|X_{ttt}\|_1^2\,dt\\
&\leq N_0+ \int_0^T\langle\bar{W}_{ttt}+\mathfrak{W}_3,  X_{ttt}\rangle\,dt.
\end{align*}
We proceed to the analysis of the terms in $\int_0^T\langle\bar{W}_{ttt}+\mathfrak{W}_3,  X_{ttt}\rangle\,dt$, and we begin with the term $I_3$ in \eqref{8.66}. We have that
\begin{align*}
\int_0^T\langle I_3, X_{ttt}\rangle\,dt& \leq \sup_{t\in[0,T]}\|(\bar J^2)_t\|_{L^{\infty}(\Omega)}\int_0^T\int_{\Omega}\dfrac{| X_{ttt}|^2}{\rho_0}\,dx\,dt\\
&\leq (N_0+\sqrt{T}C_M)T\sup_{t\in[0,T]}\int_{\Omega}\dfrac{| X_{ttt}|^2}{\rho_0}\,dx,
\end{align*}
where we have made use of the Sobolev embedding theorem giving the inequality that $\|(\bar J^2)_t\|_{L^{\infty}(\Omega)}\leq C\|(\bar J^2)_t\|_2\leq N_0+C_M\sqrt{t}$, $C_M$ depends on $M$.

\noindent To estimate the term $I_2$ in \eqref{8.66}, noticing that
\begin{align*}
\langle I_2, \bar X_{ttt}\rangle=&-2\kappa\int_{\Omega}\bar B^{jk}_t( {X_{tt}})_{,k}({ X}_{ttt})_{,j}\,dx \leq C\kappa\|\bar B_t\|_2\| X_{tt}\|_1\| X_{ttt}\|_1\\
\leq&\delta \| X_{ttt}\|_1^2+C\|\bar B_t\|_2^2\| X_{tt}\|_1^2\\
\leq & \delta \| X_{ttt}\|_1^2+C\|\bar B_t\|_2^2(\| X_{tt}(0)\|_1^2+t\| X_{ttt}\|_1^2),
\end{align*}
and thus
\begin{equation*}
\int_0^T\langle I_2, X_{ttt}\rangle\,dt \leq N_0+\delta\| X\|_{X_T}^2+TP(\|\bar v\|_{Z_T}^2)+TP(\| X\|_{X_T}^2).
\end{equation*}
It remains to estimate $\langle I_1, X_{ttt}\rangle$, we use the identity \eqref{8.29} to expand $I_1$ as
\begin{equation*}
I_1=\bar{W}_{ttt}+\partial_t\mathfrak{W}_2=\bar{W}_{ttt}+\partial_{tt}\mathfrak{W}_1+\partial_t(2\kappa[\bar B^{jk}_t( X_t)_{,k}]_{,j}-\dfrac{(\bar J^2)_t X_{tt}}{\rho_0}-\kappa\bar{J}_t{X}_{t}).
\end{equation*}
The terms $\langle\partial_t(2\kappa[\bar B^{jk}_t( X_t)_{,k}]_{,j}-\dfrac{(\bar J^2)_t X_{tt}}{\rho_0}), X_{ttt}\rangle$ are estimated as
\begin{align*}
&\quad\langle\partial_t(2\kappa[\bar B^{jk}_t( X_t)_{,k}]_{,j}-\dfrac{(\bar J^2)_t X_{tt}}{\rho_0}-\kappa\bar{J}_t{X}_{t}), X_{ttt}\rangle\\&\leq C\int_{\Omega}\bar B_{tt}D^2X_tX_{ttt}\,dx+C\int_{\Omega}\bar B_{t}D^2X_{tt}X_{ttt}\,dx\\
&\quad+C\int_{\Omega}(\bar J^2)_{tt}(\dfrac{X_{tt}}{\rho_0}+X_{tt})X_{ttt}\,dx+C\int_{\Omega}(\bar J^2)_{t}\dfrac{|X_{ttt}|^2}{\rho_0}\,dx,
\end{align*}
 and then have the same bounds as $\langle I_2, X_{ttt}\rangle$ and $\langle I_3, X_{ttt}\rangle$ above, so we focus on estimating $\langle\bar{W}_{ttt}+\partial_{tt}\mathfrak{W}_1,  X_{ttt}\rangle$.  To do so, we use the identity \eqref{8.25} to get that
\begin{equation*}
\bar{W}_{ttt}+\partial_{tt}\mathfrak{W}_1=\bar{W}_{ttt}+\underbrace{\partial_{tt}(2\kappa[\bar B^{jk}_t X_{,k}]_{,j})}_{S_1}-\underbrace{\partial_{tt}(\dfrac{(\bar J^2)_t X_{t}}{\rho_0}-\kappa\bar J_tX)}_{S_2}.
\end{equation*}
Expanding $S_1$ as
\begin{equation*}
S_1=\underbrace{2\kappa[\bar B^{jk}_{ttt} X_{,k}]_{,j}}_{S_{1a}}+\underbrace{4\kappa[\bar B^{jk}_{tt}( X_t)_{,k}]_{,j}}_{S_{1b}}+\underbrace{2\kappa[\bar B^{jk}_{t}(X_{tt})_{,k}]_{,j}}_{S_{1c}},
\end{equation*}
we see that for $\delta>0$,
\begin{align*}
\langle S_{1a}, X_{ttt}\rangle=&-2\kappa\int_{\Omega}[\bar B^{jk}_{ttt}X_{,k}] X_{ttt,j}\,dx\\
\leq&C\|\bar B^{jk}_{ttt}\|_0\| X_{,k}\|_2\|D X_{ttt}\|_0\\
\leq & C\|\bar B^{jk}_{ttt}\|_0(\| X(0)\|_3^2+t\| X_t\|_3^2)+\delta\| X_{ttt}\|_1^2,
\end{align*}
where we used the Sobolev embedding theorem, the Cauchy--Young inequality and the fundamental theorem of calculus. Thus, we have
\begin{equation*}
\int_0^T\langle S_{1a}, X_{ttt}\rangle\,dt \leq N_0+\delta\|\bar X\|_{X_T}^2+TP(\|\bar v\|_{Z_T}^2)+TP(\|X\|_{X_T}^2).
\end{equation*}
The duality pairing involving $S_{1b}$ and $S_{1c}$ can be estimated in the same way to provide the estimate
\begin{equation*}
\int_0^T\langle S_1, X_{ttt}\rangle\,dt \leq N_0+\delta\|X\|_{X_T}^2+TP(\|\bar v\|_{Z_T}^2)+TP(\|X\|_{X_T}^2).
\end{equation*}
The pair involving $S_2$ is estimated in the same manner as $I_3$ and $S_1$ to yield
\begin{align*}
\int_0^T\langle S_2,X_{ttt}\rangle\,dt \leq &N_0+\delta\|X\|_{X_T}^2+TP(\|\bar v\|_{Z_T}^2)+TP(\|X\|_{X_T}^2)\\
&+(N_0+\sqrt{T}C_M)T\sup_{t\in[0,T]}\int_{\Omega}\dfrac{|X_{ttt}|}{\rho_0}\,dx.
\end{align*}
It remains to estimate the pair $\int_0^T\langle\bar W_{ttt}, X_{ttt}\rangle\,dt$. We expand $\bar{W}_{ttt}$ as
\begin{align*}
\bar{W}_{ttt}=&-\underbrace{\partial_{ttt}(2\bar J^{-1}(\bar J_t)^2-\partial_t\bar{F}^{*j}_{\,\,\,\,i}\bar v^i_{,j})}_{L_1}-\underbrace{2\partial_{ttt}[\bar{F}^{*j}_{\,\,\,\,i}\bar {F}^{-1k}_{\quad i}(\rho_0 \bar {J}^{-1})_{,k}]_{,j}}_{L_2}\\
&-\underbrace{2\kappa\partial_{ttt}[\bar{F}^{*j}_{\,\,\,\,i}\partial_t\bar {F}^{-1k}_{\quad i}(\rho_0\bar J^{-1})_{,k}]_{,j}}_{L_3}-\underbrace{\kappa \partial_t^3(\bar J\partial_t(\bar{F}^{-1j}_{\quad i})\partial_j \bar G^i)}_{L_4}.
\end{align*}
By the Cauchy--Schwarz inequality, then we have
\begin{align*}
\int_0^T\langle L_1, X_{ttt}\rangle\,dt\leq& \int_0^T \|\sqrt{\rho_0}\partial_{ttt}(2\bar J^{-1}(\bar J_t)^2-\partial_t{\bar{F^*}}_i^j\bar v^i_{,j}\|_0\|\dfrac{ X_{ttt}}{\sqrt{\rho_0}}\|_0\,dt\\
\leq &TP(\|\bar v\|^2_{Z_T})+CT\sup_{t\in[0,T]}\int_{\Omega}\dfrac{| X_{ttt}|^2}{\rho_0}\,dx.
\end{align*}

We notice that the higher-order derivatives in $\partial_{ttt}[\bar{F}^{*j}_{\,\,\,\,i}\bar {F}^{-1k}_{\quad i}(\rho_0 \bar {J}^{-1})_{,k}]_{,j}$ scale like either $D(\rho_0D\bar v_{tt})$ or $D\bar v_{tt}$, so by the fundamental theorem of calculus and the Cauchy--Young inequality once again, we have that for $\delta >0$,
\begin{equation*}
\int_0^T\langle L_2, X_{ttt}\rangle\,dt=TP(\|\bar v\|_{Z_T}^2)+\delta\| X\|_{X_T}^2.
\end{equation*}
Next, we estimate $L_3$. We expand $L_3$ as
\begin{align*}
&\quad\kappa\partial_{ttt}[{\bar{F^*}}_i^j\partial_t{\bar{F^{-1}}}_i^k(\rho_0\bar J^{-1})_{,k}]_{,j}\\&=\kappa\partial_{ttt}[{\bar{F^*}}_i^j\partial_t{\bar{F^{-1}}}_i^k{\rho_0}_{,k}\bar J^{-1}+\rho_0{\bar{F^*}}_i^j\partial_t{\bar{F^{-1}}}_i^k{\bar J}_{,k}\bar J^{-2}]_{,j}\\&=\kappa[{\bar{F^*}}_i^j\partial_t^4{\bar{F^{-1}}}_i^k{\rho_0}_{,k}\bar J^{-1}+\rho_0{\bar{F^*}}_i^j\partial_t^4{\bar{F^{-1}}}_i^k{\bar J}_{,k}+\rho_0{\bar{F^*}}_i^j\partial_t{\bar{F^{-1}}}_i^k\partial_t^3{\bar J}_{,k}\bar J^{-2}]_{,j}+\text{l.o.t},
\end{align*}
and then by \eqref{dF}, the highest-order derivatives are $\kappa {\rho_0}_{,k}\partial_t^3 \bar v_{,\beta j}\bar{F}^{-1j}_{\quad i}\bar{F}^{-1k}_{\quad\alpha}\bar{F^{-1}}_i^{\beta}$ and $\kappa\rho_0D^3\partial_t^2\bar v$.
All of the other terms arising from the distribution of $\partial_{ttt}$ are lower-order and can be estimated in the same way as $L_2$.

Integrating by parts, we have
\begin{align*}
&\langle 2\kappa {\rho_0}_{,k}\partial_t^3 \bar v_{,\beta j}\bar{F}^{-1j}_{\quad i}\bar{F}^{-1k}_{\quad\alpha}\bar{F^{-1}}_i^{\beta},{X}_{ttt}\rangle\\=&-2\kappa\int_{\Omega}{\rho_0}_{,k}\partial_t^3 \bar v_{,\beta }\bar{F}^{-1j}_{\quad i}\bar{F}^{-1k}_{\quad\alpha}\bar{F^{-1}}_i^{\beta}\partial_j{X}_{ttt}\,dx\\&-2\kappa\int_{\Omega}\partial_j({\rho_0}_{,k}\bar{F}^{-1j}_{\quad i}\bar{F}^{-1k}_{\quad\alpha}\bar{F^{-1}}_i^{\beta})\partial_t^3 \bar v_{,\beta }{X}_{ttt}\,dx
\\\leq&\dfrac{4\kappa|D\rho_0|^2_{L^{\infty}(\Omega)}}{C_p}\|\partial_t^3D\bar v\|_0^2+\dfrac{C_p\kappa}{2}\|DX_{ttt}\|
_0^2\\&+C\|\sqrt{\rho_0}\partial_j({\rho_0}_{,k}\bar{F}^{-1j}_{\quad i}\bar{F}^{-1k}_{\quad\alpha}\bar{F^{-1}}_i^{\beta})\partial_t^3 \bar v_{,\beta }\|_0\|\dfrac{X_{ttt}}{\sqrt{\rho_0}}\|_0.
\end{align*}
Thus, when $\kappa \leq \kappa_0$, where $\kappa_0$ satisfying
\begin{equation}
\label{kappachoice}
\dfrac{4\kappa_0|D\rho_0|^2_{L^{\infty}(\Omega)}}{C_p}\leq \delta,\,\,\dfrac{C_p\kappa_0}{2}\leq \delta,
\end{equation}
we can bound $\int_{0}^{T}\langle L_3, X_{ttt}\rangle\,dt$ by
\begin{equation*}
\delta \|\bar v\|_{Z_T}^2+TP(\|\bar v\|_{Z_T}^2)+\delta\| X\|_{X_T}^2+CT\sup_{t\in[0,T]}\int_{\Omega}\dfrac{| X_{ttt}|^2}{\rho_0}\,dx.
\end{equation*}

Last, we estimate $L_4$. The highest-order of $L_4$ can be written as
\begin{equation*}
\underbrace{\kappa \bar J D\bar v_{ttt}\bar G_{,j}}_{\mathfrak{j}_1}+\underbrace{\kappa \bar J\partial_t\bar{F}^{-1j}_{\quad i}\partial_t^3\partial_j \bar G^i}_{\mathfrak{j}_2},
\end{equation*}
and
\begin{equation*}
\langle\mathfrak{j}_1,X_{ttt}\rangle\leq C\|\sqrt{\rho_0}D\bar v_{ttt}\|_0\|\dfrac{X_{ttt}}{\sqrt{\rho_0}}\|_0,
\end{equation*}
\begin{equation*}
\langle\mathfrak{j}_2,X_{ttt}\rangle\leq C\|\sqrt{\rho_0}\partial_t^3D\bar G\|_0\|\dfrac{X_{ttt}}{\sqrt{\rho_0}}\|_0,
\end{equation*}
Then with \eqref{est:G}, all these terms and lower-order terms can be estimated by the same way as $L_1$ or $L_2$.

Combining the above estimates together and taking $T>0$ sufficiently small concludes the proof.
\end{proof}
It is easy to see that we have the same estimates for the weak solutions $ X,  X_t, X_{tt}$ solving \eqref{lxversion}, \eqref{lxtversion} and \eqref{lx2tversion}, respectively, and we finally get that
\begin{align}
\nonumber \sum_{a=0}^3&\|\partial_t^a\dfrac{ X_t}{\rho_0}\|^2_{L^2(0,T;H^{-1}(\Omega))}+\sup_{t\in[0,T]}\|\dfrac{\partial_t^a X(t)}{\sqrt{\rho_0}}\|_0^2+C_p\kappa\|\partial_t^a X\|^2_{L^2(0,T; H^1_0(\Omega))}
\\&\leq N_0+TP(\| v\|_{X_T}^2)+C\delta\| X\|_{X_T}^2+TP(\|\bar v\|_{Z_T}^2)+C\delta(\|\bar v\|^2_{Z_T}).
\label{8.68}
\end{align}
\textbf{Step 2.} Recall the definition of $\mathfrak{W}_2$ in \eqref{8.29}, by using the estimate \eqref{8.68} together with the Hardy inequality, we can get
\begin{align*}
\|\bar W_{tt}+\mathfrak{W}_2\|_{L^2(0,T;L^2(\Omega))}\leq& N_0+TP(\|\bar v\|_{X_T}^2)+C\delta\| X\|_{X_T}^2\\&+TP(\|\bar v\|_{Z_T}^2)+C\delta(\|\bar v\|^2_{Z_T}).
\end{align*}
Combining this with the estimate \eqref{lx3test}, the equation \eqref{lx2t} shows that
\begin{align*}
4\kappa^2\|[\bar B^{jk}( X_{tt})_{,k}]_{,j}\|_{L^2(0,T;L^2(\Omega))}\leq &N_0+TP(\|X\|_{X_T}^2)+C\delta\| X\|_{X_T}^2\\&+TP(\|\bar v\|_{Z_T}^2)+C\delta\|\bar v\|^2_{Z_T}.
\end{align*}
By the elliptic estimates, we can obtain the desired bound
\begin{equation}
\label{x2test}
\| X_{tt}\|_{L^2(0,T;L^2(\Omega))}^2\leq N_0+TP(\|X\|_{X_T}^2)+C\delta\| X\|_{X_T}^2+TP(\|\bar v\|_{Z_T}^2)+C\delta\|\bar v\|^2_{Z_T}.
\end{equation}
\textbf{Step 3.} From the definition of $\mathfrak{W}_1$, we can similarly get that
\begin{align*}
\|\bar{W}_t+\mathfrak{W}_1\|_{L^2(0,T;H^1(\Omega))}\leq &N_0+TP(\|X\|_{X_T}^2)+C\delta\| X\|_{X_T}^2\\&+TP(\|\bar v\|_{Z_T}^2)+\delta\|\bar v\|^2_{Z_T},
\end{align*}
and thus, following the argument we used for the regularity of $ X_{tt}$, we can obtain that
\begin{align}
\nonumber \|{X}_t\|_{L^2(0,T;H^3(\Omega))}\leq &N_0+TP(\|X\|_{X_T}^2)+C\delta\| X\|_{X_T}^2\\&+TP(\|\bar v\|_{Z_T}^2)+C\delta\|\bar v\|^2_{Z_T}.
\label{xtest}
\end{align}
\textbf{Step 4.} Finally, since
\begin{align*}
\|\bar W\|_{L^2(0,T;H^2(\Omega))}^2\leq &N_0+TP(\|X\|_{X_T}^2)+C\delta\| X\|_{X_T}^2\\&+TP(\|\bar v\|_{Z_T}^2)+C\delta\|\bar v\|^2_{Z_T},
\end{align*}
we have that
\begin{align}
\nonumber\|{X}\|_{L^2(0,T;H^4(\Omega))}\leq &N_0+TP(\|X\|_{X_T}^2)+C\delta\| X\|_{X_T}^2\\&+TP(\|\bar v\|_{Z_T}^2)+C\delta\|\bar v\|^2_{Z_T}.\label{xest}
\end{align}
\subsubsection{The proof of Proposition \ref{p8}}
Combining the inequalities \eqref{lx3test}, \eqref{x2test}, \eqref{xtest} and \eqref{xest}, we have
\begin{equation*}
\| X\|_{X_T}^2\leq N_0+TP(\|X\|_{X_T}^2)+C\delta\| X\|_{X_T}^2+TP(\|\bar v\|_{Z_T}^2)+C\delta\|\bar v\|^2_{Z_T}.
\end{equation*}
Given $\mu\ll 1$, taking $\delta>0$ and $T>0$ sufficiently small and readjusting the constant, we see that
\begin{equation*}
\| X\|_{X_T}^2\leq N_0+TP(\|\bar v\|_{Z_T}^2)+\mu\|\bar v\|^2_{Z_T}.
\end{equation*}
Because the right-hand side does not depend on $\nu>0$, we can pass the limit as $\nu\rightarrow 0$ in \eqref{lx}. This completes the proof of Proposition \ref{p8}.

\subsection{Existence of the fixed-point and the proof of Theorem \ref{th2}}
\label{s89}
The purpose of this subsection is to construct smooth unique solutions to \eqref{eq:vdef}, and to show that the map $\bar v \mapsto v$ has a unique fixed-point. This fixed-point is a solution to our approximate $\kappa$-problem \eqref{approeq}. The arguments are quite similar to that in \cite{DS_2010}, we include them below for completeness and self-contained presentation.

\subsubsection{Solution to (\ref{eq:vdef}) via intermediate $\theta$-regularization and the existence of the fixed point of the map $\bar v\mapsto v$}
We will establish the existence of a solution $v$ to \eqref{eq:vdef} by two stages. Firstly, we consider the $\theta$-regularized system for any $\theta>0$, where the higher-in-space order term in \eqref{xcvj} is smoothed via two boundary convolution operators on $\Gamma$:
\begin{align}
\Div v_t^{\theta}&=\Div \bar{v}_t-\Div_{\bar{\eta}}\bar{v}_t+\dfrac{[{X}\bar{J}^2]_t}{\rho_0}-\partial_t\bar{F}^{*j}_{\,\,\,\,i}\bar{v}^i_{,j}\quad\quad\,\,\,\,\text{in}\,\,\Omega
\\
\curl v_t^{\theta}&=\curl \bar{v}_t-\curl_{\bar\eta}\bar v_t+\kappa\epsilon_{\cdot ji}\bar v_{,s}^r\bar{F}^{-1s}_{\quad i}D_{\bar\eta^r}\bar{\mathfrak{F}}^j+\bar{\mathfrak{C}}\,\text{in}\,\,\Omega\label{cpp3p}\\\nonumber
(v^{\theta})_t^3+2\kappa\rho_{0,3}\Lambda_{\theta}^2\Div_{\Gamma} v^{\theta}&= 2\kappa\rho_{0,3}\Lambda_{\theta}\Div_{\Gamma}\bar v -2\Lambda_{\theta}[\bar{J}^{-2}\bar{F}^{*3}_{\,\,\,\,3}]\rho_{0,3}\\\nonumber &\quad-2\kappa\Lambda_{\theta}[\partial_t[\bar{J}^{-2}\bar{F}^{*3}_{\,\,\,\,3}]]\rho_{0,3}
\\&\quad+\Lambda_{\theta}(\bar{G}^3+\kappa\partial_t\bar G^3)+\bar c_{\theta}(t)N^3 \,\, \quad\quad\quad\quad\quad\,\,\text{on}\,\,\Gamma\label{xcv3j}\\
\nonumber\int_{\Omega}(v^{\theta})_t^{\alpha}\,dx&=-2\int_{\Omega}\bar{F}^{-1k}_{\quad\alpha}(\dfrac{\rho_0}{\bar J})_{,k}\, dx -2\kappa\int_{\Omega}\partial_t[\bar{F}^{-1k}_{\quad\alpha}(\dfrac{\rho_0}{\bar J})_{,k}]\,dx\\&\quad+\int_{\Omega}\bar G^{\alpha}+\kappa\partial_t\bar G^{\alpha}\,dx,
\end{align}
where the vector $\bar{\mathfrak{F}}$ is defined in \eqref{cppp} and the function $\bar c_{\theta}(t)$ on the right-hand side of \eqref{xcv3j} is defined by
\begin{align}
\nonumber \bar c_{\theta}(t)=& \dfrac{1}{2}\int_{\Omega}(\Div \bar v_t-\Div_{\bar\eta}\bar v_t)\,dx+\dfrac{1}{2}\int_{\Omega}\dfrac{[ X\bar J^2]_t}{\rho_0}\,dx-\dfrac{1}{2}\int_{\Omega}\partial_t\bar{F}^{-1j}_{\quad i}\bar v^i_{,j}\,dx\\\nonumber
&+\int_{\Gamma}\Lambda_{\theta}[\bar{J}^{-2}\bar{F}^{*3}_{\,\,\,\,3}]\rho_{0,3} N^3 \,dS+\kappa\int_{\Gamma}\Lambda_{\theta}[\partial_t[\bar{J}^{-2}\bar{F}^{*3}_{\,\,\,\,3}]]\rho_{0,3} N^3\,dS\\
&+\kappa\int_{\Gamma}\Div_{\Gamma}(\Lambda_{\theta}^2v^{\theta}-\Lambda_{\theta}\bar v)\rho_{0,3} N^3\,dS+\int_{\Gamma}(\bar G^3+\kappa\partial_t\bar G^3) N^3\,dS.
\end{align}
Then we prove that the existence of a solution to $\theta$-regularized problem for a small $T_{\theta}>0$ by a fixed-point approach.

Secondly, via $\theta$-independent energy estimates on the solution to \eqref{xcv3j}, we can obtain that the time internal of existence $T_{\theta}=T$ independent of $\theta$, and that the sequence $v^{\theta}$ converges in an appropriate space to a solution $v$ of \eqref{eq:vdef}, which also satisfies the same energy estimates. These estimates then allow us to conclude the existence of a fixed-point $v=\bar v$.

\textbf{Step 1: Solutions to \eqref{xcv3j} via the contraction mapping principle.} For
\begin{equation}
\label{8.76}
\omega \in \Upsilon_T^3=\{\omega \in L^2(0,T;H^3(\Omega)):\partial_t^s\omega \in L^2(0,T;H^{4-s})), 1\leq s\leq 3 \}.
\end{equation}
with norm $\|\omega\|_{\Upsilon_T^3}^2=\|\omega\|_{L^2(0,T;H^3(\Omega))}+\sum_{s=1}^3\|\partial_t^s \omega\|_{L^2(0,T;H^{4-s}))}^2$, we set $\Phi(\omega)=u_0+\int_0^t\partial_t\Phi(\omega)$, where $\Phi(\omega)$ is defined by the elliptic system which specifies the divergence, curl, and normal trace of the vector field $\partial_t\Phi(\omega)$:
\begin{subequations}
\label{eq:theta}
\begin{align}
\label{adii}
\Div \partial_t\Phi(\omega)&=\Div \bar{v}_t-\Div_{\bar{\eta}}\bar{v}_t+\dfrac{[{X}\bar{J}^2]_t}{\rho_0}-\partial_t\bar{F}^{*j}_{\,\,\,\,i}\bar{v}^i_{,j}\quad\quad\quad\,\,\,\text{in}\,\,\Omega,
\\
\curl \partial_t\Phi(\omega)&=\curl \bar{v}_t-\curl_{\bar\eta}\bar v_t+\kappa\epsilon_{\cdot ji}\bar v_{,s}^r\bar{F}^{-1s}_{\quad i}D_{\bar\eta^r}\bar{\mathfrak{F}}^j+\bar{\mathfrak{C}} \quad\text{in}\,\,\Omega,\label{cp3p3p}\\\nonumber
\partial_t\Phi(\omega)\cdot e_3&=-2\kappa\rho_{0,3}\Lambda_{\theta}^2\Div_{\Gamma} \omega+ 2\kappa\rho_{0,3}\Lambda_{\theta}\Div_{\Gamma}\bar v -2\Lambda_{\theta}[\bar{J}^{-2}\bar{F}^{*3}_{\,\,\,\,3}]\rho_{0,3}\\\nonumber&\quad-2\kappa\Lambda_{\theta}[\partial_t[\bar{J}^{-2}\bar{F}^{*3}_{\,\,\,\,3}]]\rho_{0,3}
\\&\quad+(\bar{G}^3+\kappa\partial_t\bar G^3)+\bar c(\omega)(t) \quad\quad\quad\quad\quad\quad\quad\quad\,\, \text{on}\,\,\Gamma,\label{xc3v3j}\\
\nonumber\int_{\Omega}(\partial_t\Phi(\omega))^{\alpha}\,dx&=-2\int_{\Omega}\bar{F}^{-1k}_{\quad\alpha}(\dfrac{\rho_0}{\bar J})_{,k}\, dx -2\kappa\int_{\Omega}\partial_t[\bar{F}^{-1k}_{\quad\alpha}(\dfrac{\rho_0}{\bar J})_{,k}]\,dx\\&\quad+\int_{\Omega}\bar G^{\alpha}+\kappa\partial_t\bar G^{\alpha}\,dx,
\end{align}
\end{subequations}
The function $\bar c(\omega)(t)$ is defined by
\begin{align}
\nonumber [\bar c(\omega)](t)=& \dfrac{1}{2}\int_{\Omega}(\Div \bar v_t-\Div_{\bar\eta}\bar v_t)\,dx+\dfrac{1}{2}\int_{\Omega}\dfrac{[ X\bar J^2]_t}{\rho_0}\,dx-\dfrac{1}{2}\int_{\Omega}\partial_t\bar{F}^{-1j}_{\quad i}\bar v^i_{,j}\,dx\\\nonumber
&+\int_{\Gamma}\Lambda_{\theta}[\bar{J}^{-2}\bar{F}^{*3}_{\,\,\,\,3}]\rho_{0,3}  \,dS+\kappa\int_{\Gamma}\Lambda_{\theta}[\partial_t[\bar{J}^{-2}\bar{F}^{*3}_{\,\,\,\,3}]]\rho_{0,3} \,dS\\
&+\kappa\int_{\Gamma}\Div_{\Gamma}(\Lambda_{\theta}^2\omega-\Lambda_{\theta}\bar v)\rho_{0,3} \,dS+\int_{\Gamma}(\bar G^3+\kappa\partial_t\bar G^3) \,dS,
\end{align}
such that the elliptic system \eqref{eq:theta} satisfies all solvability conditions. Thus, the problem $\partial_t\Phi(\omega)$ is perfectly well-posed. Applying Proposition \ref{curllemma} to \eqref{eq:theta} and its first, second, and third time-differentiated versions, we have
\begin{equation}
\label{contract}
\|\partial_t\Phi(\omega)-\partial_t\Phi(\tilde\omega)\|_{\Upsilon_{T_{\theta}}^3}\leq C(M,\theta)\|\omega-\tilde\omega\|_{\Upsilon_{T_{\theta}}^3}+C_MT_{\theta}\|\omega-\tilde\omega\|_{\Upsilon_{T_{\theta}}^3},
\end{equation}
the $\theta$-dependence in the constant $C(M,\theta)$ coming from repeated use of \eqref{molli}. The lack of $\omega$ on the right-hand side of \eqref{adii} and \eqref{cp3p3p} implies that both the divergence and curl of $\partial_t\Phi(\omega)-\partial_t\Phi(\tilde\omega)$ vanish, and that on $\Gamma$:
\begin{equation*}
[\partial_t\Phi(\omega)-\partial_t\Phi(\tilde\omega)]\cdot e_3=2\kappa{\rho_0}_{,3}\Lambda_{\theta}^2\Div_{\Gamma}(\tilde \omega-\omega)+[\bar c(\omega)-\bar c(\tilde\omega)]N^3.
\end{equation*}
Then it follows from \eqref{contract} that
\begin{equation*}
\|\Phi(\omega)-\Phi(\tilde\omega)\|_{\Upsilon_{T_{\theta}}^3}\leq T_{\theta}C(M,\epsilon)\|\omega-\tilde\omega\|_{\Upsilon_{T_{\theta}}^3},
\end{equation*}
therefore the mapping $\Phi: \Upsilon_{T_{\theta}}^3\mapsto \Upsilon_{T_{\theta}}^3$ is a contraction if $T_{\theta}$ is taken sufficiently small, leading to the existence and uniqueness of a fixed-point $v^{\theta}=\Phi(v^{\theta})$, which therefore is a solution of \eqref{eq:theta} on $[0,T_{\theta}]$.

\textbf{Step 2. $\theta$-independent energy estimates for $v^{\theta}$.} Having obtained a unique solution to \eqref{eq:theta}, we now proceed with $\theta$-independent estimates on this system. We integrate the divergence and curl equations in time, and we now view the PDE for the normal trace as a parabolic equation for $v^{\theta}$ on $\Gamma$:
\begin{subequations}
\label{equ:theta}
\begin{align}
v^{\theta}(0)&=u_0\quad\quad\quad\quad\quad\quad\quad\quad\quad\quad\quad\quad\quad\text{in}\,\,\Omega,\\\label{addii}
\Div v^{\theta}&=\Div \bar{v}-\Div_{\bar{\eta}}\bar{v}+\dfrac{[{X}\bar{J}^2]}{\rho_0}\quad\quad\quad\quad\,\text{in}\,\,\Omega,
\\\nonumber
\curl v^{\theta}&=\curl u_0+\curl \bar{v}-\curl_{\bar\eta}\bar v+\kappa\int_0^t\epsilon_{\cdot ji}\bar v_{,s}^r\bar{F}^{-1s}_{\quad i}D_{\bar\eta^r}\bar{\mathfrak{F}}^j\\
&\quad+\int_0^t(\epsilon_{\cdot ji}\bar v^i_{,s}\partial_t\bar{F}^{-1s}_{\quad j}+\bar{\mathfrak{C}})\quad\quad\quad\,\,\,\,\text{in}\,\,\Omega,\label{c3p}\\\nonumber
(v^{\theta})_t^3+2\kappa\rho_{0,3}\Lambda_{\theta}^2\Div_{\Gamma} v^{\theta}&= 2\kappa\rho_{0,3}\Lambda_{\theta}\Div_{\Gamma}\bar v -2\Lambda_{\theta}[\bar{J}^{-2}\bar{F}^{*3}_{\,\,\,\,3}]\rho_{0,3}\\\nonumber&\quad-2\kappa\Lambda_{\theta}[\partial_t[\bar{J}^{-2}\bar{F}^{*3}_{\,\,\,\,3}]]\rho_{0,3}
\\&\quad+(\bar{G}^3+\kappa\partial_t\bar G^3)+\bar c_{\theta}(t)N^3 \,\, \quad\quad\,\,\,\text{on}\,\,\Gamma,\label{xv3j}\\
\nonumber\int_{\Omega}(v^{\theta})_t^{\alpha}\,dx&=-2\int_{\Omega}\bar{F}^{-1k}_{\quad\alpha}(\dfrac{\rho_0}{\bar J})_{,k}\, dx -2\kappa\int_{\Omega}\partial_t[\bar{F}^{-1k}_{\quad\alpha}(\dfrac{\rho_0}{\bar J})_{,k}]\,dx\\&\quad+\int_{\Omega}\bar G^{\alpha}+\kappa\partial_t\bar G^{\alpha}\,dx,
\end{align}
\end{subequations}
We will establish the existence of a fixed-point in $C_{T_{\kappa}}(M)$, but to do so, we will first make use of the space (depending on $\theta$):
\begin{align*}
 \Upsilon_T^4=\{\omega &\in L^{\infty}(0,T;H^{\frac{7}{2}}(\Omega))\cap L^2(0,T;\cdot H^1_0(\Omega)):\partial_t^s\omega \in L^2(0,T;H^{4-s})), \\&1\leq s\leq 3, \Lambda_{\theta}\omega \in L^2(0,T;H^4(\Omega)), \omega(0)=u_0 \},
\end{align*}
endowed with norm
\begin{equation*}
\|\omega\|_{\Upsilon_T^4}^2=\|\Lambda_{\theta}\omega\|_{L^2(0,T;H^4(\Omega)))}^2+\sum_{s=1}^3\|\partial_t^s\omega\|^2_{L^2(0,T;H^{4-s}(\Omega))}+\sup_{[0,T]}\|\omega\|_{3.5}^2.
\end{equation*}
Since $\bar v \in C_{T_{\kappa}(M)}$, equations \eqref{addii} and \eqref{c3p} show that both $\Div v^{\theta}$ and $\curl v^{\theta}$ are in $L^2(0,T_{\theta};H^3(\Omega))$, additionally, from \eqref{xv3j} and \eqref{molli}, we see that $(v^{\theta})^3$ is in $L^{\infty}(0,T_{\theta};H^{3.5}(\Gamma))$, and hence according to Proposition \ref{curllemma}, $v^{\theta}\in L^2(0,T_{\theta};H^4(\Omega))$, with a bound that depends on $\theta$. We next show that, in fact, we can control $\Lambda_{\theta} v^{\theta}$ in $Z_T$ independently of $\theta$, on a time interval $[0,T]$ with $T>0$ independent of $\theta$.

We proceed by acting $\bar\partial^3$ on each side of \eqref{xv3j}, multiplying this equation by $-\frac{N^3}{\rho_{0,3}}\bar\partial^3(v^{\theta})^3$, and then integrating over $\Gamma$. This yields the following identity
\begin{align}
\nonumber
-&\dfrac{1}{2}\dfrac{d}{dt}\int_{\Gamma}\dfrac{N^3}{\rho_{0,3}}|\bar\partial^3(v^{\theta})^3|^2\,dS
+2\kappa\int_{\Gamma}\bar\partial^3\Lambda_{\theta}^2\Div_{\Gamma} v^{\theta} \bar\partial^3(v^{\theta})^3N^3\,dS\\\nonumber
&=-\int_{\Gamma}\mathfrak{G}\bar\partial^3(v^{\theta})^3\dfrac{N^3}{\rho_{0,3}}\,dS-\int_{\Gamma}\bar\partial^3[\rho_{0,3}\Lambda_{\theta}\mathfrak{Q}]\bar\partial^3(v^{\theta})^3\dfrac{N^3}{\rho_{0,3}}\,dS\\
&\quad-\int_{\Gamma}\bar\partial^3(\bar G^3+\bar G_t^3)\bar\partial^3(v^{\theta})^3\dfrac{N^3}{\rho_{0,3}}\,dS,
\label{8.81}
\end{align}
where
\begin{align}
\mathfrak{G}&=-2\kappa[\bar\partial^3\rho_{0,3}\Lambda_{\theta}^2\Div_{\Gamma} v^{\theta}+3\bar\partial^2\rho_{0,3}\bar\partial\Lambda_{\theta}^2\Div_{\Gamma} v^{\theta}+3\bar\partial\rho_{0,3}\bar\partial^2\Lambda_{\theta}^2\Div_{\Gamma} v^{\theta}],\\
\mathfrak{Q}&=2\kappa\Div_{\Gamma}\bar v-2\bar J^{-2}\bar{F}^{*3}_{\,\,\,\,3}-2\kappa\partial_t[\bar J^{-2}\bar{F}^{*3}_{\,\,\,\,3}].
\end{align}
Since $\mathfrak{G}$ contains lower-order terms, we see that for any $t\in[0,T_{\theta}]$
\begin{equation}
-\int_{\Gamma}\mathfrak{G}\bar\partial^3(v^{\theta})^3\dfrac{N^3}{\rho_{0,3}}\,dS \leq C|\bar\partial^3(v^{\theta})|_0^2.
\end{equation}
We then write
\begin{equation}
\bar\partial^3[\rho_{0,3}\Lambda_{\theta}\mathfrak{Q}]=\rho_{0,3}\bar\partial^3\Lambda_{\theta}\mathfrak{Q}+\bar\partial^3\rho_{0,3}
\Lambda_{\theta}\mathfrak{Q}+3\bar\partial^2\rho_{0,3}\bar\partial\Lambda_{\theta}\mathfrak{Q}+3\bar\partial\rho_{0,3}\bar\partial^2\Lambda_{\theta}\mathfrak{Q}
\end{equation}
and notice that since the last three terms on the right-hand side are lower-order, we easily obtain the estimate
\begin{align}
\nonumber |\int_{\Gamma}[\bar\partial^3\rho_{0,3}\Lambda_{\theta}\mathfrak{Q}
+3\bar\partial^2\rho_{0,3}\bar\partial\Lambda_{\theta}\mathfrak{Q}
+&3\bar\partial\rho_{0,3}\bar\partial^2\Lambda_{\theta}\mathfrak{Q}]\bar\partial^3(v^{\theta})^3\dfrac{N^3}{\rho_{0,3}}\,dS|
\\ &\leq |\bar\partial^3(v^{\theta})^3|_0|\bar\partial^2D\bar v|_0.
\end{align}
Next, we estimate the highest-order term $\int_{\Gamma}\bar\partial^3\Lambda_{\theta}\mathfrak{Q}\bar\partial^3(v^{\theta})^3N^3\,dS$. By the standard properties of the boundary convolution operator $\Lambda_{\theta}$, we have that
\begin{align}
\nonumber&\int_{\Gamma}\bar\partial^3\Lambda_{\theta}\mathfrak{Q}\bar\partial^3(v^{\theta})^3N^3\,dS=\int_{\Gamma}\bar\partial^3\mathfrak{Q}\bar\partial^3\Lambda_{\theta}(v^{\theta})^3N^3\,dS
\\\nonumber=&\underbrace{-2\kappa\int_{\Gamma}\bar\partial^3(\bar J^{-2}\partial_t(\bar{F}^{*3}_{\,\,\,\,3}))\bar\partial^3\Lambda_{\theta}(v^{\theta})^3N^3\,dS}_{\mathfrak{J}_1}-\underbrace{2\kappa\int_{\Gamma}\bar\partial^3(\partial_t(\bar J^{-2})\bar{F}^{*3}_{\,\,\,\,3})\bar\partial^3\Lambda_{\theta}(v^{\theta})^3N^3\,dS}_{\mathfrak{J_2}}\\
&+2\underbrace{\kappa\int_{\Gamma}\bar\partial^3(\Div_{\Gamma}\bar v)\bar\partial^3\Lambda_{\theta}(v^{\theta})^3N^3\,dS}_{\mathfrak{J}_3}-\underbrace{2\int_{\Gamma}\bar\partial^3(\bar J^{-2}\bar{F}^{*3}_{\,\,\,\,3})\bar\partial^3\Lambda_{\theta}(v^{\theta})^3N^3\,dS}_{\mathfrak{J}_4}
\end{align}
In order to estimate the integral $\mathfrak{J}_1$, we recall the formula for $\partial_t\bar{F}^{*3}_{\,\,\,\,3}$ given in \eqref{tmp3}, and write
\begin{align}
\nonumber\mathfrak{J}_1=&-2\kappa\int_{\Gamma}\bar\partial^3(\bar J^{-2}[\bar v_{,1}\times\bar\eta_{,2}+\bar\eta_{,1}\times\bar v_{,2}]^3)\bar\partial^3\Lambda_{\theta}(v^{\theta})^3N^3\,dS\\\nonumber
=&\underbrace{-2\kappa\int_{\Gamma}\bar\partial^3\bar v_{,1}\times(\bar\eta_{,2}\bar J^{-2}-\bar\eta_{,2}(0))^3\bar\partial^3\Lambda_{\theta}(v^{\theta})^3N^3\,dS}_{\mathfrak{J}_{1a}}\\\nonumber
&\underbrace{-2\kappa\int_{\Gamma}\bar\partial^3\bar v^1_{,1}\bar\partial^3\Lambda_{\theta}(v^{\theta})^3N^3\,dS-2\kappa\int_{\Gamma}\bar\partial^3\bar v^2_{,2}\bar\partial^3\Lambda_{\theta}(v^{\theta})^3N^3\,dS}_{\mathfrak{J}_{1b}}\\
&\underbrace{-2\kappa\int_{\Gamma}((\bar\eta_{,1}\bar J^{-2}-\bar\eta_{,1}(0))\times\bar\partial^3\bar v_{,2})^3\bar\partial^3\Lambda_{\theta}(v^{\theta})^3N^3\,dS}_{\mathfrak{J}_{1c}}+R_1
\end{align}
where $R_1$ is a lower-order integral over $\Gamma$ that contains all of the remaining terms from the action of $\bar\partial^3$, so that there are at most three space derivatives on $\bar v$ on $\Gamma$. By the trace theorem and the Cauchy--Schwarz inequality, we have $|R_1|\leq C|(v^{\theta})^3|_3\|\bar v\|_4$. Next, we see that $\mathfrak{J}_{1b}+\mathfrak{J}_3=0$. Then we can use the fundamental theorem of calculus,
\begin{equation*}
\bar\eta_{,2}\bar J^{-2}(t)-\bar \eta_{,2}(0)=\int_0^t\partial_t(\bar\eta_{,2}\bar J^{-2})\,\,,\,\,\bar\eta_{,1}\bar J^{-2}(t)-\bar \eta_{,1}(0)=\int_0^t\partial_t(\bar\eta_{,1}\bar J^{-2}),
\end{equation*}
to estimate the integrals $\mathfrak{J}_{1a}$ and $\mathfrak{J}_{1c}$, and obtain that
\begin{equation*}
|\mathfrak{J}_1+\mathfrak{J}_3|\leq Ct\|\Lambda_{\theta}(v^{\theta})^3\|_4\|\bar v\|_4+C|(v^{\theta})^3|_3\|\bar v\|_4
\end{equation*}
For $\mathfrak{J}_2$, we write the integral $\mathfrak{J}_2$ as
\begin{equation*}
\mathfrak{J}_2=4\kappa\int_{\Gamma}\bar{F}^{*3}_{\,\,\,\,3}\bar J^{-2}\Div_{\bar\eta}\bar\partial^3 \bar v \bar\partial^3 \Lambda_{\theta}(v^{\theta})^3N^3\,dS+R_2
\end{equation*}
with $R_2$ scales like $\int_{\Gamma}\bar\partial^3D\bar\eta\bar\partial^3\Lambda_{\theta}(v^{\theta})^3N^3\,dS$. Then it follows that
\begin{equation*}
|R_2|\leq |D\bar\eta|_{2.5}|\Lambda_{\theta}(v^{\theta})^3|_{3.5}\leq C\|\bar\eta\|_4\|\Lambda_{\theta} v^{\theta}\|_4
\end{equation*}
by the trace theorem.  Since $\bar\eta(t)=x+\int_0^t\bar v$, we see that for some $\delta>0$,
\begin{equation*}
|R_2|\leq N_0+\delta\|\Lambda_{\theta}v^{\theta}\|_{4}^2+Ct\|\bar v\|_{L^2(0,T;H^4(\Omega))}^2.
\end{equation*}
For the remaining term in $\mathfrak{J}_2$, we have
\begin{multline}
|4\kappa\int_{\Gamma}\bar{F}^{*3}_{\,\,\,\,3}\bar J^{-2}\Div_{\bar\eta}\bar\partial^3 \bar v \bar\partial^3 \Lambda_{\theta}(v^{\theta})^3N^3\,dS|\\\leq C\|\Div \bar v\|_3\|\Lambda_{\theta} v^{\theta}\|_4 + Ct\|\Lambda_{\theta}(v^{\theta})^3\|_4\|\bar v\|_{L^2(0,T;H^4(\Omega))}
\end{multline}
Finally, the integral $\mathfrak{J}_4$ can be estimated in the same way as $R_2$ above, so we obtain
\begin{align}
\nonumber |\int_{\Gamma}\bar\partial^3&(\rho_{0,3}\Lambda_{\theta}\mathfrak{Q})\bar\partial^3(v^{\theta})\dfrac{N^3}{{\rho_0}_{,3}}\,dS|\\\nonumber
\leq &N_0+\delta\|\Lambda_{\theta}v^{\theta}\|_4^2+Ct\|\bar v\|_{L^2(0,T;H^4(\Omega))}^2+Ct\|\Lambda_{\theta}(v^{\theta})^3\|_4\|\bar v\|_{L^2(0,T;H^4(\Omega))}\\
&+C|(v^{\theta})^3|_3\|\bar v\|_4+C\|\Div \bar v\|_3\|\Lambda_{\theta}v^{\theta}\|_4.
\end{align}
Recall the vacuum condition \eqref{vacuum}, the last term on the right-hand side of \eqref{8.81} can be estimated by
\begin{align*}
\int_{\Gamma}\bar\partial^3(\bar G^3+\kappa\bar G_t^3)\bar\partial^3(v^{\theta})^3N^3\,dS
&\leq C|\bar G+\kappa\partial_t\bar G|_3|\bar\partial^3(v^{\theta})^3|_0
\\&\leq CP(\|\bar v\|_4,\|\rho_0D\bar v\|_4)|\bar\partial^3(v^{\theta})^3|_0
\end{align*}
We now return to estimate the second term on the left-hand side of \eqref{8.81}, which will give us a sign-definite energy term plus a small perturbation. We first see that by the properties of the boundary convolution $\Lambda_{\theta}$,
\begin{equation}
\int_{\Gamma}\bar\partial^3[\Lambda_{\theta}^2({v^{\theta}}_{,1}^1+{v^{\theta}}_{,2}^2)]\bar\partial^3(v^{\theta})^3N^3\,dS
=\int_{\Gamma}\bar\partial^3\Lambda_{\theta}({v^{\theta}}_{,1}^1+{v^{\theta}}_{,2}^2)\bar\partial^3\Lambda_{\theta}v^{\theta}\cdot N \,dS
\end{equation}
Then by applying the divergence theorem to the integral on the right-hand side, we have that
\begin{align*}
&\int_{\Gamma}\bar\partial^3\Lambda_{\theta}({v^{\theta}}_{,1}^1+{v^{\theta}}_{,2}^2)\bar\partial^3\Lambda_{\theta}v^{\theta}\cdot N \,dS\\
=&\int_{\Omega}\bar\partial^3\Lambda_{\theta}({v^{\theta}}_{,1}^1+{v^{\theta}}_{,2}^2)\bar\partial^3\Lambda_{\theta}{v^{\theta}}^3_{,3} \,dx+\int_{\Omega}\bar\partial^3\Lambda_{\theta}({v^{\theta}}_{,13}^1+{v^{\theta}}_{,23}^2)\bar\partial^3\Lambda_{\theta}v^{\theta}\,dx
\\=&-\int_{\Omega}|\bar\partial^3\Lambda_{\theta}(v^{\theta})^3_{,3}|^2\,dx+\int_{\Omega}\Lambda_{\theta}\bar\partial^3\Div v^{\theta}\bar\partial^3\Lambda_{\theta}{v^{\theta}}_{,3}^3\,dx\\
&-\int_{\Omega}\bar\partial^3\Lambda_{\theta}(v^{\theta})^1_{,3}\bar\partial^3\Lambda_{\theta}(v^{\theta})^3_{,1}\,dx
-\int_{\Omega}\bar\partial^3\Lambda_{\theta}(v^{\theta})^2_{,3}\bar\partial^3\Lambda_{\theta}(v^{\theta})^3_{,2}\,dx,\\
=&-\int_{\Omega}|\bar\partial^3\Lambda_{\theta}(Dv^{\theta})^3|^2\,dx+\int_{\Omega}\Lambda_{\theta}\bar\partial^3\Div v^{\theta}\bar\partial^3\Lambda_{\theta}{v^{\theta}}_{,3}^3\,dx\\
&+\int_{\Omega}\Lambda_{\theta}\bar\partial^3[(\curl v^{\theta}\cdot e_2)]\bar\partial^3\Lambda_{\theta}(v^{\theta})^3_{,1}\,dx
-\int_{\Omega}\Lambda_{\theta}\bar\partial^3[(\curl v^{\theta}\cdot e_1)]\bar\partial^3\Lambda_{\theta}(v^{\theta})^3_{,2}\,dx,
\end{align*}
Thanks to \eqref{addii}, we have that, for all $t\in[0,T_{\theta}]$,
\begin{equation}
\|\Div v^{\theta}\|_3^2\leq Ct^2\|\bar v\|_4^2+C\|\bar v\|_3^2+Ct\|\bar v\|^2_{L^2(0,T;H^4(\Omega))}+C\| X\|_4^2,
\end{equation}
where we have used the higher-order Hardy inequality.

On the other hand, with \eqref{c3p}, we see that for all $t\in[0,T_{\theta}]$
\begin{align}
\nonumber \|\curl v^{\theta}\|_3\leq& Ct\|\bar v\|_4+C\|\bar v\|_3+C\|u_0\|_4+C\sqrt{t}\|D\mathfrak{F}\|^2_{L^2(0,T;H^3(\Omega))}\\&+C\sqrt{t}\|\bar v\|_{L^2(0,t;H^4(\Omega))}^2,\\
\leq & Ct\|\bar v\|_4+C_{\kappa}\|u_0\|_4+C\sqrt{t}\|\bar v\|_{L^2(0,t;H^4(\Omega))},
\end{align}
where we have used \eqref{ccp} and the identity \eqref{f0}, relating $\mathfrak{F}$ to $\bar v$. Note that \eqref{f0} provides us with a bound which is $\theta$-independent, but which indeed depends on $\kappa$.

The action of the boundary convolution operator $\Lambda_{\theta}$ does not affect these estimates; thus using Proposition \ref{p8}, we obtain
\begin{align}
\label{p1p}
&\int_0^T\|\Div\Lambda_{\theta}v^{\theta}\|_3^2\,dt \leq N_0+TP(\|\bar v\|_{Z_T}^2)+\mu(\|\bar v\|_{Z_T}^2),\\
&\int_0^T\|\curl\Lambda_{\theta}v^{\theta}\|_3^2\,dt \leq N_0+TP(\|\bar v\|_{Z_T}^2).
\label{p2p}
\end{align}
We combine these estimates and recall vacuum condition \eqref{vacuum}, then for any $t\in[0,T]$,
\begin{align*}
|\bar\partial^3&(v^{\theta})^3(t)|_0^2+\int_0^t\|\bar\partial^3D\Lambda_{\theta}(v^{\theta})^3\|_0^2\\
&\leq N_0+Ct|(v^{\theta})^3|_3^2+Ct\|\bar v\|_{Z_T}^2+Ct\int_0^t\|\Lambda_{\theta}(v^{\theta})^3\|_4^2\\
&\quad +C\sqrt{t}\int_0^t\|\bar v\|_4^2+\frac{C}{\sqrt{t}}\int_0^t|(v^{\theta})^3|_3^2+\delta\|\Lambda_{\theta}v^{\theta}\|_4^2\\
&\quad +TP(\|\bar v\|_{Z_T}^2)+\mu\|\bar v\|_{Z_T}^2+C\|\Div \bar v\|_{Y_T}^2.
\end{align*}
By taking $\delta>0$ sufficiently small, and using \eqref{p1p} and \eqref{p2p}, we obtain
\begin{align*}
|v^{\theta}(t)|_3^2&+\int_0^t\|\Lambda_{\theta}v^{\theta}\|_4^2\\
&\leq N_0+C(t+\sqrt{t})\sup_{t\in[0,T]}|v^{\theta}|_3^2+C(t+\sqrt{t})\|\bar v\|_{Z_T}^2+Ct\int_0^t\|\Lambda_{\theta}v^{\theta}\|_4^2\\
&\quad +TP(\|\bar v\|_{Z_T}^2)+\mu\|\bar v\|_{Z_T}^2+C\|\Div \bar v\|_{Y_T}^2.
\end{align*}
Thus, we have
\begin{align}
\nonumber \sup_{t\in[0,T]}|v^{\theta}(t)|_3^2+\int_0^T\|\Lambda_{\theta}v^{\theta}\|_4^2
&\leq N_0+C\sqrt{T}P(\|v^{\theta}\|_{{\Upsilon}_T^4}^2)\\&\quad+C\sqrt{T}P(\|\bar v\|_{Z_T}^2)+\mu\|\bar v\|_{Z_T}^2+C\|\Div \bar v\|_{Y_T}^2.
\label{9.87}
\end{align}
And then by using the expression \eqref{xv3j}, we can straightforward have the following estimate
\begin{align}
\nonumber \int_0^T|(v_t^{\theta})^3|_{2.5}^2 &\leq N_0+C\int_0^T|\Lambda_{\theta} v^{\theta}|_{3.5}^2\,dt+C\sqrt{T}P(\|\bar v\|_{Z_T}^2)\\\nonumber &\leq N_0+C\sqrt{T}P(\|v^{\theta}\|_{{\Upsilon}_T^4}^2)+C\sqrt{T}P(\|\bar v\|_{Z_T}^2)+P(\|\curl \bar v\|_{Y_T}^2)\\&\quad+C\|\Div \bar v\|_{Y_T}^2.
\label{normaltrace}
\end{align}
With the same argument we used for the estimations for the divergence and curl of $v^{\theta}$, we can also estimate the divergence and curl of $v_t^{\theta}$, and then with the normal trace estimate \eqref{normaltrace}, we obtain the following estimate for $v_t^{\theta}$:
\begin{align}
\nonumber \int_0^T\|(v_t^{\theta})\|_{3}^2 \,dt &\leq N_0+C\sqrt{T}P(\|v^{\theta}\|_{{\Upsilon}_T^4}^2)+C\sqrt{T}P(\|\bar v\|_{Z_T}^2)+\mu\|\bar v\|_{Z_T}^2\\&+C\|\Div \bar v\|_{Y_T}^2.
\label{apl}
\end{align}

Similarly, we can consider the time-differentiated version of \eqref{cpp3p} and using \eqref{normaltrace} and \eqref{apl} to show that
\begin{align}
\nonumber \int_0^T\|(v_{tt}^{\theta})\|_{2}^2 \,dt &\leq N_0+C\sqrt{T}P(\|v^{\theta}\|_{{\Upsilon}_T^4}^2)+C\sqrt{T}P(\|\bar v\|_{Z_T}^2)+\mu\| \bar v\|_{Z_T}^2\\&+C\|\Div \bar v\|_{Y_T}^2,
\label{apl2}
\end{align}
and consider the second time-differentiated version of \eqref{cpp3p} and using \eqref{normaltrace},\eqref{apl} and \eqref{apl2} to show that
\begin{align}
\nonumber \int_0^T\|(v_{ttt}^{\theta})\|_{1}^2 &\leq N_0+C\sqrt{T}P(\|v^{\theta}\|_{{\Upsilon}_T^4}^2)+C\sqrt{T}P(\|\bar v\|_{Z_T}^2)+\mu\|\bar v\|_{Z_T}^2\\&+C\|\Div \bar v\|_{Y_T}^2,
\label{apl3}
\end{align}

Then combining \eqref{normaltrace}, \eqref{apl} to \eqref{apl3}, we have
\begin{align}
\nonumber \|v^{\theta}\|_{\Upsilon_T^4}^2  &\leq N_0+C\sqrt{T}P(\|v^{\theta}\|_{{\Upsilon}_T^4}^2)+C\sqrt{T}P(\|\bar v\|_{Z_T}^2)+C\mu\|\bar v\|_{Z_T}^2\\&+C\|\Div \bar v\|_{Y_T}^2,
\label{apl4}
\end{align}
where the polynomial function $P$ on the right-hand side is independent of $\theta$.

From \eqref{apl4}, we can infer there exists $T>0$ independent of $\theta$ such that $v^{\theta} \in \Upsilon_T^4$ and satisfies the estimate:
 \begin{equation}
\|v^{\theta}\|_{\Upsilon_T^4}^2 \,dt \leq 2N_0+C\sqrt{T}P(\|\bar v\|_{Z_T}^2)+C\mu\|\bar v\|_{Z_T}^2+C\|\Div \bar v\|_{Y_T}^2,
\label{ap34}
\end{equation}

\textbf{Step 3: The limit as $\theta\rightarrow 0$ and the fixed-point of the map $\bar v\mapsto v$}
We set $\theta=\dfrac{1}{n}$, and from \eqref{ap34}, there exists a subsequence and a vector field $v\in\Upsilon_T^3, V\in L^2(0,T;H^4(\Omega))\cap L^{\infty}(0,T;H^3)$ such that
\begin{align}
\label{0zz}
v^{\theta}&\rightharpoonup v\,\, \text{in}\,\, \Upsilon_T^3,\\
v^{\theta}&\rightarrow v\,\, \text{in}\,\, \Upsilon_T^2,\\
\Lambda_{\theta}v^{\theta}&\rightharpoonup V\,\, \text{in}\,\, L^2(0,T;H^4(\Omega)),
\end{align}
where the space $\Upsilon_T^3$ is defined in \eqref{8.76} and
\begin{equation*}
\omega \in \Upsilon_T^2=\{\omega \in L^2(0,T;H^2(\Omega)):\partial_t^s\omega \in L^2(0,T;H^{3-s})), 1\leq s\leq 2 \}.
\end{equation*}
Next, we notice that for any $\phi \in C_c^{\infty}(\Omega)$, we have for each $i=1,2,3$ and $t \in [0,T]$, $T$ still depending on $\kappa>0$, that
\begin{equation*}
\lim_{\theta\rightarrow0}\int_{\Omega}\Lambda_{\theta}(v^{\theta})^i\cdot \phi\,dx = \lim_{\theta\rightarrow0}\int_{\Omega}(v^{\theta})^i\cdot \Lambda_{\theta}\phi\,dx =\int_{\Omega}v^i\phi \,dx,
\end{equation*}
where we used the fact that $\Lambda_{\theta}\phi\rightarrow\phi$ in $L^2(\Omega)$. This shows us that $v=V$, and that
\begin{equation*}
\|v\|_{X_T}^2 \leq 2N_0+C\sqrt{T}P(\|\bar v\|_{Z_T}^2)+C\mu\|\bar v\|_{Z_T}^2+C\|\Div \bar v\|_{Y_T}^2.
\end{equation*}
The estimates for those terms with weight $\rho_0$ in the definition of the $Z_T$-norm follow immediately from multiplication by $\rho_0$ of the equations \eqref{diveq} and \eqref{cppp}. Because $\rho_0=0$ on $\Gamma$ and using the unweighted estimates already obtained, there is no need to consider the parabolic equation \eqref{xcvj}. Then we can get that
\begin{equation}
\label{itr}
\|v\|_{Z_T}^2\leq N_0+C\sqrt{T}P(\|\bar v\|_{Z_T}^2)+C\mu\|\bar v\|_{Z_T}^2+C\|\Div \bar v\|_{Y_T}^2.
\end{equation}
Moreover, the convergence in \eqref{0zz} and the definition of the sequence of problems \eqref{equ:theta} easily show us that $v$ is a solution of the problem \eqref{eq:vdef}. Furthermore, we can obtain the same type of energy estimates for the system \eqref{eq:vdef} as we did in {\bf Step 2} above. This shows the uniqueness of the solution $v$ of \eqref{eq:vdef}, and hence allows us to define $\Theta:\bar v\in Z_T \mapsto v \in Z_T$.

Next, we begin our iteration scheme. We choose any $v^{(1)}\in C_T(M)$ and define for $n \in \mathbb{N}$,
\begin{equation*}
v^{(n+1)}=\Theta(v^{(n)}),\,\, v^{(n)}|_{t=0}=u_0.
\end{equation*}
For each $n \in \mathbb{N}$, we set $\eta^{(n)}(x,t)=x+\int_0^t v^{(n)}(x,t')\,dt', F^{(n)}=D\eta^{(n)}, J^{(n)}=\text{det}D\eta^{(n)}, F^{*(n)}=J^{(n)}{F^{(n)}}^{-1}$, $X^{(n)}$ is the solution to \eqref{lxversion} with $v^{(n)},  F^{*(n)}, J^{(n)}$ and ${F^{(n)}}^{-1}$. Similarly, we define $\mathfrak{F}^{(n)}$ via \eqref{f0} with $v^{(n)}$ replacing $\bar v$ and define $\mathfrak{C}^{(n)}$ via \eqref{lo}.

According to \eqref{itr},
\begin{equation}
\label{itr1}
\|v^{(n+1)}\|_{Z_T}^2\leq N_0+C\sqrt{T}P(\|v^{(n)}\|_{Z_T}^2)+C\mu\|v^{(n)}\|_{Z_T}^2+C\|\Div v^{(n)}\|_{Y_T}^2.
\end{equation}
From \eqref{diveq},
\begin{equation*}
\Div v^{(n)}=\Div v^{(n-1)}-\Div_{\eta^{(n-1)}}v^{(n-1)}+\dfrac{{J^{(n-1)}}^2X^{(n-1)}}{\rho_0},
\end{equation*}
then
\begin{align}
\nonumber \|\Div v^{(n)}\|_{Y_T}^2&\leq \|\Div v^{(n-1)}-\Div_{\eta^{(n-1)}}v^{(n-1)}\|_{Y_T}^2+\|\dfrac{{J^{(n-1)}}^2X^{(n-1)}}{\rho_0}\|_{Y_T}^2\\
&\leq N_0+\sqrt{T}P(\|v^{(n-1)}\|_{Z_T}^2)+C\mu\|v^{(n)}\|_{Z_T}^2,
\end{align}
where we used the higher-order Hardy inequality and Proposition \ref{p8} for the second inequality.
Thus, we obtain the inequality
\begin{multline*}
\|v^{(n+1)}\|_{Z_T}^2\leq N_0+\sqrt{T}P(\|v^{(n)}\|_{Z_T}^2)+\sqrt{T}P(\|v^{(n-1)}\|_{Z_T}^2)\\+\sqrt{T}P(\|v^{(n-2)}\|_{Z_T}^2)+C\mu\|v^{(n)}\|_{Z_T}^2.
\end{multline*}
This shows that choosing $C\mu\leq \dfrac{1}{2}$ and taking $T>0$ sufficiently small and $M\gg N_0$ sufficiently large, the convex set $C_T(M)$ is stable under the map $\Theta$. Furthermore, we can also show that
\begin{multline}
\|v^{(n+1)}-v^{(n)}\|_{Z_T}^2\leq N_0+\sqrt{T}P(\|v^{(n)}-v^{(n-1)}\|_{Z_T}^2)+\sqrt{T}P(\|v^{(n-1)}-v^{(n-2)}\|_{Z_T}^2)\\+\sqrt{T}P(\|v^{(n-2)}-v^{(n-3)}\|_{Z_T}^2)+\dfrac{1}{2}(\|v^{(n)}-v^{(n-1)}\|_{Z_T}^2)
\end{multline}
by a similar argument as we did above. Thus, when $\kappa\leq \kappa_0$, the map $\Theta$ has a unique fixed-point $v=\Theta(v)$ for $T=T_{\kappa}$ sufficiently small. Recall the choice of $\mu$ and $\kappa_0$ (see \eqref{kappachoice}), $\kappa_0$ actually only depends on $\|\rho_0\|_4$ and the domain $\Omega$.
\subsubsection{The fixed-point of the map $\Theta$ is a solution of the $\kappa$-problem}
In this subsection, we will verify that the fixed-point is actually a solution to $\kappa$-problem. Our description is in the Lagrangian coordinates, which is slightly different from \cite{DS_2010}.

In a straightforward manner, we deduce from \eqref{eq:vdef} the following relations for our fixed-point $v=\Theta(v)$:
\begin{subequations}
\label{eq:fix}
\begin{align}
\label{afd}
\Div_{\eta} v_t&= \dfrac{[XJ]_t}{\rho_0}-\partial_t{F^{-1}}_i^jv^i_{,j}\,\,\quad\quad\quad\quad\quad\quad\quad\quad\text{in}\,\,\Omega,\\
\curl v_t&=-\kappa\epsilon_{\cdot ji} v_{,s}^r{F^{-1}}_j^sD_{\eta^r}{\mathfrak{F}}^i+{\mathfrak{C}}\quad\quad\quad\quad\quad\quad\text{in}\,\,\Omega,\label{cpppt}\\\nonumber
v_t^3&= -2{J}^{-2}{F^*}_3^3\rho_{0,3}-2\kappa\partial_t[{J}^{-2}{F^*}_3^3]\rho_{0,3}\\&\quad+{G}^3+\kappa\partial_t G^3+ c(t)N^3 \quad\quad\quad\quad\quad\quad\,\, \text{on}\,\,\Gamma,\label{xcvjt}\\
\nonumber\int_{\Omega}v_t^{\alpha}\,dx&=-2\int_{\Omega}{F^{-1}}_{\alpha}^k(\dfrac{\rho_0}{ J})_{,k}\, dx -2\kappa\int_{\Omega}\partial_t[{F^{-1}}_{\alpha}^k(\dfrac{\rho_0}{ J})_{,k}]\,dx\\&\quad+\int_{\Omega} G^{\alpha}+\kappa\partial_t G^{\alpha}\,dx,\label{per}\\
(x_1,x_2) &\mapsto v_t(x_1,x_2,x_3,t)\,\, \quad\quad\quad\quad\quad\quad\quad\quad\quad\quad\text{is $1$-periodic},
\end{align}
\end{subequations}
where $X$ is a solution of \eqref{appro1}, $G=D_{\eta}\Phi$, $\Phi(x,t)=\int_{\mathbb{R}^3}\dfrac{1}{|y|}\rho(\eta(x,t)-y,t)\,dy$, $f(x,t)=\rho(\eta(x,t),t)=\rho_0J^{-1}(x,t)$ and the function $c(t)$ is defined by
\begin{align*}
c(t)=& \dfrac{1}{2}\int_{\Omega}(\Div  v_t-\Div_{\eta} v_t)\,dx+\dfrac{1}{2}\int_{\Omega}\dfrac{[ X J]_t}{\rho_0}\,dx-\dfrac{1}{2}\int_{\Omega}\partial_t{F^{-1}}_i^j\bar v^i_{,j}\,dx\\
&+\int_{\Gamma}{J}^{-2}{F^*}_3^3\rho_{0,3} N^3 \,dS+\kappa\int_{\Gamma}\partial_t[{J}^{-2}{F^*}_3^3]\rho_{0,3} N^3\,dS\\
&+\int_{\Gamma}( G^3+\kappa\partial_t G^3) N^3\,dS\\=&\dfrac{1}{2}\int_{\Omega}(\Div  v_t-\Div_{\eta} v_t)\,dx+\dfrac{1}{2}\int_{\Omega}\dfrac{[ X J]_t}{\rho_0}\,dx-\dfrac{1}{2}\int_{\Omega}\partial_t{F^{-1}}_i^j\bar v^i_{,j}\,dx\\
&+\int_{\Gamma}D_{\eta}(2f-\Phi)\cdot N \,dS+\kappa\int_{\Gamma}\partial_t[D_{\eta}(2f-\Phi)]\cdot N\,dS,
\end{align*}
where $dS=dx_1dx_2$. By using \eqref{afd} and the divergence theorem, we can obtain the identity (since the volume of $\Omega$ is equal to 1)
\begin{equation}
\label{in:c}
c(t)=\dfrac{1}{2}\int_{\Gamma}[v_t+D_{\eta}(2f-\Phi)+\kappa[D_{\eta}(2f-\Phi)]_t]\cdot N\, dS.
\end{equation}
The fixed-point of the map $\Theta$ also satisfies the equation
\begin{subequations}
\label{fequ}
\begin{align}
\label{eq:fv}
v_t+\mathfrak{F}+\kappa\mathfrak{F}_t&=0,\\
\mathfrak{F}(0)&=2D\rho_0-D\phi_0.\label{eq:fini}
\end{align}
\end{subequations}
Now if we can prove that
\begin{equation}
c(t)=0\quad\text{and}\quad \mathfrak{F}=D_{\eta}(2f-\Phi),
\end{equation}
then from \eqref{xcvjt} and \eqref{fequ}, we can say that the fixed-point is a solution to the $\kappa$-problem \eqref{eq:lag}. We prove this claim by three steps.

\noindent\textbf{Step 1.} We apply $\curl_{\eta}$ on \eqref{eq:fv} and compare it to \eqref{cpppt}. This implies that
\begin{equation}
\kappa\epsilon_{\cdot ji}v_{,s}^r{F^{-1}}_i^sD_{\eta^r}\mathfrak{F}^i+\curl_{\eta}\mathfrak{F}+\kappa[\curl_{\eta}\mathfrak{F}_t]=-\mathfrak{C}.
\end{equation}
Then we have
\begin{equation}
\label{fpsi}
\curl_{\eta}\mathfrak{F}+\kappa[\curl_{\eta}\mathfrak{F}]_t=-D_{\eta}\psi-\kappa[D_{\eta}\psi]_t.
\end{equation}
According to \eqref{eq:fini}, we have $\curl_{\eta}\mathfrak{F}(x,0)=0$. Furthermore, by our definition \eqref{eq:b}, we have $(D_{\eta}\psi)(x,0)=0$ in $\Omega$, then with \eqref{fpsi}, we can conclude that for $t\in[0,T]$,
\begin{equation*}
[\curl_{\eta}\mathfrak{F}+D_{\eta}\psi] (x,t) =0,
\end{equation*}
and therefore we can consider the following elliptic problem
\begin{align*}
\Delta_{\eta\eta}\psi=-\Div_{\eta}(\curl_{\eta}\mathfrak{F})&=0\quad\text{in}\quad\Omega,\\
\psi&=0 \quad \text{on}\quad \Gamma,\\
(x_1,x_2)\mapsto \psi(x_1,x_2,&x_3,t)\,\,\text{is 1-periodic},
\end{align*}
which shows that $\psi=0$ and hence $\mathfrak{C}=0$.

\noindent Therefore, $\curl_{\eta}\mathfrak{F}=0$ in $\Omega$ and  there exists a scalar function $Y$ defined in $\Omega$ such that
\begin{equation}
\mathfrak{F}=D_{\eta}Y.
\end{equation}

It remains to establish that $D_{\eta}Y=D_{\eta}(2f-\Phi)$.

\noindent \textbf{Step 2.} We take the scalar product of \eqref{eq:fv} with $e_3$ to get that
\begin{equation}
v_t^3+D_{\eta}Y\cdot e_3+\kappa[D_{\eta}Y]_t\cdot e_3=0,
\end{equation}
then by comparison with \eqref{xcvjt}, we can get the following identity on $\Gamma$:
\begin{multline}
\label{0lpo}
[D_{\eta}(Y-2f+\Phi)+\kappa[D_{\eta}(Y-2f+\Phi)]_t]\cdot e_3=-c(t)N^3\\
=\dfrac{1}{2}\int_{\Gamma}[D_{\eta}(Y-2f+\Phi)+\kappa[D_{\eta}(Y-2f+\Phi)]_t]\cdot N dS N^3,
\end{multline}
where we used the expression \eqref{in:c} for $c(t)$. Denoting
\begin{equation}
q=2f-\Phi-Y.
\end{equation}
Since $N=(0,0,N^3)$ on $\Gamma$, then \eqref{0lpo} implies:
\begin{equation*}
D_{\eta}q\cdot N+\kappa[D_{\eta}q\cdot N]_t=c(t).
\end{equation*}
By integration with respect to $t$, and taking into account that $(D_{\eta}(2f-\Phi))(x,0)=(D_{\eta}Y)(x,0)$, we have
\begin{equation*}
D_{\eta}q\cdot N=\int_0^t \dfrac{c(s)}{\kappa}e^{\frac{s}{\kappa}}\,ds,
\end{equation*}
which is indeed a function depending only on time. We denote it by $k(t)$. Integrating this relation over $\Gamma$, we finally obtain that on $\Gamma$
\begin{equation}
\label{gamm}
D_{\eta}q\cdot N=k(t)=\dfrac{1}{2}\int_{\Gamma}D_{\eta}q\cdot N \,dS.
\end{equation}

\noindent\textbf{Step 3.} We now apply $\Div_{\eta}$ to \eqref{eq:fv}, and compare the resulting equation with \eqref{appro1}. Using \eqref{afd} and the fact that $X(0)=\rho_0\Div u_0$, we have
\begin{equation}
X=\rho_0J^{-1}\Div_{\eta}v.
\end{equation}
This leads us to
\begin{equation}
\Div_{\eta}[D_{\eta}q+\kappa[D_{\eta}q]_t]=0\quad\text{in}\,\,\Omega.
\end{equation}
This is equivalent to
\begin{equation}
\label{ode}
\Delta_{\eta\eta}q+\kappa[\Delta_{\eta\eta}q]_t-\partial_t{F^{-1}}_i^k\partial_kD_{\eta^i}q=0, \,\,\text{in}\,\, [0,T]\times \Omega.
\end{equation}
Since $2D\rho_0-D\phi_0=DY(0)$, we have that
\begin{equation}
\label{qini}
(\Delta_{\eta\eta}q)(x,0)=0\quad\text{in}\,\,\Omega.
\end{equation}
Also, from \eqref{gamm}, we have the perturbed Neumann boundary condition
\begin{equation}
D_{\eta^3}qN_3=k(t),\quad\text{on}\,\,\Gamma.
\end{equation}
By \eqref{per} and \eqref{fequ}, we obtain that for $\alpha=1,2,$
\begin{equation*}
\int_{\Omega}D_{\eta^{\alpha}}q\,dx+\kappa\int_{\Omega}[D_{\eta^{\alpha}}q]_t\,dx=0,
\end{equation*}
which together with the initial condition $\int_{\Omega}(D_{\eta^{\alpha}})q(x,0)\,dx=0$ implies that
\begin{equation}
\int_{\Omega}D_{\eta^{\alpha}}q(x,t)\,dx=0.
\label{avg}
\end{equation}
Therefore, by setting $\tilde f = \Delta_{\eta\eta}q$, we have the system for all $t\in[0,T]$:
\begin{subequations}
\label{qequ}
\begin{align}
\Delta_{\eta\eta}q&=\tilde f\quad\text{in}\,\,\Omega\label{laplace}\\
\int_{\Omega}D_{\eta^{\alpha}}q\,dx&=0,\\
D_{\eta^3}qN^3&=k(t)\quad\text{on}\,\,\Gamma,
\label{d3q}\\
D_{\eta}q \quad&\text{is 1-periodic in the directions $e_1$ and $e_2$}.\label{dq}
\end{align}
\end{subequations}
We now act $D_{\eta^3}$ on \eqref{laplace}, then multiply the equation by $JD_{\eta^3}q$ and integrate by parts in $\Omega$. With the condition \eqref{dq} and the Piola indentity, we have
\begin{multline*}
-\int_{\Omega}J|D_{\eta^3}D_{\eta}q|^2\,dx+\int_{\Gamma}D_{\eta^i}D_{\eta^3}q{F^*}_i^kN^kD_{\eta^3}q\,dx\\=-\int_{\Omega}\tilde fJD_{\eta^3}D_{\eta^3}q\,dx+\int_{\Gamma}\tilde f{F^*}_i^kN^kD_{\eta^3}q\,dx.
\end{multline*}
We denote $\zeta$ as a smooth function in $\Omega$ such that $\zeta=\dfrac{1}{N^3}$ on $\Gamma$. Then with \eqref{d3q}, we have that
\begin{multline}
\label{vm}
-\int_{\Omega}J|D_{\eta^3}D_{\eta}q|^2\,dx+k(t)\int_{\Gamma}D_{\eta^i}D_{\eta^3}q{F^*}_i^kN^k\zeta\,dS\\=-\int_{\Omega}\tilde fJD_{\eta^3}D_{\eta^3}q\,dx+k(t)\int_{\Gamma}\tilde f{F^*}_3^kN^k\zeta\,dS.
\end{multline}
By the divergence theorem, the boundary integral of \eqref{vm} can be written as
\begin{align*}
&\int_{\Gamma}D_{\eta^i}D_{\eta^3}q{F^*}_i^kN^k\zeta\,dS=\int_{\Omega}J\Delta_{\eta\eta}D_{\eta^3}q\zeta\,dx+\int_{\Omega}D_{\eta^i}D_{\eta^3}JD_{\eta^i}\zeta\,dx,\\
&\int_{\Gamma}\tilde f{F^*}_3^kN^k\zeta\,dS=\int_{\Omega}JD_{\eta^3}\tilde f\zeta\,dx+\int_{\Omega}\tilde f JD_{\eta^3}\zeta\,dx.
\end{align*}
Thus, we can obtain that
\begin{multline*}
\int_{\Omega}J|D_{\eta}D_{\eta^3}q|^2\,dx=\int_{\Omega}\tilde fJD_{\eta^3}D_{\eta^3}q\,dx\\+k(t)\int_{\Omega}D_{\eta^i}D_{\eta^3}JD_{\eta^i}\zeta\,dx-k(t)\int_{\Omega}\tilde f JD_{\eta^3}\zeta\,dx,
\end{multline*}
which provides us with the estimate
\begin{equation}
\label{d3qq}
\|D_{\eta^3}D_{\eta}q\|_0^2\leq C\|\tilde f\|_0^2+Ck^2(t).
\end{equation}
From \eqref{gamm}, by the divergence theorem,
\begin{equation*}
k(t)=\dfrac{1}{2}\int_{\Omega}[D_{\eta^i}q]_{,i}\,dx=\dfrac{1}{2}\int_{\Omega}F^j_iD_{\eta^j}(D_{\eta^i}q)\,dx,
\end{equation*}
and since $|F^j_i-\delta^j_i|\leq Ct$, we have
\begin{equation}
\label{kt}
|k(t)|\leq C\|\Delta_{\eta\eta}q\|_0+Ct\|D_{\eta}D_{\eta}q\|_0.
\end{equation}
Combining with \eqref{d3qq}, we can get
\begin{equation}
\label{d3qqq}
\|D_{\eta^3}D_{\eta}q\|_0^2\leq C\|\tilde f\|_0^2+Ct\|D_{\eta}D_{\eta}q\|_0^2.
\end{equation}

Now we write \eqref{laplace} as
\begin{equation}
\label{d12}
D_{\eta^1}D_{\eta^1}q+D_{\eta^2}D_{\eta^2}q = g,
\end{equation}
where $g=-D_{\eta^3}D_{\eta^3}q+\tilde f$.

\noindent
It follows from \eqref{d12} that
\begin{equation*}
\int_{\Omega}J(D_{\eta^1}D_{\eta^1}q+D_{\eta^2}D_{\eta^2}q)^2\,dx=\int_{\Omega}J|g|^2\,dx.
\end{equation*}
Integrating by parts on the left-hand side of this equation, together with the Piola identity, we have
\begin{multline}
\label{vm9}
\int_{\Omega}J|D_{\eta^{\alpha}}D_{\eta^{\beta}}q|^2\,dx+2\int_{\Gamma}D_{\eta^1}D_{\eta^1}qD_{\eta^2}q{F^*}_2^kN^k\,dS\\
-2\int_{\Gamma}D_{\eta^1}D_{\eta^2}qD_{\eta^2}q{F^*}^k_1N^k\,dS=\int_{\Omega}J|g|^2\,dx.
\end{multline}
For $i=1,2,3$, we smoothly extend $N^i$ into $\Omega$. Recall the definition of $\zeta$, with integration by parts with respect to $x_3$, we have
\begin{multline}
\label{sfa}
\int_{\Gamma}D_{\eta^1}D_{\eta^1}qD_{\eta^2}q{F^*}_2^kN^k\,dS=\int_{\Omega}D_{\eta^1}\partial_3D_{\eta^1}qD_{\eta^2}q{F^*}_2^kN^k\zeta\,dx\\
+\int_{\Omega}\partial_3{F^{-1}}_1^i\partial_iD_{\eta^1}qD_{\eta^2}q{F^*}_2^kN^k\zeta\,dx+\int_{\Omega}D_{\eta^1}D_{\eta^1}q\partial_3(D_{\eta^2}q{F^*}_2^kN^k\zeta)\,dx,
\end{multline}
then by integration by parts for the first term on the right-hand side of \eqref{sfa}, we can get that
\begin{multline}
\int_{\Gamma}D_{\eta^1}D_{\eta^1}qD_{\eta^2}q{F^*}_2^kN^k\,dS=-\int_{\Omega}\partial_3D_{\eta^1}qJD_{\eta^1}(D_{\eta^2}q{F^{-1}}_2^kN^k\zeta)\,dx\\
+\int_{\Gamma}\partial_3D_{\eta^1}qD_{\eta^2}q{F^{-1}}_2^kN^k\zeta{F^*}_1^sN^s\,dx+\int_{\Omega}\partial_3{F^{-1}}_1^i\partial_iD_{\eta^1}qD_{\eta^2}q{F^*}_2^kN^k\zeta\,dx
\\+\int_{\Omega}D_{\eta^1}D_{\eta^1}q\partial_3(D_{\eta^2}q{F^*}_2^kN^k\zeta)\,dx.
\label{vmp1}
\end{multline}
Similarly, we also can get that
\begin{multline}
\int_{\Gamma}D_{\eta^1}D_{\eta^2}qD_{\eta^2}q{F^*}_1^kN^k\,dS=-\int_{\Omega}\partial_3D_{\eta^1}qJD_{\eta^2}(D_{\eta^2}q{F^{-1}}_1^kN^k\zeta)\,dx\\
+\int_{\Gamma}\partial_3D_{\eta^1}qD_{\eta^2}q{F^{-1}}_1^kN^k\zeta{F^*}_2^sN^s\,dx+\int_{\Omega}\partial_3{F^{-1}}_2^i\partial_iD_{\eta^1}qD_{\eta^2}q{F^*}_1^kN^k\zeta\,dx
\\+\int_{\Omega}D_{\eta^1}D_{\eta^2}q\partial_3(D_{\eta^2}q{F^*}_1^kN^k\zeta)\,dx.
\label{vmp2}
\end{multline}
Combining \eqref{vm9}, \eqref{vmp1} and \eqref{vmp2}, we have
\begin{multline*}
\int_{\Omega}J|D_{\eta^{\alpha}}D_{\eta^{\beta}}q|^2\,dx=\int_{\Omega}J|g|^2\,dx+2\int_{\Omega}\partial_3D_{\eta^1}qJD_{\eta^1}(D_{\eta^2}q{F^{-1}}_2^kN^k\zeta)\,dx
\\-2\int_{\Omega}D_{\eta^1}D_{\eta^1}q\partial_3(D_{\eta^2}q{F^*}_2^kN^k\zeta)\,dx
-2\int_{\Gamma}\partial_3D_{\eta^1}qD_{\eta^2}q{F^{-1}}_2^kN^k\zeta{F^*}_1^sN^s\,dx\\+2\int_{\Omega}D_{\eta^1}D_{\eta^2}q\partial_3(D_{\eta^2}q{F^*}_1^kN^k\zeta)\,dx-2\int_{\Omega}\partial_3D_{\eta^1}qJD_{\eta^2}(D_{\eta^2}q{F^{-1}}_1^kN^k\zeta)\,dx\\
+2\int_{\Gamma}\partial_3D_{\eta^1}qD_{\eta^2}q{F^{-1}}_1^kN^k\zeta{F^*}_2^sN^s\,dx,
\end{multline*}
Then with the estimate \eqref{d3qqq}, the relations $|F^j_i-\delta^j_i|\leq Ct$ and $|{F^*}_{\alpha}^kN_k|_{W^{1,\infty}}\leq Ct$, we have
\begin{multline}
\int_{\Omega}|D_{\eta^{\alpha}}D_{\eta^{\beta}}q|^2\,dx\leq C\|\tilde f\|_0^2+Ct\|D_{\eta}D_{\eta}q\|_0^2+Ct\|D_{\eta}D_{\eta}q\|_0\|D_{\eta}q\|_0.
\end{multline}
Combining this estimate with \eqref{d3qqq}, we obtain that
\begin{equation}
\int_{\Omega}|D_{\eta}D_{\eta}q|^2\,dx\leq C\|\tilde f\|_0^2+Ct\|D_{\eta}D_{\eta}q\|_0^2+Ct\|D_{\eta}D_{\eta}q\|_0\|D_{\eta}q\|_0.
\end{equation}
Now, we notice that the conditions \eqref{d3q}, \eqref{dq} and \eqref{avg} yield Poincar\'e's inequalities for $D_{\eta^{\alpha}}q$ and $D_{\eta^3}q$, so we can get
\begin{align*}
&\|D_{\eta^3}q\|_0\leq C\|\partial_3D_{\eta}q\|_0+|k(t)|\leq C\|D_{\eta^3}D_{\eta}q\|_0+C\|\Delta_{\eta\eta}q\|_0+Ct\|D_{\eta}D_{\eta}q\|_0\\
&\|D_{\eta^{\alpha}}q\|_0\leq C\|DD_{\eta^{\alpha}}q\|_0\leq C\|D_{\eta}D_{\eta^{\alpha}}q\|_0+Ct\|D_{\eta}D_{\eta}q\|_0,\quad \alpha=1,2,
\end{align*}
where we also used the estimate \eqref{kt}. Then by taking $T>0$ small enough, we finally have that
\begin{equation}
\label{qqq}
\|D_{\eta}D_{\eta}q\|_0^2\leq C\|\Delta_{\eta\eta}q\|_0^2,
\end{equation}
where $T, C$ depends on $M$.

Therefore, by using the Gronwall inequality, the ODE \eqref{ode} with initial condition \eqref{qini} implies that on $[0,T]\times\Omega$,
\begin{equation}
\Delta_{\eta\eta}q=0.
\label{qq}
\end{equation}
From \eqref{qq} and \eqref{qqq}, we infer that
\begin{equation*}
D_{\eta}q=0 \quad\text{in}\,[0,T]\times\Omega,
\end{equation*}
which proves that $D_{\eta}Y=D_{\eta}(2f-\Phi)$, and therefore $\mathfrak{F}=D_{\eta}(2f-\Phi)$ and $c(t)=0$. This finally establishes that $v$ is a solution of the $\kappa$-problem \eqref{approeq} on a time interval $[0,T_{\kappa}]$ and concludes the proof of Theorem \ref{th2}.
\section{$\kappa$-independent estimates for $\kappa$-problem \ref{approeq} and solutions to the compressible Euler--Poisson equations }
\label{ee}
In this section, we obtain $\kappa$-independent estimates for the smooth solutions to $\kappa$-problem \eqref{approeq}. This allows us to consider the limit of this sequence of solutions as $\kappa\rightarrow 0$. We prove that this limit exists, and it is the unique solution of \eqref{eq:ep}. This kind a prior estimate was developed by Coutand and Shkoller in \cite{DS_09, DS_2010} for the compressible Euler equations. Due to the presence of the potential force, we need to get a estimate for the potential force $G$ at first. The analysis difficulty arising from getting this estimate is that the possible increasing singularity of the convolution kernel, we use Taylor's formula and the Sobolev $C^{\alpha}$ estimates to overcome this difficulty. These calculations and estimates will be explained explicitly in Section \ref{sg2}, then mainly follow the arguments in \cite[Section 9]{DS_2010} to carry out a prior estimates.
\subsection{The higher-order energy function} The higher-order energy function on $[0,T_{\kappa}]$ is defined as follows:
\begin{equation}
  \begin{split}
    \tilde{E}(t;v)&=\sum_{s=0}^4 [\|\partial_t^{2s} \eta(t,.)\|_{{4-s}}^2+\|\rho_0\partial_t^{2s}\bar{\partial}^{4-s}D\eta(t,.)\|_{0}^2
    +\|\sqrt{\rho_0}\partial_t^{2s}\bar{\partial}^{4-s}v(t,.)\|_{0}^2\\&+\int_0^t\|\sqrt{\kappa}\rho_0\partial_t^{2s}\bar{\partial}^{4-s}Dv(\tau,.)\|_{0}^2\,d\tau]
    +\sum_{s=0}^3\|\rho_0\partial_t^{2s}J^{-2}\|_{4-s}^2
    \\&\quad+\| \curl_{\eta}v(t)\|_3^2+\|\rho_0\bar{\partial}^4 \curl_{\eta}v(t)\|_0^2.
\end{split}
  \label{energy}
\end{equation}
The function $\tilde E(t)$ is appropriate for our $\kappa$-independent estimates for \eqref{approeq}.
We also denote $P$ as a generic polynomial function of its arguments whose meaning may change from line to line. Let
\begin{equation}
\tilde{M}_0=P(\tilde{E}(0;v)).
\label{M0000}
\end{equation}
\subsection{Assumptions on a prior bounds on $[0,T_{\kappa}]$}
For the remainder of this section, similar to \cite[Section 9.2]{DS_2010}, we assume that we have solutions $\eta_{\kappa} \in X_{T_{\kappa}}$ on a time interval $[0,T_{\kappa}]$, and that for all such solutions, the time $T_{\kappa}>0$ is taken sufficiently small such that for $t\in[0,T_{\kappa}]$ and $\xi\in\mathbb{R}^3$,
\begin{multline}
\label{ini2}
\dfrac{7}{8}\leq J(t)\leq\dfrac{9}{8},\quad  \frac{7}{8}|\xi|^2\leq J^2{F^{-1}}_l^j{F^{-1}}_k^l\xi_j\xi_k,\\\text{det}\sqrt{g(\eta(t))}\leq 2\text{det}\sqrt{g(\eta_0)}=2,\quad
\|J^{-1}{F^{-1}}_r^k{F^{-1}}_s^k-\delta^r_k\delta^s_k\|_{L^{\infty}}\leq \dfrac{1}{4}.
\end{multline}
and
\begin{equation}
\label{eta}
\dfrac{7}{8}|x-y|\leq|\eta(x,t)-\eta(y,t)|\leq \dfrac{9}{8}|x-y|.
\end{equation}
We further assume that our solutions satisfy the bounds
\begin{align}
\nonumber&\|\eta(t)\|_{H^{3.5}(\Omega)}\leq 2|e|^2_{3.5}+1,\\
\nonumber&\|\partial_t^av(t)\|^2_{H^{3-a/2}(\Omega)}\leq 2\|\partial_t^av(0)\|^2_{H^{3-a/2}(\Omega)}+1 \quad\text{for}\quad a=0,1,\dots,6,\\
\nonumber&\|\rho_0\partial_t^{2a}\eta(t)\|_{H^{4.5-a/2}(\Omega)}^2\leq2\|\rho_0\partial_t^a\eta(0)\|_{H^{4.5-a/2}(\Omega)}^2+1 \quad\text{for} \quad a=0,1,\dots,7,\\
&\|\sqrt{\kappa}\partial_t^{2a+1}v(t)\|^2_{H^{3-a}(\Omega)}\leq 2\|\sqrt{\kappa}\partial_t^{2a+1}v(0)\|_{H^{3-a}(\Omega)}^2+1 \quad\text{for} \quad a=0,1,2,3.
\label{ini}
\end{align}
The right-hand sides in the inequalities \eqref{ini} will be denoted by a generic constant $C$ in the estimates below. We will show that this can be achieved in a time interval independent of $\kappa$. We continue to assume that $\rho_0$ is smooth coming from our approximation \eqref{smoothrho}.
\subsection{The estimates for $G$}
\label{sg2}
As we mentioned in the introduction, our main obstacle is how to overcome the increased singularity of the kernel when we estimate the derivatives of $G$. In this subsection, we will show that we can get a good estimate of $G$ with
the help of the Sobolev embedding inequality ($C^{\alpha}$ estimate) and Taylor's formula, which will be used to reduce the increased singularity.
\begin{proposition}
\label{ineq:g}
For all $t\in(0,T)$, where we take $T\in(0,T_{\kappa}),$
\begin{equation}
\label{gest}
\sum_{a=0}^4[\|\sqrt{\rho_0}\bar{\partial}^{4-a}\partial_t^{2a}G(t)\|_0^2+\int_0^t\kappa\|\sqrt{\rho_0}\bar{\partial}^{4-a}\partial_t^{2a+1}G\|_0^2\,dt'] \leq \tilde{M}_0+CP(\sup_{t\in[0,T]}\tilde{E}(t)).
\end{equation}
\end{proposition}
\begin{proof}
\textbf{Step 1. Estimates for $\|\sqrt{\rho_0}\bar{\partial}^4 G\|_0^2+\int_0^t\kappa\|\sqrt{\rho_0}\bar{\partial}^4\partial_tG(t',\cdot)\|_0^2\,dt'$.}

\noindent From \eqref{forceF}, we have
\begin{align}
\nonumber \bar{\partial}_x^4 G^i=&\sum_{a=0}^4c_a\int_{\Omega}(\bar\partial_x-\bar\partial_z)^a\dfrac{1}{|\eta(x,t)-\eta(z,t)|}[\bar\partial^{4-a}\partial_k(\rho_0{F^{-1}}_i^k)](z,t)\,dz.
\\=&\sum_{a=0}^4 c_ag_a^i.\label{barG}
\end{align}
In fact, by using integration by parts, it is easy to check that
\begin{align*}
\bar\partial G^i=&\int_{\Omega}(\bar\partial_x-\bar\partial_z)\dfrac{1}{|\eta(x,t)-\eta(z,t)|}\partial_k(\rho_0{F^{-1}}_i^k)\,dz\\&\quad+\int_{\Omega}\bar\partial_z\dfrac{1}{|\eta(x,t)-\eta(z,t)|}\partial_k(\rho_0{F^{-1}}_i^k)\,dz\\
=&\int_{\Omega}(\bar\partial_x-\bar\partial_z)\dfrac{1}{|\eta(x,t)-\eta(z,t)|}\partial_k(\rho_0{F^{-1}}_i^k)\,dz\\&\quad-\int_{\Omega}\dfrac{1}{|\eta(x,t)-\eta(z,t)|}\bar\partial_z\partial_k(\rho_0{F^{-1}}_i^k)\,dz,
\end{align*}
then we can get \eqref{barG} inductively.

\noindent We also denote that $R'(\eta)=(\eta(x,t)-\eta(z,t))\cdot(\bar{\partial}\eta(x,t)-\bar{\partial}\eta(z,t))$, and $R(\eta,v)=(\eta(x,t)-\eta(z,t))\cdot(v(x,t)-v(z,t))$ for notational convenience.

\noindent
When $a=0$, by the divergence theorem, we have
\begin{align*}
g_0^i=&-\int_{\Omega}\partial_k\dfrac{1}{|\eta(x,t)-\eta(z,t)|}\bar{\partial}^4(\rho_0{F^{-1}}_i^k)(z,t)\,dz\\=&C\int_{\Omega}\dfrac{(\eta(x,t)-\eta(z,t))\cdot \partial_k\eta(z,t)}{|\eta(x,t)-\eta(z,t)|^3}\bar{\partial}^4(\rho_0{F^{-1}}_i^k)\,dz.
\end{align*}
Recall that \eqref{eta}, and by Taylor's formula
\begin{equation}
\label{taylor}
\eta(x,t)-\eta(z,t)=(x-z) \cdot \int_0^1 D\eta(z+s(x-z))\,ds,
\end{equation}
we can obtain that
\begin{multline*}
|\int_{\Omega}\dfrac{(\eta(x,t)-\eta(z,t))\cdot \partial_k\eta(z,t)}{|\eta(x,t)-\eta(z,t)|^3}\bar{\partial}^4(\rho_0{F^{-1}}_i^k)\,dz|\\\leq C\|D\eta\|_{L^{\infty}(\Omega)}\int_{\Omega}\dfrac{1}{|x-z|^2}|\bar{\partial}^4(\rho_0{F^{-1}}_i^k)|\,dz.
\end{multline*}
By \eqref{dF}, and $\|\dfrac{\bar\partial\rho_0}{\rho_0}\|_{L^{\infty}(\Omega)}\leq C\|\rho_0\|_4$, which is followed by the higher-order Hardy inequality, we have
\begin{align}
\nonumber\bar{\partial}^4(\rho_0{F^{-1}}_i^k)=&\sum_{b=0}^4\bar\partial^b\rho_0\bar\partial^{4-b}{F^{-1}}_i^k\\
=&C\rho_0[(\bar\partial D\eta)^4+\bar\partial^3D\eta\bar\partial D\eta+(\bar\partial^2D\eta)^2+\bar\partial^2D\eta(\bar\partial D\eta)^2+\bar\partial^4D\eta]+\text{l.o.t},
\label{d4eta}
\end{align}
where $C$ depends on $\|{F^{-1}}\|_{L^{\infty}(\Omega)}$ which is close to $1$ by \eqref{ini2}, and $\|\rho_0\|_4$. Thus,
\begin{align*}
\|g_0\|_0^2\leq & C(|D\eta|_{L^{\infty}(\Omega)})\|\dfrac{1}{|x|^2}\|_{L^1(\Omega)}^2\|\bar{\partial}^4(\rho_0{F^{-1}}_i^k)\|_0^2
\\\leq &CP(\|D\eta\|_{L^{\infty}(\Omega)},\|D^2\eta\|_{L^{\infty}(\Omega)})(\|\rho_0\bar\partial^4D\eta\|_0^2+\|\bar\partial^3D\eta\|_0^2+\|\bar\partial^2D\eta\|_{L^4(\Omega)}^4)
\\\leq & CP(\sup_{t\in[0,T]}\tilde{E}(t)).
\end{align*}
On the other hand, we have
\begin{align*}
\partial_tg_0^i=&C\int_{\Omega}\partial_t\partial_k\dfrac{1}{|\eta(x,t)-\eta(z,t)|}\bar{\partial}^4(\rho_0{F^{-1}}_i^k)(z,t)\,dz
\\&+C\int_{\Omega}\partial_k\dfrac{1}{|\eta(x,t)-\eta(z,t)|}\bar{\partial}^4(\rho_0\partial_t{F^{-1}}_i^k)(z,t)\,dz
\\=&C\int_{\Omega}\dfrac{(v(x,t)-v(z,t))\cdot \partial_k\eta(z,t)}{|\eta(x,t)-\eta(z,t)|^3}\bar{\partial}^4(\rho_0{F^{-1}}_i^k)\,dz\\&+C\int_{\Omega}\dfrac{(\eta(x,t)-\eta(z,t))\cdot \partial_kv(z,t)}{|\eta(x,t)-\eta(z,t)|^3}\bar{\partial}^4(\rho_0{F^{-1}}_i^k)\,dz\\&+C\int_{\Omega}\dfrac{(\eta(x,t)-\eta(z,t))\cdot \partial_k\eta(z,t)R(\eta,v)}{|\eta(x,t)-\eta(z,t)|^5}\bar{\partial}^4(\rho_0{F^{-1}}_i^k)\,dz\\
&+C\int_{\Omega}\partial_k\dfrac{1}{|\eta(x,t)-\eta(z,t)|}\bar{\partial}^4(\rho_0\partial_t{F^{-1}}_i^k)(z,t)\,dz,
\end{align*}
and then by using Taylor's formula similarly and \eqref{eta}, we can obtain that
\begin{multline*}
\|\partial_tg_0\|_0^2 \leq CP(\|D\eta\|_{L^{\infty}(\Omega)},\|Dv\|_{L^{\infty}(\Omega)})
\|\dfrac{1}{|x|^2}\|_{L^1(\Omega)}^2[\|\bar{\partial}^4(\rho_0{F^{-1}}_i^k)\|_0^2\\+\|\bar{\partial}^4(\rho_0\partial_t{F^{-1}}_i^k)\|_0^2].
\end{multline*}
With \eqref{d4eta} and
\begin{align*}
\nonumber\bar{\partial}^4(\rho_0\partial_t{F^{-1}}_i^k)=&\sum_{b=0}^4\bar\partial^b\rho_0\bar\partial^{4-b}\partial_t{F^{-1}}_i^k\\\nonumber
=&C\rho_0[((\bar\partial D\eta)^4+\bar\partial^2D\eta(\bar\partial D\eta)^2+(\bar\partial^2 D\eta)^2+\bar\partial^3D\eta\bar\partial D\eta+\bar\partial^4D\eta)Dv\\
&+(\bar\partial^3D\eta+\bar\partial^2D\eta\bar\partial D\eta+(\bar\partial D\eta)^3)\bar\partial Dv\\
&+(\bar\partial^2D\eta+(\bar\partial D\eta)^2)\bar\partial^2Dv+\bar\partial^3Dv\bar\partial D\eta+\bar\partial^4Dv]+\text{l.o.t},
\end{align*}
we have
\begin{align*}
\int_0^t\|\sqrt\kappa\partial_tg_0\|_0^2\,dt' &\leq CP(\|D\eta\|_{L^{\infty}(\Omega)},\|D^2\eta\|_{L^{\infty}(\Omega)},\|Dv\|_{L^{\infty}(\Omega)})\\&\times[\int_0^t(\|\rho_0\bar{\partial}^4D\eta\|_0^2+\|\sqrt\rho_0\bar{\partial}^4Dv\|_0^2\,dt'\\
&+\int_0^t\|\rho_0\bar{\partial}^2Dv\|_{L^3(\Omega)}^2\|\bar{\partial}^2D\eta\|_{L^6(\Omega)}^2+\|\bar{\partial}Dv\|_{L^6(\Omega)}^2\|\rho_0\bar{\partial}^3D\eta\|_{L^3(\Omega)}^2)\,dt']
\\&\leq \tilde{M}_0+CP(\sup_{t\in[0,T]}\tilde{E}(t)),
\end{align*}
where we used \eqref{ini} and the Sobolev embedding inequality to get that $\|\rho_0\bar\partial^2Dv\|_{L^3(\Omega)}\leq \|\rho_0\bar\partial^2Dv\|_1\leq\tilde M_0$ and $\|\bar{\partial}^2D\eta\|_{L^6(\Omega)}\leq \|\eta\|_4$.

\noindent
When $a=1$, using the divergence theorem again, we have
\begin{align*}
g_1^i=&\int_{\Omega}\partial_k(\bar\partial_x-\bar\partial_z)\dfrac{1}{|\eta(x,t)-\eta(z,t)|}[\bar{\partial}^3(\rho_0{F^{-1}}_i^k)](z,t)\,dz\\
=&C\int_{\Omega}\dfrac{\partial_k[R'(\eta)]}{|\eta(x,t)-\eta(z,t)|^3}[\bar{\partial}^3(\rho_0{F^{-1}}_i^k)](z,t)\,dz
\\&+C\int_{\Omega}\dfrac{R'(\eta)(\eta(x,t)-\eta(z,t))\cdot\partial_k\eta(x,t)}{|\eta(x,t)-\eta(z,t)|^5}[\bar{\partial}^3(\rho_0{F^{-1}}_i^k)](z,t)]\,dz.
\end{align*}
Since we have that
\begin{equation}
\label{d3f}
\bar{\partial}^3(\rho_0{F^{-1}}_i^k)=C\rho_0(\bar\partial^3D\eta+\bar\partial^2D\eta\bar\partial D\eta+(\bar\partial D\eta)^3)+\text{l.o.t}.
\end{equation}
Then, similarly, with \eqref{eta},\eqref{taylor}, and \eqref{ini}, we can get
\begin{align*}
|g_1^i|\leq CP(|D\eta|_{L^{\infty}},|D^2\eta|_{L^{\infty}}) \int_{\Omega}|\dfrac{1}{|x-z|^2}[\bar{\partial}^3(\rho_0{F^{-1}}_i^k)](z,t)|\,dz,
\end{align*}
and
\begin{align*}
\|g_1\|_0^2&\leq CP(\|D\eta\|_{L^{\infty}(\Omega)},\|D^2\eta\|_{L^{\infty}(\Omega)})\|\dfrac{1}{|x|^2}\|_{L^1(\Omega)}^2\|\rho_0\bar{\partial}^3D\eta\|_0^2\\
\leq & CP(\sup_{t\in[0,T]}\tilde{E}(t)).
\end{align*}
Next, we have that
\begin{align*}
\partial_tg_1^i=&C\int_{\Omega}\partial_t\partial_k(\bar\partial_x-\bar\partial_z)\dfrac{1}{|\eta(x,t)-\eta(z,t)|}[\bar{\partial}^3(\rho_0{F^{-1}}_i^k)](z,t)\,dz
\\&+C\int_{\Omega}\partial_k(\bar\partial_x-\bar\partial_z)\dfrac{1}{|\eta(x,t)-\eta(z,t)|}[\partial_t\bar{\partial}^3(\rho_0\partial_t{F^{-1}}_i^k)](z,t)\,dz
\\=&C\int_{\Omega}\dfrac{\partial_t\partial_k[R'(\eta)]}{|\eta(x,t)-\eta(z,t)|^3}[\bar{\partial}^3(\rho_0{F^{-1}}_i^k)](z,t)\,dz
\\&+C\int_{\Omega}\dfrac{\partial_k[R'(\eta)]R(\eta,v)}{|\eta(x,t)-\eta(z,t)|^5}[\bar{\partial}^3(\rho_0{F^{-1}}_i^k)](z,t)\,dz
\\&+\underbrace{C\int_{\Omega}\dfrac{\partial_t[R'(\eta)(\eta(x,t)-\eta(z,t))\cdot\partial_k\eta(z,t)]}{|\eta(x,t)-\eta(z,t)|^5}[\bar{\partial}^3(\rho_0{F^{-1}}_i^k)](z,t)\,dz}_{B_1}
\\&+C\int_{\Omega}\dfrac{[R'(\eta)(\eta(x,t)-\eta(z,t))\cdot\partial_k\eta(z,t)]R(\eta,v)}{|\eta(x,t)-\eta(z,t)|^7}[\bar{\partial}^3(\rho_0{F^{-1}}_i^k)](z,t)\,dz
\\&+C\int_{\Omega}\partial_k(\bar{\partial}_x-\bar\partial_z)\dfrac{1}{|\eta(x,t)-\eta(z,t)|}[\bar{\partial}^3(\rho_0\partial_t{F^{-1}}_i^k)](z,t)\,dz.
\end{align*}
We first estimate $B_1$. By a straightforward computation, we have
\begin{align*}
Z(x,z,t):=&\dfrac{\partial_t(R'(\eta)(\eta(x,t)-\eta(z,t))\cdot \partial_k\eta(z,t)}{|\eta(x,t)-\eta(z,t)|^5}\\
=&\dfrac{(\eta(x,t)-\eta(z,t))\cdot(\bar\partial v(x,t)-\bar\partial v(z,t))(\eta(x,t)-\eta(z,t))\cdot \partial_k\eta(z,t)}{|\eta(x,t)-\eta(z,t)|^5}\\
&+\dfrac{(v(x,t)-v(z,t))\cdot(\bar\partial \eta(x,t)-\bar\partial \eta(z,t))(\eta(x,t)-\eta(z,t))\cdot \partial_k\eta(z,t)}{|\eta(x,t)-\eta(z,t)|^5}\\
&+\dfrac{R'(\eta)[(v(x,t)-v(z,t))\cdot \partial_k\eta(z,t)+(\eta(x,t)-\eta(z,t))\cdot\partial_k v(z,t)]}{|\eta(x,t)-\eta(z,t)|^5},
\end{align*}
then by Taylor's formula, the Sobolev embedding theorem
\begin{equation}
\label{calpha}
|\bar\partial v(x,t)-\bar\partial v(z,t)|\leq C \|v\|_3|x-z|^{1/2},
\end{equation}
and \eqref{eta}, we can obtain that
\begin{multline*}
|Z(x,z,t)|\leq C\dfrac{1}{|x-z|^2}(|D\eta||D^2\eta(\tilde x(x,z))||D v(\tilde x(x,z))|+|Dv||D^2\eta(\tilde x(x,z))|)\\
+C\dfrac{1}{|x-z|^{5/2}}|D\eta|\|v\|_3.
\end{multline*}
Thus, we have
\begin{align*}
B_1=&\int_{\Omega}Z(x,z,t)[\bar\partial^3(\rho_0{F^{-1}}_i^k)](z,t)\,dz\\
\leq&C(\|D\eta\|_{L^{\infty}(\Omega)},\|v\|_3)\int_{\Omega}\dfrac{1}{|x-z|^{\frac 5 2}}[\bar{\partial}^3(\rho_0{F^{-1}}_i^k)](z,t)|\,dz.
\end{align*}
Combining with \eqref{d3f}, and using Young's inequality for convolution, we can obtain that,
\begin{align}
\nonumber\|B_1\|_0^2\leq& CP(\|D\eta\|_{L^{\infty}(\Omega)},\|v\|_3)\|\dfrac{1}{|x|^{\frac 5 2}}\|_{L^1(\Omega)}^2\|\bar{\partial}^3(\rho_0{F^{-1}}_i^k)\|_{0}^{2}
\\\nonumber&+CP(\|D\eta\|_{L^{\infty}(\Omega)},\|D^2\eta\|_{L^{\infty}(\Omega)},\|Dv\|_{L^{\infty}(\Omega)})\|\dfrac{1}{|x|^2}\|_{L^1(\Omega)}^2\|\bar{\partial}^3(\rho_0{F^{-1}}_i^k)\|_0^2\\
\leq& \tilde{M}_0+CP(\sup_{t\in[0,T]}\tilde{E}(t)).
\label{tio}
\end{align}
As we show in the case $a=0$, the $L^2$-norm of other terms can also be bounded by $\tilde{M}_0+CP(\sup_{t\in[0,T]}\tilde{E}(t))$, thus
\begin{equation*}
\int_0^t\|\kappa\partial_tg_1\|_0^2\,dt' \leq \tilde{M}_0+CP(\sup_{t\in[0,T]}\tilde{E}(t)).
\end{equation*}
When $a=2$, we have
\begin{align*}
g_2^i=&\int_{\Omega}(\bar{\partial}_x-\bar\partial_z)^2\dfrac{1}{|\eta(x,t)-\eta(z,t)|}[\bar{\partial}^2\partial_k(\rho_0{F^{-1}}_i^k)](z,t)\,dz\\
=&C\int_{\Omega}\dfrac{(\bar{\partial}_x-\bar\partial_z)R'(\eta)}{|\eta(x,t)-\eta(z,t)|^3}[\bar{\partial}^2\partial_k(\rho_0{F^{-1}}_i^k)](z,t)\,dz\\&+
C\int_{\Omega}\dfrac{R'^2(\eta)}{|\eta(x,t)-\eta(z,t)|^5}[\bar{\partial}^2\partial_k(\rho_0{F^{-1}}_i^k)](z,t)\,dz,
\end{align*}
and
\begin{align*}
\partial_tg_2^i=&\int_{\Omega}\partial_t(\bar{\partial}_x-\bar\partial_z)^2\dfrac{1}{|\eta(x,t)-\eta(z,t)|}[\bar{\partial}^2\partial_k(\rho_0{F^{-1}}_i^k)](z,t)\,dz\\
&+\int_{\Omega}(\bar{\partial}_x-\bar\partial_z)^2\dfrac{1}{|\eta(x,t)-\eta(z,t)|}[\partial_t\bar{\partial}^2\partial_k(\rho_0{F^{-1}}_i^k)](z,t)\,dz
\\=&C\int_{\Omega}\dfrac{\partial_t(\bar{\partial}_x-\bar\partial_z)R'(\eta)}{|\eta(x,t)-\eta(z,t)|^3}[\bar{\partial}^2\partial_k(\rho_0{F^{-1}}_i^k)](z,t)\,dz\\&+
C\int_{\Omega}\dfrac{\partial_tR'^2(\eta)}{|\eta(x,t)-\eta(z,t)|^5}[\bar{\partial}^2\partial_k(\rho_0{F^{-1}}_i^k)](z,t)\,dz\\
&+C\int_{\Omega}\dfrac{\partial_t(\bar{\partial}_x-\bar\partial_z)R'(\eta)}{|\eta(x,t)-\eta(z,t)|^3}[\bar{\partial}^2\partial_k(\rho_0{F^{-1}}_i^k)](z,t)\,dz
\\&+C\int_{\Omega}\dfrac{(\bar{\partial}_x-\bar\partial_z)R'(\eta)R(\eta,v)}{|\eta(x,t)-\eta(z,t)|^5}[\bar{\partial}^2\partial_k(\rho_0{F^{-1}}_i^k)](z,t)\,dz
\\&+C\int_{\Omega}\dfrac{R'^2(\eta)R(\eta,v)}{|\eta(x,t)-\eta(z,t)|^7}[\bar{\partial}^2\partial_k(\rho_0{F^{-1}}_i^k)](z,t)\,dz
\\&+C\int_{\Omega}(\bar{\partial}_x-\bar\partial_z)^2\dfrac{1}{|\eta(x,t)-\eta(z,t)|}[\partial_t\bar{\partial}^2\partial_k(\rho_0{F^{-1}}_i^k)](z,t)\,dz.
\end{align*}
With \eqref{eta}, \eqref{taylor}, and $\bar{\partial}^2\partial_k(\rho_0{F^{-1}}_i^k)$ scales like $\rho_0 D^4\eta+D^3\eta+\text{l.o.t}$, $\partial_t\bar{\partial}^2\partial_k(\rho_0{F^{-1}}_i^k)$ scales like $\rho_0 D^4v+D^3v+\text{l.o.t}$, and by a similar argument as we did for the case $a=1$, we have
\begin{align*}
\|g_2\|_0^2\leq CP(\sup_{t\in[0,T]}\tilde{E}(t)),
\end{align*}
and
\begin{align*}
\|\partial_tg_2\|_0^2\leq  \tilde{M}_0+CP(\sup_{t\in[0,T]}\tilde{E}(t)).
\end{align*}
When $a=3$, we have
\begin{align*}
g_3^i=&\int_{\Omega}(\bar{\partial}_x-\bar\partial_z)^3\dfrac{1}{|\eta(x,t)-\eta(z,t)|}[\bar{\partial}\partial_k(\rho_0{F^{-1}}_i^k)](z,t)\,dz\\
=&\underbrace{C\int_{\Omega}\dfrac{(\bar{\partial}_x-\bar\partial_z)^2R'(\eta)}{|\eta(x,t)-\eta(z,t)|^3}[\bar{\partial}\partial_k(\rho_0{F^{-1}}_i^k)](z,t)\,dz}_{B_2}\\&+
C\int_{\Omega}\dfrac{(\bar{\partial}_x-\bar\partial_z)(R'^2(\eta))}{|\eta(x,t)-\eta(z,t)|^5}[\bar{\partial}\partial_k(\rho_0{F^{-1}}_i^k)](z,t)\,dz
\\&+C\int_{\Omega}\dfrac{R'^3(\eta)}{|\eta(x,t)-\eta(z,t)|^7}[\bar{\partial}\partial_k(\rho_0{F^{-1}}_i^k)(z,t)]\,dz,
\end{align*}
First, we estimate $B_2$. Since
\begin{multline*}
(\bar\partial_x-\bar\partial_z)^2R'(\eta)=C(\bar\partial\eta(x,t)-\bar\partial\eta(z,t))\cdot(\bar\partial^2\eta(x,t)-\bar\partial^2\eta(z,t))\\
+C(\eta(x,t)-\eta(z,t))\cdot(\bar\partial^3\eta(x,t)-\bar\partial^3\eta(z,t)),
\end{multline*}
then with \eqref{ini}, \eqref{taylor} and \eqref{eta}, $\|B_2\|_0^2$ can be bounded by
\begin{align}
\nonumber CP(&\|D\eta\|_{L^{\infty}(\Omega)},\|D^2\eta\|_{L^{\infty}(\Omega)})\|\int_{\Omega}\dfrac{1+\bar\partial^3\eta(x,t)-\bar\partial^3\eta(z,t)}{|x-z|^2}|\rho_0 D^3\eta|\,dz\|_0^2\\
&\leq \|\dfrac{1}{|x|^2}\|_{L^1}^2(\|D^3\eta\|_1^2\|D^3\eta\|_{1}^2).
\label{b1}
\end{align}
The $L^2$ norm of other terms can be estimated by
\begin{equation*} CP(\|D\eta\|_{L^{\infty}(\Omega)},\|D^2\eta\|_{L^{\infty}(\Omega)})\|\dfrac{1}{|x|^2}\|_{L^1(\Omega)}^2(\|\rho_0D^3\eta\|_0^2),
\end{equation*}
then we have
\begin{equation*}
\|g_3\|_0^2 \leq  \tilde{M}_0+CP(\sup_{t\in[0,T]}\tilde{E}(t)).
\end{equation*}
With the same methodology we showed above, and
\begin{align*}
\partial_tg_3^i=&C\int_{\Omega}(\bar{\partial}_x-\bar\partial_z)^3\dfrac{1}{|\eta(x,t)-\eta(z,t)|}\bar{\partial}\partial_t\partial_k(\rho_0{F^{-1}}_i^k)(z,t)\,dz\\
&+C\int_{\Omega}\partial_t\dfrac{(\bar{\partial}_x-\bar\partial_z)^2R'(\eta)}{|\eta(x,t)-\eta(z,t)|^3}\bar{\partial}_x\partial_k(\rho_0{F^{-1}}_i^k)(z,t)\,dz\\&+
C\int_{\Omega}\partial_t\dfrac{(\bar{\partial}_x-\bar\partial_z)(R'^2(\eta))}{|\eta(x,t)-\eta(z,t)|^5}\bar{\partial}_x\partial_k(\rho_0{F^{-1}}_i^k)(z,t)\,dz
\\&+C\int_{\Omega}\partial_t\dfrac{R'^3(\eta)}{|\eta(x,t)-\eta(z,t)|^7}\bar{\partial}_x\partial_k(\rho_0{F^{-1}}_i^k)(z,t)\,dz,
\end{align*}
we can bounded $\|\kappa\partial_tg_3\|_0^2$ by
$\tilde{M}_0+CP(\sup_{t\in[0,T]}\tilde{E}(t))$.

\noindent When $a=4$, we have
\begin{align*}
g_4^i=&\int_{\Omega}(\bar{\partial}_x-\bar\partial_z)^4\dfrac{1}{|\eta(x,t)-\eta(z,t)|}\partial_k(\rho_0{F^{-1}}_i^k)(z,t)\,dz\\
=&C\underbrace{\int_{\Omega}\dfrac{(\bar{\partial}_x-\bar\partial_z)^3R'(\eta)}{|\eta(x,t)-\eta(z,t)|^3}\partial_k(\rho_0{F^{-1}}_i^k)(z,t)\,dz}_{B_3}\\&+
C\int_{\Omega}\dfrac{(\bar{\partial}_x-\bar\partial_z)^2(R'^2(\eta))}{|\eta(x,t)-\eta(z,t)|^5}\partial_k(\rho_0{F^{-1}}_i^k)(z,t)\,dz
\\&+C\int_{\Omega}\dfrac{(\bar{\partial}_x-\bar\partial_z)(R'^3(\eta))}{|\eta(x,t)-\eta(z,t)|^7}\partial_k(\rho_0{F^{-1}}_i^k)(z,t)\,dz
\\&+C\int_{\Omega}\dfrac{R'^4(\eta)}{|\eta(x,t)-\eta(z,t)|^9}\partial_k(\rho_0{F^{-1}}_i^k)(z,t)\,dz.
\end{align*}
First, we estimate $B_3$:
\begin{multline*}
B_3=\sum_{b=0}^3\int_{\Omega}(\bar{\partial}_x-\bar\partial_z)^b[\eta(x,t)-\eta(z,t)]\cdot(\bar{\partial}_x-\bar\partial_z)^{3-b}[\bar\partial\eta(x,t)-\bar\partial\eta(z,t)]\\\times\dfrac{1}{|\eta(x,t)-\eta(z,t)|^3}\partial_k(\rho_0{F^{-1}}_i^k)(z,t)\,dz.
\end{multline*}
For $b=0,1,3$, the integral could be estimated by using a similar argument we used for $B_2$ (see \eqref{b1}). When $b=2$, with \eqref{eta} and the Sobolev embedding inequality $|\bar{\partial}^2\eta(x,t)-\bar\partial^2\eta(z,t)|\leq C\|\eta\|_4|x-z|^{1/2}$, we have
\begin{align*}
&|\int_{\Omega}\dfrac{[\bar{\partial}^2\eta(x,t)-\bar\partial^2\eta(z,t)]\cdot(\bar{\partial}^2\eta(x,t)-\bar\partial^2\eta(z,t))}{|\eta(x,t)-\eta(z,t)|^3}\partial_k(\rho_0{F^{-1}}_i^k)(z,t)\,dz|^2
\\\leq& CP(\|\eta\|_4)|\int_{\Omega}\dfrac{1}{|x-z|^{\frac 5 2}}\partial_k(\rho_0{F^{-1}}_i^k)(z,t)\,dz|^2\\
\leq &CP(\|\eta\|_4)\|\dfrac{1}{|x|^{5/2}}\|_{L^1(\Omega)}^{2}\|D^2\eta\|_{0}^2.
\end{align*}
Thus, $\|B_3\|_0^2\leq \tilde{M}_0+CP(\sup_{t\in[0,T]}\tilde{E}(t))$.

\noindent The other terms can be estimated by applying the same methodology we showed above, and finally we have
\begin{equation*}
\|g_4\|_0^2\leq \tilde{M}_0+CP(\sup_{t\in[0,T]}\tilde{E}(t)),
\end{equation*}
For $\partial_tg_4$, we have
\begin{align*}
\partial_tg_4^i=&\int_{\Omega}\partial_t(\bar{\partial}_x-\bar\partial_z)^4\dfrac{1}{|\eta(x,t)-\eta(z,t)|}\partial_k(\rho_0{F^{-1}}_i^k)(z,t)\,dz
\\&+\int_{\Omega}(\bar{\partial}_x-\bar\partial_z)^4\dfrac{1}{|\eta(x,t)-\eta(z,t)|}\partial_t\partial_k(\rho_0{F^{-1}}_i^k)(z,t)\,dz
\end{align*}
and similarly, we can bounded $\|\partial_tg_4\|_0^2$ by $\tilde{M}_0+CP(\sup_{t\in[0,T]}\tilde{E}(t))$.

\noindent Finally, we have get that
\begin{equation*}
\|\sqrt{\rho_0}\bar{\partial}^4G\|_0^2+\int_0^t\kappa\|\sqrt{\rho_0}\bar{\partial}^4\partial_tG\|_0^2\,dt' \leq \tilde{M}_0+CP(\sup_{t\in[0,T]}\tilde{E}(t)).
\end{equation*}
\textbf{Step 2. Estimates for $\|\sqrt{\rho_0}\partial_t^8G\|_0^2+\int_0^t\kappa\|\sqrt{\rho_0}\partial_t^9G(t',\cdot)\|_0^2\,dt'$.}

\noindent By the formula \eqref{forceF}, we have
\begin{align*}
&\int_{\Omega}\kappa|\sqrt{\rho_0}\partial_t^9G^i|^2\,dx
\\=&C\int_{\Omega}\kappa|\sqrt{\rho_0}\int_{\Omega}\partial_t^9\dfrac{[\partial_k(\rho_0{F^{-1}}_i^k)](z,t)\,dz}{|\eta(x,t)-\eta(z,t)|}|^2\,dx
\\\leq &C \sum_{p=0}^9\int_{\Omega}|\sqrt{\rho_0}\int_{\Omega}\partial_t^p[\partial_k(\rho_0{F^{-1}}_i^k)](z,t)\partial_t^{9-p}\dfrac{1}{|\eta(x,t)-\eta(z,t)|}\,dz|^2\,dx,
\end{align*}
We denote
\begin{equation*}g_p^i=\int_{\Omega}\partial_t^p\partial_k(\rho_0{F^{-1}}_i^k)\partial_t^{9-p}\dfrac{1}{|\eta(x,t)-\eta(z,t)|}\,dz,
\end{equation*}
then we have $\kappa\|\sqrt{\rho_0}\partial_t^9G^i\|_0^2\leq C\sum_{p=0}^9\kappa\|\sqrt{\rho_0}g_p^i\|_0^2$.
We also denote $R(\eta,v)=(\eta(x,t)-\eta(z,t))\cdot ((v(x,t)-v(z,t))$ for notational convenience.

\noindent When $p=9$, integrating by part with respect of $z_k$, we can get
\begin{align*}
&\quad|\int_{\Omega}\dfrac{1}{|\eta(x,t)-\eta(z,t)|}\partial_k\partial_t^9(\rho_0{F^{-1}}_i^k)\,dz|\\
&=|\int_{\Omega}\dfrac{(\eta(x,t)-\eta(z,t))\cdot\partial_k\eta(z,t)}{|\eta(x,t)-\eta(z,t)|^3}\rho_0\partial_t^9({F^{-1}}_i^k)\,dz|
\\&\leq C\int_{\Omega}\dfrac{1}{|x-z|^2}|\partial_k\eta(z,t)\rho_0\partial_t^9({F^{-1}}_i^k)|\,dz,
\end{align*}
where we used \eqref{eta} and Taylor's formula \eqref{taylor}.
$\partial_t^9{F^{-1}}_i^k$ scales like $l[\partial_t^9D\eta+\partial_t^6D\eta\partial_t^3D\eta+\partial_t^5D\eta\partial_t^4D\eta+\text{l.o.t}]$ where $l$ denote $L^{\infty}$ function. Then with \eqref{ini2}, we have
\begin{align*}\|g_9^i\|_0^2\leq & C\|\dfrac{1}{|x|^2}\|_{L^1(\Omega)}^2\|\rho_0(\partial_t^9D\eta+\partial_t^6D\eta\partial_t^3D\eta+\partial_t^5D\eta\partial_t^4D\eta)\|_0^2\\
\leq &C(\|\sqrt{\rho_0}\partial_t^8Dv\|_0^2+\|\partial_t^6D\eta\|_0^2\|\rho_0\partial_t^3D\eta\|_{L^{\infty}(\Omega)}^2+\|\partial_t^4D\eta\|_1^2\|\rho_0\partial_t^5D\eta\|_{1}^2),
\end{align*}
and thus $\int_0^t\kappa\|g_9^i\|_0^2\,dt'\leq \tilde M_0+CP(\sup_{t\in[0,T]}\tilde{E}(t))$.

\noindent For the case $p=8$, similarly, after integrating by part with respect to $z_k$, and using \eqref{eta} and \eqref{taylor},
\begin{align*}
&\int_{\Omega}\partial_k\dfrac{R(\eta,v)}{|\eta(x,t)-\eta(z,t)|^3}\rho_0\partial_t^{8}({F^{-1}}_i^k)\,dz
\\=&\int_{\Omega}\dfrac{\partial_k\eta(z,t)\cdot ((v(x,t)-v(z,t))}{|\eta(x,t)-\eta(z,t)|^3}\rho_0\partial_t^{8}({F^{-1}}_i^k)\,dz\\&+\int_{\Omega}\dfrac{(\eta(x,t)-\eta(z,t))\cdot \partial_k v(z,t)}{|\eta(x,t)-\eta(z,t)|^3}\rho_0\partial_t^{8}({F^{-1}}_i^k)\,dz
\\&+C\int_{\Omega}\dfrac{R(\eta,v)(\eta(x,t)-\eta(z,t))\cdot\partial_k\eta(z,t)}{|\eta(x,t)-\eta(z,t)|^5}\rho_0\partial_t^{8}({F^{-1}}_i^k)\,dz
\\\leq &CP(\|Dv\|_{L^{\infty}(\Omega)}+\|D\eta\|_{L^{\infty}(\Omega)})\int_{\Omega}\dfrac{1}{|x-z|^2}\rho_0\partial_t^8({F^{-1}}_i^k)\,dz.
\end{align*}
$\rho_0\partial_t^8({F^{-1}}_i^k)$ scales like $l[\partial_t^8D\eta+\partial_t^5D\eta\partial_t^3D\eta+\partial_t^4D\eta\partial_t^4D\eta+\text{l.o.t}]$, thus with \eqref{ini2}, we have
\begin{align*}\|g_8^i\|_0^2\leq & C\|\dfrac{1}{|x|^2}\|_{L^1(\Omega)}^2\|\rho_0(\partial_t^8D\eta+\partial_t^5D\eta\partial_t^3D\eta+\partial_t^4D\eta\partial_t^4D\eta\|_0^2\\
\leq &C(\|\rho_0\partial_t^8D\eta\|_0^2+\|\partial_t^5D\eta\|_0^2\|\rho_0\partial_t^3D\eta\|_{L^{\infty}(\Omega)}^2+\|\partial_t^4D\eta\|_1^2\|\rho_0\partial_t^5D\eta\|_{1}^2)
\\\leq & CP(\sup_{t\in[0,T]}\tilde{E}(t)).
\end{align*}
For the case $p=7$, we used same methodology as $p=8$ to get that
\begin{align*}
|g^i_7|&\leq
|\int_{\Omega}\partial_k\dfrac{\partial_tR(\eta,v)}{|\eta(x,t)-\eta(z,t)|^3}\rho_0\partial_t^{7}({F^{-1}}_i^k)\,dz|
\\&\quad+|\int_{\Omega}\partial_k\dfrac{R^2(\eta,v)}{|\eta(x,t)-\eta(z,t)|^5}\rho_0\partial_t^{7}({F^{-1}}_i^k)\,dz|\\&\leq CP(\|Dv\|^2_{L^{\infty}(\Omega)},\|D^2\eta\|_{L^{\infty}(\Omega)},\|\partial_t^2\eta\|_{L^{\infty}(\Omega)},\|D\partial_t^2\eta\|_{L^{\infty}(\Omega)})
\\&\quad\quad\times|\int_{\Omega}\dfrac{1}{|x-z|^2}\rho_0\partial_t^{7}({F^{-1}}_i^k)\,dz|
\\&+CP(\|Dv\|^2_{L^{\infty}(\Omega)},\|D\eta\|_{L^{\infty}(\Omega)},\|v\|_{L^{\infty}(\Omega)})
|\int_{\Omega}\dfrac{1}{|x-z|^2}\rho_0\partial_t^{7}({F^{-1}}_i^k)\,dz|,\end{align*}
and then we can obtain
\begin{align*}
\|g^i_7\|_0^2&\leq C\|\dfrac{1}{|x|^2}\|_{L^1(\Omega)}^2\|\rho_0\partial_t^7({F^{-1}}_i^k)\|_0^2\leq C(\|\rho_0\partial_t^7D\eta\|_0^2+\|\partial_t^3D\eta\|_1^2\|\partial_t^4D\eta\|_1^2)
\\&\leq CP(\sup_{t\in[0,T]}\tilde{E}(t)).
\end{align*}
For the case $p=6$, we have
\begin{align*}
|g_6^i|&\leq|\int_{\Omega}\partial_k\dfrac{\partial_t^2R(\eta,v)}{|\eta(x,t)-\eta(z,t)|^3}(\rho_0\partial_t^{6}({F^{-1}}_i^k))\,dz|
\\&\quad +|\int_{\Omega}\partial_k\dfrac{\partial_t(R^2(\eta,v))}{|\eta(x,t)-\eta(z,t)|^5}(\rho_0\partial_t^{6}({F^{-1}}_i^k))\,dz|
\\&\quad +|\int_{\Omega}\partial_k\dfrac{R^3(\eta,v)}{|\eta(x,t)-\eta(z,t)|^7}(\rho_0\partial_t^{6}({F^{-1}}_i^k))\,dz|\\
&\leq CP(\|Dv\|_{L^{\infty}(\Omega)},\|\partial_t^2\eta\|_{L^{\infty}(\Omega)}, \|\eta\|_{2})|\partial_k\partial_t^2v||\int_{\Omega}\dfrac{1}{|x-z|^2}(\rho_0\partial_t^{6}({F^{-1}}_i^k))\,dz|\\
&\quad+CP(\|Dv\|_{L^{\infty}(\Omega)},\|D\eta\|_{L^{\infty}(\Omega})\dfrac{1}{|x-z|^2}(\rho_0\partial_t^{6}({F^{-1}}_i^k))\,dz|.
\end{align*}
With H\"older's inequality, Young's inequality for convolution and the Sobolev embedding inequality, we can obtain that
\begin{align*}
\|g_6^i\|_0^2&\leq CP(\|Dv\|_{L^{\infty}(\Omega)},\|\partial_t^2\eta\|_{L^{\infty}(\Omega)}, \|\eta\|_{2})\|\dfrac{1}{|x|^2}\|_{L^1(\Omega)}^2\\&\quad\quad\times\|D\partial_t v\|_1^2(\|\rho_0\partial_t^6D\eta\|_{H^{1/2}}^2+\|\partial_t^3D\eta\|_1^4)\\
&\leq \tilde M_0+ CP(\sup_{t\in[0,T]}\tilde E(t)),
\end{align*}
where we also used \eqref{ini2}.

\noindent
For the case $p=5$,
\begin{align*}
|g_5^i|\leq &C|\int_{\Omega}\dfrac{\partial_t^3R(\eta,v)}{|\eta(x,t)-\eta(z,t)|^3}\partial_k(\rho_0\partial_t^{5}({F^{-1}}_i^k))\,dz|
\\&+C|\int_{\Omega}\dfrac{\partial_t^2(R^2(\eta,v))}{|\eta(x,t)-\eta(z,t)|^5}\partial_k(\rho_0\partial_t^{5}({F^{-1}}_i^k))\,dz|
\\&+C|\int_{\Omega}\dfrac{\partial_t(R^3(\eta,v))}{|\eta(x,t)-\eta(z,t)|^7}\partial_k(\rho_0\partial_t^{5}({F^{-1}}_i^k))\,dz|
\\&+C|\int_{\Omega}\dfrac{R^4(\eta,v)}{|\eta(x,t)-\eta(z,t)|^9}\partial_k(\rho_0\partial_t^{5}({F^{-1}}_i^k))\,dz|\\
\leq & CP(\|\eta\|_{4},\|v\|_{3},\|\partial_t^2\eta\|_{3},\|\partial_t^4\eta\|_2)\int_{\Omega}\dfrac{1}{|x-z|^2}\partial_k(\rho_0\partial_t^{5}({F^{-1}}_i^k))\,dz.
\end{align*}
Then with \eqref{ini2}, Young's inequality for convolution, we have
\begin{align*}
\|g_5^i\|_0^2&\leq CP(\|\eta\|_{4},\|v\|_{3},\|\partial_t^2\eta\|_{3},\|\partial_t^4\eta\|_2)\|\dfrac{1}{\|x|^2}\|_{L^1(\Omega)}^2(\|\rho_0D^2\partial_t^5\eta\|_0^2+\|\partial_t^5D\eta\|_0^2
\\&\quad+\|\partial_t^4D\eta\|_1^2\|\partial_tD^2\eta\|_1^2+\|\partial_t^3D\eta\|_1^2\|\partial_t^2D^2\eta\|_1^2)
\\&\leq \tilde M_0+ CP(\sup_{t\in[0,T]}\tilde E(t)).
\end{align*}
For the case $p=4$,
\begin{align*}
|g_4^i|\leq
&C|\int_{\Omega}\dfrac{\partial_t^4[R(\eta,v)]}{|\eta(x,t)-\eta(z,t)|^3}\partial_k(\rho_0\partial_t^{4}({F^{-1}}_i^k))\,dz|
\\&+C|\int_{\Omega}\dfrac{\partial_t^3(R^2(\eta,v))}{|\eta(x,t)-\eta(z,t)|^5}\partial_k(\rho_0\partial_t^{4}({F^{-1}}_i^k))\,dz|
\\&+C|\int_{\Omega}\dfrac{\partial_t^2(R^3(\eta,v))}{|\eta(x,t)-\eta(z,t)|^7}\partial_k(\rho_0\partial_t^{4}({F^{-1}}_i^k))\,dz|
\\&+C|\int_{\Omega}\dfrac{\partial_t(R^4(\eta,v))}{|\eta(x,t)-\eta(z,t)|^9}\partial_k(\rho_0\partial_t^{4}({F^{-1}}_i^k))\,dz|
\\&+C|\int_{\Omega}\dfrac{R^5(\eta,v)}{|\eta(x,t)-\eta(z,t)|^{11}}\partial_k(\rho_0\partial_t^{4}({F^{-1}}_i^k))\,dz|\\
\leq&CP(\|v\|_2,\|\eta\|_2,\|\partial_t^2\eta\|_2)\int_{\Omega}\dfrac{1+|\partial_t^4v(x,t)-\partial_t^4v(z,t)|}{|x-z|^2}\partial_k(\rho_0\partial_t^{4}({F^{-1}}_i^k))\,dz,\\
\end{align*}
and by \eqref{ini2}, H\"older's inequality, the Sobolev embedding inequalities and Young's inequality, we have
\begin{align*}
\|g_4^i\|_0^2&\leq CP(\|v\|_2,\|\eta\|_2,\|\partial_t^2\eta\|_2)\|\dfrac{1}{|x|^2}\|_{L^1(\Omega)}^2\|\partial_t^4v\|_1^2\|\partial_k(\rho_0\partial_t^{4}({F^{-1}}_i^k))\|_1^2\\
&\leq CP(\|v\|_2,\|\eta\|_2,\|\partial_t^2\eta\|_2)\|\partial_t^4v\|_1^2(\|\rho_0\partial_t^4D^2\eta\|_1^2+\|\partial_t^4D\eta\|_1^2\\&\quad\quad+\|\rho_0\partial_t^3D\eta\|_2^2\|\partial_tD^2\eta\|_1^2)\\
&\leq \tilde M_0+ CP(\sup_{t\in[0,T]}\tilde E(t)).
\end{align*}
For the case $p=3$
\begin{align*}
|g_3^i|\leq
&C|\int_{\Omega}\dfrac{\partial_t^5[R(\eta,v)]}{|\eta(x,t)-\eta(z,t)|^3}\partial_k(\rho_0\partial_t^{3}({F^{-1}}_i^k))\,dz|
\\&+C|\int_{\Omega}\dfrac{\partial_t^4[R^2(\eta,v)]}{|\eta(x,t)-\eta(z,t)|^5}\partial_k(\rho_0\partial_t^{3}({F^{-1}}_i^k))\,dz|
\\&+C|\int_{\Omega}\dfrac{\partial_t^3[R^3(\eta,v)]}{|\eta(x,t)-\eta(z,t)|^7}\partial_k(\rho_0\partial_t^{3}({F^{-1}}_i^k))\,dz|
\\&+C|\int_{\Omega}\dfrac{\partial_t^2[R^4(\eta,v)]}{|\eta(x,t)-\eta(z,t)|^9}\partial_k(\rho_0\partial_t^{3}({F^{-1}}_i^k))\,dz|
\\&+C|\int_{\Omega}\dfrac{\partial_t[R^5(\eta,v)]}{|\eta(x,t)-\eta(z,t)|^{11}}\partial_k(\rho_0\partial_t^{3}({F^{-1}}_i^k))\,dz|
\\&+C|\int_{\Omega}\dfrac{[R^6(\eta,v)]}{|\eta(x,t)-\eta(z,t)|^{13}}\partial_k(\rho_0\partial_t^{3}({F^{-1}}_i^k))\,dz|
\\\leq &CP(\|v\|_3,\|\eta\|_4,\|\partial_t^3\eta\|_2)(\underbrace{\int_{\Omega}\dfrac{|\partial_t^3\eta(x,t)-\partial_t^3\eta(z,t)|}{|x-z|^3}\partial_k(\rho_0\partial_t^{3}({F^{-1}}_i^k))\,dz}_{C_1}\\
&\quad+\underbrace{\int_{\Omega}\dfrac{1+|\partial_t^5v(x,t)-\partial_t^5v(z,t)|}{|x-z|^2}\partial_k(\rho_0\partial_t^{3}({F^{-1}}_i^k))\,dz}_{C_2}).
\end{align*}
Then we use $|\partial_t^3\eta(x,t)-\partial_t^3\eta(x,t)|\leq C\|\partial_t^3\eta\|_2|x-z|^{1/2}$ to estimate $C_1$ and $L^3$-$L^6$ H\"older's inequality to estimate $C_2$, and with \eqref{ini2} to get
\begin{align*}
\|g_3^i\|_0^2\leq \tilde M_0+ CP(\sup_{t\in[0,T]}\tilde E(t)).
\end{align*}
For the case $p=2$,
\begin{align*}
|g_2^i|\leq
&C|\int_{\Omega}\dfrac{\partial_t^6[R(\eta,v)]}{|\eta(x,t)-\eta(z,t)|^3}\partial_k(\rho_0\partial_t^2({F^{-1}}_i^k))\,dz|
\\&+C|\int_{\Omega}\dfrac{\partial_t^5[R^2(\eta,v)]}{|\eta(x,t)-\eta(z,t)|^5}\partial_k(\rho_0\partial_t^2({F^{-1}}_i^k))\,dz|
\\&+C|\int_{\Omega}\dfrac{\partial_t^4[R^3(\eta,v)]}{|\eta(x,t)-\eta(z,t)|^7}\partial_k(\rho_0\partial_t^2({F^{-1}}_i^k))\,dz|
\\&+C|\int_{\Omega}\dfrac{\partial_t^3[R^4(\eta,v)]}{|\eta(x,t)-\eta(z,t)|^9}\partial_k(\rho_0\partial_t^2({F^{-1}}_i^k))\,dz|
\\&+C|\int_{\Omega}\dfrac{\partial_t^2[R^5(\eta,v)]}{|\eta(x,t)-\eta(z,t)|^{11}}\partial_k(\rho_0\partial_t^2({F^{-1}}_i^k))\,dz|
\\&+C|\int_{\Omega}\dfrac{\partial_t[R^6(\eta,v)]}{|\eta(x,t)-\eta(z,t)|^{13}}\partial_k(\rho_0\partial_t^2({F^{-1}}_i^k))\,dz|
\\&+C|\int_{\Omega}\dfrac{[R^7(\eta,v)]}{|\eta(x,t)-\eta(z,t)|^{15}}\partial_k(\rho_0\partial_t^2({F^{-1}}_i^k))\,dz|
\end{align*}
\begin{align*}
\leq &CP(\|v\|_3,\|\eta\|_4,\|\partial_t^3\eta\|_2)(\int_{\Omega}\dfrac{|\partial_t^4\eta(x,t)-\partial_t^4\eta(z,t)|}{|x-z|^3}\partial_k(\rho_0\partial_t^{2}({F^{-1}}_i^k))\,dz\\
&\quad+\int_{\Omega}\dfrac{1+|\partial_t^6v(x,t)-\partial_t^6v(z,t)|}{|x-z|^2}\partial_k(\rho_0\partial_t^{2}({F^{-1}}_i^k))\,dz).
\end{align*}
Then by a similar argument we used for $p=3$, we have
\begin{align*}
\|g_2^i\|_0^2\leq \tilde M_0+ CP(\sup_{t\in[0,T]}\tilde E(t)).
\end{align*}
For the case $p=1$, we have
\begin{align*}
|g_1^i|\leq
&C|\int_{\Omega}\dfrac{\partial_t^7[R(\eta,v)]}{|\eta(x,t)-\eta(z,t)|^3}\partial_k(\rho_0\partial_t({F^{-1}}_i^k))\,dz|
\\&+C|\int_{\Omega}\dfrac{\partial_t^6[R^2(\eta,v)]}{|\eta(x,t)-\eta(z,t)|^5}\partial_k(\rho_0\partial_t({F^{-1}}_i^k))\,dz|
\\&+C|\int_{\Omega}\dfrac{\partial_t^5[R^3(\eta,v)]}{|\eta(x,t)-\eta(z,t)|^7}\partial_k(\rho_0\partial_t({F^{-1}}_i^k))\,dz|
\\&+C|\int_{\Omega}\dfrac{\partial_t^4[R^4(\eta,v)]}{|\eta(x,t)-\eta(z,t)|^9}\partial_k(\rho_0\partial_t({F^{-1}}_i^k))\,dz|
\\&+C|\int_{\Omega}\dfrac{\partial_t^3[R^5(\eta,v)]}{|\eta(x,t)-\eta(z,t)|^{11}}\partial_k(\rho_0\partial_t({F^{-1}}_i^k))\,dz|
\\&+C|\int_{\Omega}\dfrac{\partial_t^2[R^6(\eta,v)]}{|\eta(x,t)-\eta(z,t)|^{13}}\partial_k(\rho_0\partial_t({F^{-1}}_i^k))\,dz|
\\&+C|\int_{\Omega}\dfrac{\partial_t[R^7(\eta,v)]}{|\eta(x,t)-\eta(z,t)|^{15}}\partial_k(\rho_0\partial_t({F^{-1}}_i^k))\,dz|
\\&+C|\int_{\Omega}\dfrac{[R^8(\eta,v)]}{|\eta(x,t)-\eta(z,t)|^{17}}\partial_k(\rho_0\partial_t({F^{-1}}_i^k))\,dz|
\\\leq&CP(\|v\|_3,\|\eta\|_4,\|\partial_t^4\eta\|_2)(\int_{\Omega}\dfrac{|\partial_t^4\eta(x,t)-\partial_t^4\eta(z,t)|}{|x-z|^3}\partial_k(\rho_0\partial_t({F^{-1}}_i^k))\,dz\\
&\quad+\underbrace{\int_{\Omega}\dfrac{|\partial_t^3\eta(x,t)-\partial_t^3\eta(z,t)|}{|x-z|^3}|\partial_t^4v(x,t)-\partial_t^4v(z,t)|\partial_k(\rho_0\partial_t({F^{-1}}_i^k))\,dz}_{C_3}
\\&\quad+\int_{\Omega}\dfrac{1+|\partial_t^7v(x,t)-\partial_t^7v(z,t)|}{|x-z|^2}\partial_k(\rho_0\partial_t({F^{-1}}_i^k))\,dz).
\end{align*}
We use the Sobolev embedding inequality, H\"older's inequality and Young's inequality to estimate $C_3$ as
\begin{align*}
C_3\leq &C\|\partial_t^3\eta\|_2|\partial_t^4v||\int_{\Omega}\dfrac{1}{|x-z|^{5/2}}\partial_k(\rho_0\partial_t({F^{-1}}_i^k))\,dz|\\
&+C\|\partial_t^3\eta\|_2|\int_{\Omega}\dfrac{1}{|x-z|^{5/2}}|\partial_t^4v|\partial_k(\rho_0\partial_t({F^{-1}}_i^k))\,dz|
\end{align*}
and
\begin{align*}
\|C_3\|_0^2\leq C\|\partial_t^3\eta\|_2^2\|\partial_t^4v\|_{L^6(\Omega)}^2\|\dfrac{1}{|x-z|^{5/2}}\|_{L^1(\Omega)}^2\|\partial_k(\rho_0\partial_t({F^{-1}}_i^k)\|_{L^3(\Omega)}^2.
\end{align*}
Thus, we can obtain that
\begin{align*}
\|g_1^i\|_0^2\leq \tilde M_0+ CP(\sup_{t\in[0,T]}\tilde E(t)).
\end{align*}
For the case $p=0$,
\begin{align*}
|g_0^i|\leq
&C|\int_{\Omega}\dfrac{\partial_t^8[R(\eta,v)]}{|\eta(x,t)-\eta(z,t)|^3}\partial_k(\rho_0({F^{-1}}_i^k))\,dz|
\\&+C|\int_{\Omega}\dfrac{\partial_t^7[R^2(\eta,v)]}{|\eta(x,t)-\eta(z,t)|^5}\partial_k(\rho_0({F^{-1}}_i^k))\,dz|
\\&+C|\int_{\Omega}\dfrac{\partial_t^6[R^3(\eta,v)]}{|\eta(x,t)-\eta(z,t)|^7}\partial_k(\rho_0({F^{-1}}_i^k))\,dz|
\\&+C|\int_{\Omega}\dfrac{\partial_t^5[R^4(\eta,v)]}{|\eta(x,t)-\eta(z,t)|^9}\partial_k(\rho_0({F^{-1}}_i^k))\,dz|
\\&+C|\int_{\Omega}\dfrac{\partial_t^4[R^5(\eta,v)]}{|\eta(x,t)-\eta(z,t)|^{11}}\partial_k(\rho_0({F^{-1}}_i^k))\,dz|
\\&+C|\int_{\Omega}\dfrac{\partial_t^3[R^6(\eta,v)]}{|\eta(x,t)-\eta(z,t)|^{13}}\partial_k(\rho_0({F^{-1}}_i^k))\,dz|
\\&+C|\int_{\Omega}\dfrac{\partial_t^2[R^7(\eta,v)]}{|\eta(x,t)-\eta(z,t)|^{15}}\partial_k(\rho_0({F^{-1}}_i^k))\,dz|
\\&+C|\int_{\Omega}\dfrac{\partial_t[R^8(\eta,v)]}{|\eta(x,t)-\eta(z,t)|^{17}}\partial_k(\rho_0({F^{-1}}_i^k))\,dz|
\\&+C|\int_{\Omega}\dfrac{[R^9(\eta,v)]}{|\eta(x,t)-\eta(z,t)|^{19}}\partial_k(\rho_0({F^{-1}}_i^k))\,dz|
\\\leq&CP(\|v\|_3,\|\eta\|_4,\|\partial_t^4\eta\|_2)(\int_{\Omega}\dfrac{|\partial_t^4\eta(x,t)-\partial_t^4\eta(z,t)|}{|x-z|^3}\partial_k(\rho_0({F^{-1}}_i^k))\,dz\\
&\quad+\int_{\Omega}\dfrac{|\partial_t^4\eta(x,t)-\partial_t^4\eta(z,t)|}{|x-z|^3}|\partial_t^5\eta(x,t)-\partial_t^5\eta(z,t)|\partial_k(\rho_0({F^{-1}}_i^k))\,dz
\\&\quad+\int_{\Omega}\dfrac{|\partial_t^3\eta(x,t)-\partial_t^3\eta(z,t)|}{|x-z|^3}|\partial_t^6\eta(x,t)-\partial_t^6\eta(z,t)|\partial_k(\rho_0({F^{-1}}_i^k))\,dz
\\&\quad+\int_{\Omega}\dfrac{1+|\partial_t^8v(x,t)-\partial_t^8v(z,t)|}{|x-z|^2}\partial_k(\rho_0({F^{-1}}_i^k))\,dz).
\end{align*}
By the weighted Sobolev embedding inequality \eqref{embd}, we can bound $\kappa\|\partial_t^8v\|_0^2$ by
\begin{equation*}
\kappa\|\rho_0\partial_t^8v\|^2_0+\kappa\|\rho_0\partial_t^8Dv\|_0^2.
\end{equation*}
Then combining with a similar argument we used to get the estimates for $C_2, C_3$, we obtain that
\begin{equation*}
\int_0^t\kappa\|\sqrt{\rho_0}\partial_t^9G\|_0^2\,dt' \leq \tilde M_0 +CP(\sup_{t\in[0,T]}\tilde E(t)).
\end{equation*}
Similarly, we also have $\|\sqrt{\rho_0}\partial_t^8G\|_0^2\leq \tilde M_0 +CP(\sup_{t\in[0,T]}\tilde E(t))$.

\noindent\textbf{Step 3. Estimates for $\sum_{a=1}^3[\|\sqrt{\rho_0}\bar{\partial}^{4-a}\partial_t^{2a}G(t)\|_0^2+\int_0^t\kappa\|\sqrt{\rho_0}\bar{\partial}^{4-a}\partial_t^{2a+1}G(t',\cdot)\|_0^2\,dt']$.}

\noindent Since we have provided detailed proofs of the energy estimates for the two end-point cases, the $\bar{\partial}^4$-estimates and the $\partial_t^8$-estimates, we have covered all of the estimation strategies for the rest three remaining intermediate estimates. And we can finally get that
\begin{equation*}
\sum_{a=1}^3[\|\sqrt{\rho_0}\bar{\partial}^{4-a}\partial_t^{2a}G(t)\|_0^2+\int_0^t\kappa\|\sqrt{\rho_0}\bar{\partial}^{4-a}\partial_t^{2a+1}G\|_0^2\,dt'] \leq \tilde{M}_0+CP(\sup_{t\in[0,T]}\tilde{E}(t)).
\end{equation*}
\end{proof}
After we getting the estimates for $G$, the remainder of a prior estimates in this section are quite similar to \cite[Section 9]{DS_2010}. For a self-contained presentation, we still carry out the proof. We also combine some ideas from \cite{MJ_2010} to make this proof a little simpler.
\subsection{Curl Estimates}
\begin{proposition}
\label{propc}
For all $t \in(0,T)$, where we take $T \in(0,T_{\kappa}),$
\begin{align}
\nonumber\sum_{a=0}^3\|& \curl\partial_t^{2a}\eta(t)\|_{3-a}^2+\sum_{l=0}^4\|\rho_0\bar{\partial}^{4-l}\curl\partial_t^{2l}\eta(t)\|_0^2\\
+&\sum_{l=0}^4\int_0^t\|\sqrt{\kappa}\rho_0 \curl_{\eta}\bar{\partial}^{4-l}\partial_t^{2l}v(s)\|_0^2 \leq \tilde{M}_0+CTP(\sup_{t\in[0,T]}\tilde{E}(t)).
\label{curlest}
\end{align}
\end{proposition}
\begin{proof}
Letting $\curl_{\eta}$ act on \eqref{lagappro}, we can obtain the identity
\begin{equation}
\label{c0}
(\curl_{\eta} v_t)^k = -\kappa \epsilon_{kji}v^r_{,s}{F^{-1}}_j^s[(2\rho_0J^{-1}-\Phi)_{,l}{F^{-1}}_r^l]_{,m}{F^{-1}}_i^m,
\end{equation}
and then
\begin{equation}
\label{c1}
\partial_t(\curl_{\eta} v)^k =\epsilon_{kji}{F^{-1}}_{t\,j}^sv^i_{,s} -\kappa \epsilon_{kji}v^r_{,s}{F^{-1}}_j^s[(2\rho_0J^{-1}-\Phi)_{,l}{F^{-1}}_r^l]_{,m}{F^{-1}}_i^m.
\end{equation}
Defining the $k$-th component of the vector field $B(F^{-1},Dv)$ by
\begin{equation*}
B^{k}(F^{-1},Dv)=-\epsilon_{kji}{F^{-1}}_r^sv^r_{,l}{F^{-1}}^l_jv^i_{,s},
\end{equation*}
and the $k$-th component of the vector field $Q$ by
\begin{equation*}
Q^{k}(F^{-1},Dv)=-\kappa \epsilon_{kji}v^r_{,s}{F^{-1}}_j^s[(2\rho_0J^{-1}-\Phi)_{,l}{F^{-1}}_r^l]_{,m}{F^{-1}}_i^m.
\end{equation*}
Applying the fundamental theorem of calculus, we can get
\begin{equation}
\label{c2}
\curl_{\eta}v(t)=\curl u_0+\int_0^t [B(F^{-1}(\tau),Dv(\tau))+Q(\tau)]\,d\tau
\end{equation}
Acting $D$ on \eqref{c2} and using the fundamental theorem of calculus again, we finally obtain that
\begin{align}
\label{c3}
\nonumber D\curl\eta(t)=&tD\curl u_0-\epsilon_{\cdot ji}\int_0^t{F^{-1}}_{t\,j}^s(\tau)\,d\tau D\eta^i_{,s}\\
\nonumber &+ \epsilon_{\cdot ji}\int_0^t[{F^{-1}}_{t\,j}^sD\eta^i_{,s}-D{F^{-1}}^s_j v^i_{,s}]\,d\tau\\
&+\int_0^t\int_0^{\tau}[DB({F}^{-1}(t'),Dv(t'))+DQ(t')]\,dt'\,d\tau.
\end{align}
\textbf{Step 1. Estimate for $\curl \eta$.} Acting $D^2$ on \eqref{c3}, with $\partial_t {F^{-1}}^s_j = -{F^{-1}}^s_lv^l_{,p}{F^{-1}}^p_j$ and $D{F^{-1}}^s_j=-{F^{-1}}^s_lD\eta^l_{,p}{F^{-1}}^p_j$, we see that terms arising from the action of $D^2$ on the first three terms on the right-hand side of \eqref{c3} are bounded by $\tilde{M}_0+CTP(\sup_{t\in[0,T]}\tilde{E}(t))$. Since
\begin{equation*}
DB^k(F^{-1},Dv)=-\epsilon_{kji}[Dv^i_{,s}{F^{-1}}^s_lv^l_{,p}{F^{-1}}^p_j+v^i_{,s}{F^{-1}}^s_lDv^l_{,p}{F^{-1}}_j^p+v^i_{,s}v^l_{,p}D({F^{-1}}^s_l{F^{-1}}^p_j)],
\end{equation*}
the highest-order term arising from the action of $D^2$ on $DB(F^{-1},Dv)$ is written as
\begin{equation*}
-\epsilon_{kji}\int_0^t\int_0^{\tau}[D^3v^i_{,s}{F^{-1}}^s_lv^l_{,p}{F^{-1}}^p_j+v^i_{,s}{F^{-1}}^s_lD^3v^l_{,p}{F^{-1}}_j^p]\,dt'\,d\tau.
\end{equation*}
Both summands in the integral scale like $D^4vDv F^{-1}F^{-1}$. Integrating by parts in time,
\begin{multline*}
\int_0^t\int_0^{\tau}D^4v Dv F^{-1}F^{-1}\,dt'\,d\tau = - \int_0^t\int_0^{\tau}D^4\eta \partial_t(Dv F^{-1}F^{-1})\,dt'\,d\tau\\+\int_0^tD^4\eta Dv F^{-1}F^{-1}\,dt',
\end{multline*}
from which it follows that
\begin{equation*}
\|\int_0^t\int_0^{\tau}D^3B(F^{-1}(t'),Dv(t'))\,dt'\,d\tau\|_0^2\leq CTP(\sup_{t\in[0,T]}\tilde{E}(t)).
\end{equation*}
We next estimate the term associated with $Q$. Since
\begin{align*}
DQ^k=&-\kappa \epsilon_{kji}[Dv^r_{,s}{F^{-1}}_j^s[(2\rho_0J^{-1}-\Phi)_{,l}{F^{-1}}_r^l]_{,m}{F^{-1}}_i^m\\&+v^r_{,s}[(2\rho_0J^{-1}-\Phi)_{,l}{F^{-1}}_r^l]_{,m}D({F^{-1}}_j^s{F^{-1}}_i^m)\\
&+v^r_{,s}{F^{-1}}_j^sD[(2\rho_0J^{-1}-\Phi)_{,l}{F^{-1}}_r^l]_{,m}{F^{-1}}_i^m],
\end{align*}
then the highest-order term arising from the action of $D^2$ on $DQ$ is written as
\begin{multline*}
\kappa\epsilon_{kji}\int_0^t\int_0^{\tau}[D^3v^r_{,s}{F^{-1}}_j^s[(2\rho_0J^{-1}-\Phi)_{,l}{F^{-1}}_r^l]_{,m}{F^{-1}}_i^m\\
+v^r_{,s}{F^{-1}}_j^sD^3[(2\rho_0J^{-1}-\Phi)_{,l}{F^{-1}}_r^l]_{,m}{F^{-1}}_i^m]\,dt'\,d\tau.
\end{multline*}
The first summand in the integrand scales like $D^4v[D^2(\rho_0J^{-1})+DG]F^{-1}F^{-1}$. Recall \eqref{ineq:est00}, we have that
\begin{align*}
\|D^3G^i\|_0^2\leq &C \|\int_{\Omega}\dfrac{1}{|\eta(x,t)-\eta(z,t)|}JD_{\eta}^4(\rho_0J^{-1})\,dx\|_0^2
\\\leq &C\|\int_{\Omega}\dfrac{1}{|\eta(x,t)-\eta(z,t)|}(\rho_0D^5\eta)\,dx\|_0^2+\|\int_{\Omega}\dfrac{1}{|\eta(x,t)-\eta(z,t)|}(D^4\eta)\,dx\|_0^2,
\end{align*}
 then we can use integration by parts for the first term of right-hand side, and get that $\|DG\|_2^2\leq P(\sup_{t\in[0,T]}\tilde E(t))$.  We can also control $\|\partial_tDG\|_2^2$ by the same bound. Then the integrand can be estimated by integrating by parts in time in a similar way as we did for the terms associated to $D^3B(F^{-1},Dv)$.

The estimate for the second summand is more complicated. First, by integrating by parts in time (in the integral from $0$ to $\tau$), we have
\begin{align*}
&\kappa\int_0^t\int_0^{\tau}D^4(D_{\eta}(2f-\phi))Dv F^{-1}F^{-1}\,dt'\,d\tau\\
=&-\kappa\int_0^t\int_0^{\tau}\partial_t(DvF^{-1}F^{-1})D^4\int_0^{t'}D_{\eta}(2f-\Phi)\,dt^{''}\,dt'\,d\tau\\
&+\kappa\int_0^tDvF^{-1}F^{-1}D^4\int_0^{\tau}D_{\eta}(2f-\Phi)\,dt'\,d\tau.
\end{align*}
Now recall \eqref{lagappro}, we have that
\begin{equation}
\label{c4}
v_t+D_{\eta}(2\rho_0J^{-1}-\Phi)+\kappa[D_{\eta}(2\rho_0J^{-1}-\Phi)]_t =0,
\end{equation}
then by integrating \eqref{c4} in time twice, we can get
\begin{multline*}
\int_0^t\int_0^{\tau}D^4(D_{\eta}(2\rho_0J^{-1}-\Phi))\,dt'\,d\tau+\kappa\int_0^tD^4(D_{\eta}(2\rho_0J^{-1}-\Phi))\,d\tau \\= -D^4\eta(t)+tD^4u_0,
\end{multline*}
where we have used the fact that $D^4\eta(0)=0$ since $\eta(0)=e$.
By using Lemma \ref{elli}, we see that the following estimates hold independently of $\kappa$:
\begin{equation*}
\|\int_0^t\int_0^{\tau}D^4(D_{\eta}(2\rho_0J^{-1}-\Phi))\,dt'\,d\tau\|_0^2\leq \tilde{M}_0+C\tilde{E}(t),
\end{equation*}
then by using \eqref{c4},
\begin{equation*}
\|\int_0^t\int_0^{\tau}\kappa D^4(D_{\eta}(2\rho_0J^{-1}-\Phi))\,dt'\,d\tau\|_0^2\leq \tilde{M}_0+C\tilde{E}(t),
\end{equation*}
Thus, we finally get the estimate
\begin{equation*}
\|\int_0^t\int_0^{\tau}D^3Q\,dt'\,d\tau\|_0^2\leq CTP(\sup_{t\in[0,T]}\tilde{E}(t)),
\end{equation*}
and hence
\begin{equation*}
\sup_{t\in[0,T]}\|\curl\eta(t)\|_3^2\leq \tilde{M}_0+CTP(\sup_{t\in[0,T]}\tilde{E}(t)).
\end{equation*}
\textbf{Step 2. Estimate for $\curl v_t$.} From \eqref{c0},
\begin{equation}
\label{1c}
\curl v_t=\epsilon_{\cdot ji}\int_0^t{F^{-1}}_{t\,j}^s(t')\,dt'v^i_{t,s}+Q.
\end{equation}
Since
\begin{equation}
\kappa\partial_t[D_{\eta}(2\rho_0J^{-1}-\Phi)]+D_{\eta}(2\rho_0J^{-1}-\Phi)=-v_t,
\end{equation}
by Lemma \ref{elli}, we see that
\begin{equation}
\|D_{\eta}(2\rho_0J^{-1}-\Phi)\|_3^2\leq \tilde{M}_0+\|v_t\|_3^2.
\label{c6}
\end{equation}
from which it immediately follows that $\|Q\|_2^2\leq\tilde{M}_0+CTP(\sup_{t\in[0,T]}\tilde{E}(t))$. For later use, from equation \eqref{lagc} and the estimate \eqref{c6}, we have that
\begin{equation}
\label{4c}
\|\kappa\partial_t[D_{\eta}(2\rho_0J^{-1}-\Phi)]\|_3^2\leq \tilde{M}_0+\|v_t\|_3^2.
\end{equation}
Since the highest-order term in $D^2B(F^{-1},Dv)$ is $D^3v$, then $\|B(F^{-1},Dv)\|_2^2\leq\tilde{M}_0+CTP(\sup_{t\in[0,T]}\tilde{E}(t))$ and we can obtain
\begin{equation}
\|\curl v_t(t)\|_2^2\leq \tilde{M}_0+CTP(\sup_{t\in[0,T]}\tilde{E}(t)).
\end{equation}
\textbf{Step 3. Estimates for $\curl\partial_t^3 v$ and $\curl\partial_t^5 v$.} By time-differentiating \eqref{1c}, and estimating in the same way as we did \textbf{Step 2}, we can show that
\begin{equation*}
\|\curl \partial_t^3v\|_1^2+\|\curl\partial_t^5v\|_0^2\leq \tilde{M}_0+CTP(\sup_{t\in[0,T]}\tilde{E}(t)).
\end{equation*}
\textbf{Step 4. Estimate for $\rho_0\bar{\partial}^4\curl \eta$.} To prove this weighted estimate, we write \eqref{c2} as
\begin{equation*}
\curl v(t)= \epsilon_{jki}v^i_{,s}\int_0^t{F^{-1}}_{t\,j}^2(t')\,dt'+\curl u_0+\int_0^t[B(F^{-1},Dv)+Q](t')\,dt',
\end{equation*}
and integrate in time to find that
\begin{align}
\nonumber \curl\eta(t)=&t\curl u_0+\underbrace{\int_0^t\epsilon_{jki}v^i_{,s}\int_{0}^{t'}{F^{-1}}_{t\,j}^s(\tau)d\tau dt'}_{I_1}\\
&+\underbrace{\int_0^t\int_0^{t'}B(F^{-1},Dv)(\tau)d\tau dt'}_{I_2}+\underbrace{\int_0^t\int_0^{t'}Q(\tau)d\tau dt'}_{I_3}.
\end{align}
It follows that
\begin{equation}
\rho_0\bar{\partial}^4\curl\eta(t)=t\rho_0\bar{\partial}^4\curl u_0+\rho_0\bar{\partial}^4I_1+\rho_0\bar{\partial}^4I_2+\rho_0\bar{\partial}^4I_3.
\end{equation}
By the definition of $\tilde{M}_0$, we have $\|t\rho_0\bar{\partial}^4\curl u_0\|_0^2\leq \tilde{M}_0$, then we only need to estimate the $L^2(\Omega)$-norm of $\rho_0\bar{\partial}^4I_1+\rho_0\bar{\partial}^4I_2+\rho_0\bar{\partial}^4I_3$. We first estimate $\rho_0\bar{\partial}^4I_2$. We write $\rho_0\bar{\partial}^4I_2$ as
\begin{equation*}
\rho_0\bar{\partial}^4I_2(t)=\underbrace{\int_0^t\int_0^{t'}\epsilon_{kji}{F^{-1}}_{t\,j}^s\rho_0\bar{\partial}^4v^i_{,s}d\tau dt'}_{K_1}+\underbrace{int_0^t\int_0^{t'}\epsilon_{kji}\rho_0\bar{\partial}^4{F^{-1}}_{t\,j}^sv^i_{,s}d\tau dt'}_{K_2}+R,
\end{equation*}
where $R$ denotes remainder terms which are lower-order in the derivative count, in particular the terms with the highest derivative order in $R$ scale like $\rho_0\bar{\partial}^3Dv$ or $\rho_0\bar{\partial}^4\eta$, and hence satisfy the inequality $\|R\|_0^2\leq \tilde{M}_0+CTP(\sup_{t\in[0,T]}\tilde{E}(t))$. First we focus on the integral $K_1$, by integrating by parts in time, we get
\begin{equation*}
K_1(t)=\int_0^t\int_0^{t'}\epsilon_{kji}\partial_t^2{F^{-1}}_{j}^s\rho_0\bar{\partial}^4\eta^i_{,s}d\tau dt'+\int_0^t\epsilon_{kji}{F^{-1}}_{t\,j}^s\rho_0\bar{\partial}^4\eta^i_{,s}dt',
\end{equation*}
and hence
\begin{equation*}
\|K_1\|_0^2\leq \tilde{M}_0+CTP(\sup_{t\in[0,T]}\tilde{E}(t)).
\end{equation*}
Using the identity $\partial_t{F^{-1}}_j^s=-{F^{-1}}^s_pv^p_{,b}{F^{-1}}^b_j$, we can show that $K_2$ can be bounded in the same fashion as $K_1$, then we obtain
\begin{equation}
\|\rho_0\bar{\partial}^4I_2(t)\|_0^2\leq\tilde{M}_0+CTP(\sup_{t\in[0,T]}\tilde{E}(t)).
\end{equation}
Similarly, using the same integration-by-parts argument, we also have
\begin{equation}
\|\rho_0\bar{\partial}^4I_1(t)\|_0^2\leq\tilde{M}_0+CTP(\sup_{t\in[0,T]}\tilde{E}(t)).
\end{equation}
Now we estimate $\rho_0\bar{\partial}^4I_3$, which can be written as
\begin{multline*}
\underbrace{\kappa\int_0^t\int_0^{t'}\rho_0\bar{\partial}^4DvD(D_{\eta}(2f-\phi))F^{-1}d\tau dt'}_{J_1}\\+\underbrace{\kappa\int_0^t\int_0^{t'}\rho_0\bar{\partial}^4D(D_{\eta}(2f-\phi))DvF^{-1}d\tau dt'}_{J_2}+R,
\end{multline*}
where $R$ also denotes the remainder terms and satisfies the estimate $\|R\|_0^2\leq\tilde{M}_0+CTP(\sup_{t\in[0,T]}\tilde{E}(t))$.
Since
\begin{equation*}
D(D_{\eta}(2\rho_0J^{-1}-\Phi))+\kappa D(D_{\eta}(2\rho_0J^{-1}-\Phi))=Dv_t,
\end{equation*}
by Lemma \ref{elli}, we see that independently of $\kappa$,
\begin{equation*}
\|D(D_{\eta}(2\rho_0J^{-1}-\Phi))\|_2^2\leq \tilde{M}_0+\|v_t\|_3^2\leq \tilde{M}_0+C\tilde{E}(t),
\end{equation*}
and that
\begin{equation*}
\kappa\|D(D_{\eta}(2\rho_0J^{-1}-\Phi))\|_2^2\leq  \tilde{M}_0+C\tilde{E}(t).
\end{equation*}
Thus, by the Sobolev embedding inequality,
\begin{equation*}
\kappa\|D(D_{\eta}(2\rho_0J^{-1}-\Phi))\|_{L^{\infty}}^2\leq  \tilde{M}_0+C\tilde{E}(t).
\end{equation*}
Hence, using a similar integration-by-parts in time argument we used to estimate $K_1$ above, we can obtain
\begin{equation}
\|J_1(t)\|_0^2\leq\tilde{M}_0+CTP(\sup_{t\in[0,T]}\tilde{E}(t)).
\end{equation}
In order to estimate $J_2$, we will use the structure of the Euler--Poisson equations \eqref{lagc} again. Integrating in time twice, we see that
\begin{equation}
\label{5c}
\int_0^t\int_0^{t'}\rho_0\bar{\partial}^4D(D_{\eta}(2f-\phi)+\kappa\int_0^t\rho_0\bar{\partial}^4D(D_{\eta}(2f-\phi)=-\rho_0\bar{\partial}^4D\eta(t)+t\rho_0\bar{\partial}^4Du_0.
\end{equation}
According to Lemma \ref{elli}, independently of $\kappa$, we have
\begin{equation*}
\|\int_0^t\int_0^{t'}\rho_0\bar{\partial}^4D(D_{\eta}(2f-\Phi)\|_{L^{\infty}(0,T;L^2)}^2\leq \tilde{M}_0+C\|\rho_0\bar{\partial}^4D\eta(t)\|_0^2\leq\tilde{M}_0+C\tilde{E}(t),
\end{equation*}
and then by using \eqref{5c},
\begin{equation}
\|\kappa\int_0^t\rho_0\bar{\partial}^4D(D_{\eta}(2\rho_0J^{-1}-\Phi))\|_{L^{\infty}(0,T;L^2)}^2\leq \tilde{M}_0+C\tilde{E}(t).
\end{equation}
To estimate $J_2$, we integrate-by-parts in time and get
\begin{align*}
J_2=&\int_0^t\int_0^{t'}\kappa\int_0^{\tau}\rho_0\bar{\partial}^4D(D_{\eta}(2\rho_0J^{-1}-\Phi))(s)ds[F^{-1}Dv]_t(\tau)d\tau dt'\\
&+\int_0^t\kappa\int_0^{t'}\rho_0\bar{\partial}^4D(D_{\eta}(2\rho_0J^{-1}-\Phi))(s)dsF^{-1}Dv(t')dt',
\end{align*}
then
\begin{equation*}
\|J_2\|_0^2\leq\tilde{M}_0+CTP(\sup_{t\in[0,T]}\tilde{E}(t)),
\end{equation*}
and $\|\rho_0\bar{\partial}^4I_3(t)\|_0^2\leq \tilde{M}_0+CTP(\sup_{t\in[0,T]}\tilde{E}(t))$. Thus, we finally get
\begin{equation}
\|\rho_0\bar{\partial}^4\curl\eta(t)\|_0^2\leq\tilde{M}_0+CTP(\sup_{t\in[0,T]}\tilde{E}(t))
\end{equation}
\textbf{Step 5. Estimate for $\rho_0\bar{\partial}^3\curl v_t$.} From \eqref{1c}, we have
\begin{align*}
\|\rho_0\bar{\partial}^3\curl v_t(t)\|_0^2&\leq\|\epsilon_{\cdot ji}\rho_0\bar{\partial}^3(\int_0^t{F^{-1}}_{t\,j}^s(t')dt'v_{t,s}^i)\|_0^2+\|\rho_0\bar{\partial}^3Q(t)\|_0^2
\\&\leq CTP(\sup_{t\in[0,T]}\tilde{E}(t))+\|\rho_0\bar{\partial}^3Q(t)\|_0^2,
\end{align*}
and by \eqref{ini}, we have
\begin{multline*}
\|\rho_0\bar{\partial}^3Q(t)\|_0^2\leq C\|\rho_0\bar{\partial}^3Dv(t)\|_0^2\|\kappa D[D_{\eta}(2\rho_0J^{-1}-\Phi)]\|_{L^{\infty}(\Omega)}^2\\+C\|\kappa\rho_0\bar{\partial}^3D[D_{\eta}(2\rho_0J^{-1}-\Phi)]\|_0^2+CTP(\sup_{t\in[0,T]}\tilde{E}(t)).
\end{multline*}
First, we use \eqref{4c} and the fundamental theorem of calculus to get that
\begin{equation*}
\|\kappa D[D_{\eta}(2\rho_0J^{-1}-\Phi)]\|_{L^{\infty}(\Omega)}^2\leq C\|\kappa D_{\eta}(2\rho_0J^{-1}-\Phi)\|_3^2\leq C\tilde{M}_0+CTP(\sup_{t\in[0,T]}\tilde{E}(t)).
\end{equation*}
Then employing the fundamental theorem of calculus again, we have
\begin{equation*}
\|\rho_0\bar{\partial}^3Dv(t)\|_0^2\leq \tilde{M}_0+CTP(\sup_{t\in[0,T]}\tilde{E}(t)),
\end{equation*}
and hence $\|\rho_0\bar{\partial}^3Dv(t)\|_0^2\|D[D_{\eta}(2\rho_0J^{-1}-\Phi)]\|^2_{L^{\infty}(\Omega)}\leq\tilde{M}_0+CTP(\sup_{t\in[0,T]}\tilde{E}(t)).$

On the other hand, since
\begin{equation*}
\rho_0\kappa\partial_t\bar{\partial}^3D[D_{\eta}(2\rho_0J^{-1}-\Phi)]+\rho_0\bar{\partial}^3D[D_{\eta}(2\rho_0J^{-1}-\Phi)]=-\rho_0\bar{\partial}^3Dv_t,
\end{equation*}
by Lemma \ref{elli}, we see that independently of $\kappa$,
\begin{equation*}
\|\rho_0\bar{\partial}^3D[D_{\eta}(2\rho_0J^{-1}-\Phi)]\|\leq \tilde{M}_0+C\tilde{E}(t),
\end{equation*}
and, in turn,
\begin{equation*}
\|\kappa\rho_0\bar{\partial}^3D\partial_t[D_{\eta}(2\rho_0J^{-1}-\Phi)]\|_0^2\leq \tilde{M}_0+C\tilde{E}(t).
\end{equation*}
By the fundamental theorem of calculus, we thus get
\begin{equation*}
\|\kappa\rho_0\bar{\partial}^3D[D_{\eta}(2\rho_0J^{-1}-\Phi)]\|_0^2 \leq\tilde{M}_0+CTP(\sup_{t\in[0,T]}\tilde{E}(t)),
\end{equation*}
which shows that $\|\rho_0\bar{\partial}^3\curl v_t\|_0^2\leq\tilde{M}_0+CTP(\sup_{t\in[0,T]}\tilde{E}(t)).$

\noindent \textbf{Step 6. Estimates for $\rho_0\bar{\partial}^2\curl\partial_t^3v$, $\rho_0\bar{\partial}\curl\partial_t^5v$, and $\rho_0\curl \partial_t^7v$.} By time-differentiating \eqref{1c} and estimating as in \textbf{Step 5}, we immediately obtain the inequality
\begin{align*}
\|\rho_0&\bar{\partial}^2\curl\partial_t^3v\|_0^2+\|\rho_0\bar{\partial}\curl\partial_t^5v\|_0^2+\|\rho_0\curl \partial_t^7v\|_0^2\\
&\leq \tilde{M}_0+CTP(\sup_{t\in[0,T]}\tilde{E}(t))
\end{align*}
\textbf{Step 7. Estimate for $\sqrt{\kappa}\rho_0\curl_{\eta}\bar{\partial}^4v$.} From \eqref{c2},
\begin{align*}
\sqrt{\kappa}\rho_0\curl_{\eta}\bar{\partial}^4v(t)=&\underbrace{\sqrt{\kappa}\rho_0\bar{\partial}^4\curl u_0}_{S_1}+\underbrace{\sqrt{\kappa}\rho_0\epsilon_{\cdot ij}v^i_{,r}\bar{\partial}^4{F^{-1}}_j^r(t)}_{S_2}\\
&+\underbrace{\int_0^t\sqrt{\kappa}\rho_0\bar{\partial}^4B(F^{-1},Dv)dt'}_{S_3}+\underbrace{\int_0^t\sqrt{\kappa}\rho_0\bar{\partial}^4Q dt'}_{S_4}+R(t)
\end{align*}
where $R$ is a lower-order remainder term satisfying an inequality of the type $\int_0^T|R(t)|^2dt\leq CTP(\sup_{t\in[0,T]}\tilde{E}(t))$. It is easy to check that
$\int_0^t\|S_1\|_0^2dt'\leq t\tilde{M}_0$, and $\int_0^t\|S_2\|_0^2dt'\leq CTP(\sup_{t\in[0,T]}\tilde{E}(t))$.
By Jensen's inequality, we also have $\int_0^t\|S_3\|_0^2dt'\leq CTP(\sup_{t\in[0,T]}\tilde{E}(t))$.

\noindent The highest-order terms in $S_4$ can be written under the form
\begin{multline*}
\underbrace{\int_0^t\kappa^{\frac{3}{2}}\rho_0\bar{\partial}^4Dv F^{-1}D[D_{\eta}(2\rho_0J^{-1}-\Phi)]F^{-1}dt'}_{S_{4a}}\\
+\underbrace{\int_0^t\kappa^{\frac{3}{2}}\rho_0\bar{\partial}^4D^2[D_{\eta}(2\rho_0J^{-1}-\Phi)]v F^{-1}F^{-1}dt'}_{S_{4b}},
\end{multline*}
with all other terms being lower-order and easily estimated. By Jensen's inequality and using \eqref{ini},
\begin{equation*}
\int_0^t\|S_{4a}\|_0^2dt'\leq C\kappa\int_0^t t'\int_0^{t'}\|\sqrt{\kappa}\rho_0\bar{\partial}^4Dv(t{''})\|_0^2dt^{''}dt'\leq CT\sup_{t\in[0,T]}\tilde{E}(t).
\end{equation*}
In order to estimate the term $S_{4b}$, we will use the identity
\begin{align*}
\kappa^{\frac{3}{2}}\rho_0\bar{\partial}^4D&[D_{\eta}(2f-\Phi)]+\sqrt{\kappa}\int_0^t \rho_0\bar{\partial}^4D[D_{\eta}(2\rho_0J^{-1}-\Phi)]dt'\\&=-\sqrt{\kappa}\rho_0\bar{\partial}^4Dv(t)+\sqrt{\kappa}\rho_0\bar{\partial}^4Du_0+\kappa^{\frac{3}{2}}\rho_0\bar{\partial}^4D^2(2\rho_0J^{-1}-\Phi),
\end{align*}
which follows from differentiating the equation \eqref{lagc}. Taking the $L^2(\Omega)$-inner product of this equation with $\kappa^{\frac{3}{2}}\rho_0\bar{\partial}^4D[D_{\eta}(2\rho_0J^{-1}-\Phi)](t)$ and integrating in time, we deduce that
\begin{equation*}
\int_0^t \|\kappa^{\frac{3}{2}}\rho_0\bar{\partial}^4D[D_{\eta}(2\rho_0J^{-1}-\Phi)]\|_0^2dt'\leq \tilde{M}_0+\sup_{t\in[0,T]}\tilde{E}(t),
\end{equation*}
from which it follows, using Jensen's inequality, that
\begin{equation*}
\int_0^t\|S_{4b}\|_0^2dt'\leq CTP(\sup_{t\in[0,T]}\tilde{E}(t)),
\end{equation*}
and thus $\int_0^t\|\sqrt{\kappa}\rho_0\curl_{\eta}\bar{\partial}^4v\|_0^2dt'\leq CTP(\sup_{t\in[0,T]}\tilde{E}(t))$.

\noindent \textbf{Step 8. Estimates for $\sqrt{\kappa}\rho_0\curl_{\eta}\bar{\partial}^{4-l}\partial_t^{2l}v$, $l=1,2,3,4$.} Following the identical methodology as we used in \textbf{Step 7}, we obtain the desired inequality
\begin{equation*}
\sum_{l=1}^4\int_0^t\|\sqrt{\kappa}\rho_0\curl_{\eta}\bar{\partial}^{4-l}\partial_t^{2l}v(t')\|_0^2\,dt' \leq CTP(\sup_{t\in[0,T]}\tilde{E}(t)).
\end{equation*}
\end{proof}
\subsection{$\kappa$-independent energy estimates for tangential and time derivatives }
\label{energyestimate}
In this subsection, we take $T\in(0,T_{\kappa})$. In the estimates below we would provide a detailed explanation about how the energy \eqref{energy} is formed and how to control error terms by the higher-order energy function. We will also show that all of the estimates do not depend on the parameter $\kappa$.
\subsubsection{The $\bar{\partial}^4$-problem}
\begin{proposition}
For $\delta>0$ and letting the constant $\tilde{M}_0$ depend on $1/\delta$,
\begin{multline}
\sup_{t\in[0,T]}(\|\sqrt{\rho_0}\bar{\partial}^4v(t)\|_0^2+\|\rho_0\bar{\partial}^4D\eta(t)\|_0^2+\int_0^t\|\sqrt{\kappa}\rho_0\bar{\partial}^4Dv(s)\|_0^2\,ds)\\
\leq \tilde{M}_0+\delta\sup_{t\in[0,T]}\tilde{E}(t)+C\sqrt{T}P(\sup_{t\in[0,T]}\tilde{E}(t)).
\label{p4prop}
\end{multline}
\end{proposition}
\begin{proof}
Letting $\bar{\partial}^4$ act on \eqref{lagappro}, and taking $L^2(\Omega)$-inner product of this with $\bar{\partial}^4 v^i$ yields
\begin{multline*}
\dfrac{1}{2}\dfrac{d}{dt}\int_{\Omega}\rho_0|\bar{\partial}^4v|^2\,dx+\underbrace{\int_{\Omega}\bar{\partial}^4\partial_k(\rho_0^2{F^{-1}}_i^kJ^{-1})\bar{\partial}^4v^i\,dx}_{I_1}
\\+\underbrace{\int_{\Omega}\kappa\partial_t\bar{\partial}^4\partial_k(\rho_0^2{F^{-1}}_i^kJ^{-1})\bar{\partial}^4v^i\,dx}_{I_2}
=\underbrace{\int_{\Omega}\rho_0\bar{\partial}^4G^i\bar{\partial}^4v^i\,dx}_{I_3}+\underbrace{\kappa\int_{\Omega}\rho_0\bar{\partial}^4\partial_tG^i\bar{\partial}^4v^i\,dx}_{I_4}.
\end{multline*}
To estimate $I_1$, we use \eqref{pn} and integrate by parts to obtain that
\begin{align*}
I_1=&\int_{\Omega}\rho_0^2\big(J^{-1}{F^{-1}}_r^k[D_{\eta}\bar{\partial}^4\eta]_r^i+J^{-1}{F^{-1}}_i^k \Div_{\eta}\bar{\partial}^4\eta\big)\bar{\partial}^4v^i_{,k}\,dx\\&+\underbrace{\int_{\Omega}\rho_0^2J^{-1}{F^{-1}}_r^k[\text{Curl}_{\eta}\bar{\partial}^4\eta]_i^r\bar{\partial}^4v^i_{,k}\,dx}_{I_{11}}\\
&-\int_{\Omega}\sum_{p=1}^{3}\partial_k\big\{\bigg(\bar{\partial}^p[\rho_0^2J^{-1}{F^{-1}}_r^k{F^{-1}}_i^s]\bar{\partial}^{4-p}\eta^r_{,s}
\\&\quad\quad+\bar{\partial}^p[\rho_0^2J^{-1}{F^{-1}}_i^k{F^{-1}}_r^s]\bar{\partial}^{4-p}\eta^r_{,s}\bigg)\big\}\bar{\partial}^4v^i\,dx+R
\\=&\dfrac{1}{2}\dfrac{d}{dt}\int_{\Omega}\rho_0^2J^{-1}|D_{\eta}\bar{\partial}^4\eta|^2\,dx+\dfrac{1}{2}\dfrac{d}{dt}\int_{\Omega}\rho_0^2J^{-1}|\Div_{\eta}\bar{\partial}^4\eta|^2\,dx
\\&+\dfrac{1}{2}\int_{\Omega}\rho_0^2J^{-2}J_t|D_{\eta}\bar{\partial}^4\eta|^2\,dx
-\int_{\Omega}\rho_0^2\bar{\partial}^4\eta^i_{,k}J^{-1}\partial_t{F^{-1}}_r^k[D_{\eta}\bar{\partial}^4\eta]_r^i\,dx\\
&+\dfrac{1}{2}\int_{\Omega}\rho_0^2J^{-2}J_t|\Div_{\eta}\bar{\partial}^4\eta|^2\,dx
-\int_{\Omega}\rho_0^2\bar{\partial}^4\eta^i_{,k}J^{-1}\partial_t{F^{-1}}_i^k[\Div_{\eta}\bar{\partial}^4\eta]\,dx\\
&+I_{11}+R_1+R,
\end{align*}
where
\begin{multline*}
R_1=-\int_{\Omega}\sum_{p=1}^{3}\partial_k\big\{\bigg(\bar{\partial}^p[\rho_0^2J^{-1}{F^{-1}}_r^k{F^{-1}}_i^s]\bar{\partial}^{4-p}\eta^r_{,s}
\\+\bar{\partial}^p[\rho_0^2J^{-1}{F^{-1}}_i^k{F^{-1}}_r^s]\bar{\partial}^{4-p}\eta^r_{,s}\bigg)\big\}\bar{\partial}^4v^i\,dx,
\end{multline*}
and $R=\int_{\Omega}\partial_k(\bar\partial^3[\bar{\partial}(\rho_0^2){F^{-1}}_i^kJ^{-1}])\bar\partial^4v^i\,dx.$

\noindent For $I_{11}$, we can use the following anti-symmetrization to obtain the curl structure
\begin{align*}
I_{11}=&(\sum_{i>r}+\sum_{i<r})\int_{\Omega}\rho_0^2J^{-1}{F^{-1}}_r^k\bar{\partial}^4v^i_{,k}({F^{-1}}_i^s\bar{\partial}^4\eta^r_{,s}-{F^{-1}}_r^s\bar{\partial}^4\eta^i_{,s})\,dx\\
=&\sum_{i>r}\int_{\Omega}\rho_0^2J^{-1}({F^{-1}}_r^k\bar{\partial}^4v^i_{,k}-{F^{-1}}_i^k\bar{\partial}^4v^r_{,k})({F^{-1}}_i^s\bar{\partial}^4\eta^r_{,s}-{F^{-1}}_r^s\bar{\partial}^4\eta^i_{,s})\,dx
\\=&-\dfrac{1}{2}\dfrac{d}{dt}\int_{\Omega}\rho_0^2J^{-1}|\curl_{\eta}\bar{\partial}^4\eta|^2\,dx-\dfrac{1}{2}\int_{\Omega}\rho_0^2J^{-2}J_t|\curl_{\eta}\bar{\partial}^4\eta|^2\,dx
\\&+\underbrace{\sum_{i>r}\int_{\Omega}\rho_0^2J^{-1}(\partial_t{F^{-1}}_r^k\bar{\partial}^4\eta^i_{,k}-\partial_t{F^{-1}}_i^k\bar{\partial}^4\eta^r_{,k})[\text{Curl}_{\eta}\bar{\partial}^4\eta]_i^r\,dx}_{R_{12}}.
\end{align*}
For $R_1$, we claim that
\begin{equation}
\label{rest}
\int_0^T R_1(t)\,dt \leq \tilde{M}_0+ C\sqrt{T}P(\sup_{t\in[0,T]}\tilde{E}(t)).
\end{equation}
By integrating by parts with respect to $x_k$ and then with respect to $t$, and using \eqref{piola}, we can obtain that
\begin{multline*}
\int_0^T R_1\,dt=\sum_{p=1}^3\int_0^T\int_{\Omega}\rho_0^2\partial_t\bigg(\bar{\partial}^p[J^{-1}{F^{-1}}_r^k{F^{-1}}_i^s]\bar{\partial}^{4-p}\eta^r_{,s}
\\+\bar{\partial}^p[J^{-1}{F^{-1}}_i^k{F^{-1}}_r^s]\bar{\partial}^{4-p}\eta^r_{,s}\bigg)\bar{\partial}^4\eta^i_{,k}\,dxdt\\
-\sum_{p=1}^3\int_{\Omega}\rho_0^2\bigg(\bar{\partial}^p[J^{-1}{F^{-1}}_r^k{F^{-1}}_i^s]\bar{\partial}^{4-p}\eta^r_{,s}
\\+\bar{\partial}^p[J^{-1}{F^{-1}}_i^k{F^{-1}}_r^s]\bar{\partial}^{4-p}\eta^r_{,s}\bigg)\bar{\partial}^4\eta^i_{,k}\,dx\bigg|_0^T+R_{11}
\end{multline*}
$R_{11}$ is defined by
$$
\int_0^T\int_{\Omega}\sum_{p=1,1\leq q\leq p}^{3}\bigg(\bar{\partial}^{q}(\rho_0^2)\bar{\partial}^{p-q}[J^{-1}{F^{-1}}_r^k{F^{-1}}_i^s+J^{-1}{F^{-1}}_i^k{F^{-1}}_r^s]\bar{\partial}^{4-p}\eta^r_{,s}\bigg)\bar{\partial}^4{v^i}_{,k}\,dx\,dt
$$
Notice that when $p=1$, the space-time integral on the right side hand scale like $l[\rho_0\bar{\partial}Dv\bar{\partial}^3D\eta+\rho_0\bar{\partial}D\eta\bar{\partial}^3Dv]\rho_0\bar{\partial}^4D\eta$ where $l$ denotes $L^{\infty}(\Omega)$ function. Since $\|\rho_0\bar{\partial}^4D\eta\|_0^2$ is contained in the energy function $\tilde{E}$, $\|\bar{\partial}D\eta\|_{\infty}\leq C\|D^2\eta\|_2$ with $\|D^2\eta\|_2$ being contained in $\tilde{E}$, and we can also write $\rho_0\bar{\partial}^3Dv(t)=\rho_0\bar{\partial}^3Dv(0)+\int_0^t\rho_0\bar{\partial}^3Dv_t(\tau)\,d\tau$, the second term could be estimated by using $L^{\infty}$-$L^2$-$L^2$ H\"{o}lder's inequality and be controlled by a bound indicated in \eqref{rest}. Similarly, for the first term, we use $L^4$-$L^4$-$L^2$ H\"older's inequality.

\noindent For the case that $p=2$, the space-time integral scales like $l[\rho_0\bar{\partial}^2Dv\bar{\partial}^2D\eta+\rho_0\bar{\partial}Dv(\bar{\partial}^2Dv+\bar{\partial}^2D\eta)]\rho_0\bar{\partial}^4D\eta$, the first part can be estimated by using $L^4$-$L^4$-$L^2$ H\"older's inequality, and the second part could be estimated by $L^{\infty}$-$L^2$-$L^2$ H\"{o}lder's inequality. For the case that $p=3$, the estimate is just the same as $p=1$.

Now we deal with the space integral on the right-hand side of the expression for $\int_0^TR_1\,dt$. The integral at time $t=0$ is equal to zero since $\eta(x,0)=x$. And by using the fundamental theorem of calculus the integral evaluated at $t=T$ can be written as
\begin{multline*}
-\sum_{p=1}^3\int_{\Omega}\rho_0^2\bigg(\bar{\partial}^p[J^{-1}{F^{-1}}_r^k{F^{-1}}_i^s]\bar{\partial}^{4-p}\eta^r_{,s}
\\+\bar{\partial}^p[J^{-1}{F^{-1}}_i^k{F^{-1}}_r^s]\bar{\partial}^{4-p}\eta^r_{,s}\bigg)\bar{\partial}^4\eta^i_{,k}\,dx\bigg|_{t=T}\\
=-\sum_{p=1}^3\int_{\Omega}\rho_0^2\int_0^T\bigg(\bar{\partial}^p[J^{-1}{F^{-1}}_r^k{F^{-1}}_i^s]\bar{\partial}^{4-p}\eta^r_{,s}
\\+\bar{\partial}^p[J^{-1}{F^{-1}}_i^k{F^{-1}}_r^s]\bar{\partial}^{4-p}\eta^r_{,s}\bigg)_t\,dt\bar{\partial}^4\eta^i_{,k}(T)\,dx,
\end{multline*}
which can be estimated in the identical fashion as the corresponding space-time integral. For $R_{11}$, it can be estimated in the same fashion before.
As such, we have proved that $R_1$ has the claimed bound \eqref{rest}.

For $\int_0^TR\,dt$, by integration-by-parts with respect to $x_k$ and time $t$, we have
\begin{multline*}
\int_0^T R\,dt= \int_0^T\int_{\Omega}\partial_t\bar{\partial}^3[\bar{\partial}(\rho_0^2){F^{-1}}_i^kJ^{-1}]\bar{\partial}^4{\eta^i}_{,k}\,dx\,dt\\-\int_{\Omega}\bar{\partial}^3[\bar{\partial}(\rho_0^2){F^{-1}}_i^kJ^{-1}]\bar{\partial}^4{\eta^i}_{,k}\,dx\bigg|_{t=T}
\end{multline*}
For $\int_{\Omega}\bar{\partial}^4(\rho_0^2){F^{-1}}_i^kJ^{-1}\bar{\partial}^4{\eta^i}_{,k}\,dx (T)$, by the fundamental theorem of calculus, we have
\begin{multline*}
\int_{\Omega}\bar{\partial}^4(\rho_0^2){F^{-1}}_i^kJ^{-1}\bar{\partial}^4{\eta^i}_{,k}\,dx (T)=\int_{\Omega}\bar{\partial}^4(\rho_0^2)\Div\bar\partial^4\eta\,dx(T)\\+\int_{\Omega}\bar{\partial}^4(\rho_0^2)\int_0^T({F^{-1}}_i^kJ^{-1})_t\,d\tau\bar\partial^4{\eta^i}_{,k}\,dx(T)
\\\leq \tilde{M}_0+\delta\sup_{t\in[0,T]}\tilde{E}(t)+C\sqrt{T}P(\sup_{t\in[0,T]}\tilde{E}(t)),
\end{multline*}
for some $\delta>0$. Here we also used Young's inequality.
All other integrals can be estimated in a similar fashion for $R_1$. Thus we have
$$
\int_0^T R\,dt\leq \tilde{M}_0+\delta\sup_{t\in[0,T]}\tilde{E}(t)+ C\sqrt{T}P(\sup_{t\in[0,T]}\tilde{E}(t)).$$

It is also easy to see that
\begin{equation*}\int_0^TR_{12}(\tau)\,d\tau \leq C \int_0^T \|\rho_0 \bar{\partial}^4D\eta\|_0 d\tau \leq CT\sup_{[0,T]}\tilde{E}(t).
\end{equation*}

Now we estimate the integral $I_2$. By integrating by parts with respect to $x_k$, we have
\begin{align*}
I_2=&\kappa\int_{\Omega}\rho_0^2\big(J^{-1}{F^{-1}}_r^k[D_{\eta}\bar{\partial}^4v]_r^i+J^{-1}{F^{-1}}_i^k \Div_{\eta}\bar{\partial}^4v\big)\bar{\partial}^4v^i_{,k}\,dx\\&+\underbrace{\kappa\int_{\Omega}\rho_0^2J^{-1}{F^{-1}}_r^k[\text{Curl}_{\eta}\bar{\partial}^4v]_i^r\bar{\partial}^4v^i_{,k}\,dx}_{I_{21}}\\
&-\int_{\Omega}\sum_{p=1}^{3}\partial_k\{\kappa\rho_0^2\bigg(\bar{\partial}^p[J^{-1}{F^{-1}}_r^k{F^{-1}}_i^s]\bar{\partial}^{4-p}v^r_{,s}
\\&\quad\quad+\bar{\partial}^p[J^{-1}{F^{-1}}_i^k{F^{-1}}_r^s]\bar{\partial}^{4-p}v^r_{,s}\bigg)\}\bar{\partial}^4v^i\,dx
\\=&\int_{\Omega}\kappa\rho_0^2J^{-1}(|D_{\eta}\bar{\partial}^4v|^2+|\Div_{\eta}\bar{\partial}^4v|^2)\,dx+I_{21}+R_2,
\end{align*}
where
\begin{multline*}
R_2=-\int_{\Omega}\sum_{p=1}^{3}\partial_k\{\kappa\rho_0^2\bigg(\bar{\partial}^p[J^{-1}{F^{-1}}_r^k{F^{-1}}_i^s]\bar{\partial}^{4-p}v^r_{,s}
\\+\bar{\partial}^p[J^{-1}{F^{-1}}_i^k{F^{-1}}_r^s]\bar{\partial}^{4-p}v^r_{,s}\bigg)\}\bar{\partial}^4v^i\,dx.
\end{multline*}
For $I_{21}$, we see that
\begin{align*}
I_{21}=&(\sum_{i>r}+\sum_{i<r})\int_{\Omega}\kappa\rho_0^2J^{-1}{F^{-1}}_r^k\bar{\partial}^4v^i_{,k}({F^{-1}}_i^s\bar{\partial}^4v^r_{,s}-{F^{-1}}_r^s\partial_t^7v^i_{,s})\,dx\\
=&\sum_{i>r}\int_{\Omega}\kappa\rho_0^2J^{-1}({F^{-1}}_r^k\bar{\partial}^4v^i_{,k}-{F^{-1}}_i^k\partial_t^8v^r_{,k})({F^{-1}}_i^s\bar{\partial}^4v^r_{,s}-{F^{-1}}_r^s\bar{\partial}^4v^i_{,s})\,dx\\
\\=&-\int_{\Omega}\kappa\rho_0^2J^{-1}|\curl_{\eta}\bar{\partial}^4v|^2\,dx-\dfrac{1}{2}\kappa\int_{\Omega}\rho_0^2J^{-2}J_t|\curl_{\eta}\bar{\partial}^4v|^2\,dx
\\&+\underbrace{\sum_{i>r}\kappa\int_{\Omega}\rho_0^2J^{-1}(\partial_t{F^{-1}}_r^k\bar{\partial}^4v^i_{,k}-\partial_t{F^{-1}}_i^k\bar{\partial}^4\eta^r_{,k})[\text{Curl}_{\eta}\bar{\partial}^4v]_i^r\,dx}_{R_{22}}.
\end{align*}
The remainder terms $R_2,R_{22}$ can be estimated by a similar way we used above for $R_1, R_{12}$, then
\begin{equation*}
\int_0^T R_2(t)\,dt \leq \tilde{M}_0+\delta\sup_{t\in[0,T]}\tilde{E}(t)+ C\sqrt{T}P(\sup_{t\in[0,T]}\tilde{E}(t)),
\end{equation*}
and
\begin{equation*}\int_0^TR_{22}(\tau)\,d\tau \leq CT\sup_{[0,T]}\tilde{E}(t).
\end{equation*}
For $I_3$ and $I_4$, we have
\begin{align*}
&I_3\leq\|\sqrt{\rho_0}\bar{\partial}^4G\|_0\|\sqrt{\rho_0}\bar{\partial}^4v\|_0,\\
&I_4\leq\kappa\|\sqrt{\rho_0}\partial_t\bar{\partial}^4G\|_0\|\sqrt{\rho_0}\bar{\partial}^4v\|_0,
\end{align*}
then with \eqref{gest}, $\int_0^T I_3+I_4\,dt\leq \tilde{M}_0+C\sqrt{T}P(\sup_{t\in[0,T]}\tilde E(t))$.

Combining all these estimates above, the proposition is proved.
\end{proof}
\begin{corollary}[Estimates for the trace of the tangential components of $\eta(t)$]
For $\alpha=1,2$, and $\delta>0$,
\begin{equation*}
\sup_{t\in[0,T]}|\eta^{\alpha}(t)|^2_{3.5}\leq \tilde{M}_0+\delta\sup_{t\in[0,T]}\tilde{E}+C\sqrt{T}P(\sup_{t\in[0,T]}\tilde{E}).
\end{equation*}
\end{corollary}
\begin{proof}
The weighted embedding estimate \eqref{embd} shows that
\begin{equation*}
\|\bar{\partial}^4\eta(t)\|_0^2 \leq C\int_{\Omega}\rho_0^2(|\bar{\partial}^4\eta|^2+|\bar{\partial}^4D\eta|^2)\,dx.
\end{equation*}
And by the fundamental theorem of calculus, we see that
\begin{equation*}
\sup_{t\in[0,T]}\int_{\Omega}\rho_0^2|\bar{\partial}^4\eta|^2\,dx=\sup_{t\in[0,T]}\int_{\Omega}\rho_0^2|\int_0^t\bar{\partial}^4v\,dt'|^2\,dx\leq T^2 \sup_{t\in[0,T]}\|\sqrt{\rho_0}\bar{\partial}^4v\|_0^2.
\end{equation*}
Then it follows from \eqref{p4prop} that
\begin{equation*}
\sup_{t\in[0,T]}\|\bar{\partial}^4\eta(t)\|_0^2\leq\tilde{M}_0+C\sqrt{T}P(\sup_{t\in[0,T]}\tilde{E}).
\end{equation*}
According to our curl estimates \eqref{curlest}, $\sup_{[0,T]}\|\curl\eta\|_3^2\leq \tilde{M}_0+CTP(\sup_{[0,T]}\tilde{E})$, from which it follows that
\begin{equation*}
\sup_{t\in[0,T]}\|\bar{\partial}^4\curl\eta(t)\|_{H^1(\Omega)'}^2\leq \tilde{M}_0+CTP(\sup_{t\in[0,T]}\tilde{E}),
\end{equation*}
since $\bar{\partial}$ is a tangential derivative, and the integration by parts with respect to $\bar{\partial}$ does not produce any boundary contributions. From the tangential trace inequality \eqref{trineq}, we find that
\begin{equation*}
\sup_{t \in [0,T]}|\bar{\partial}^4\eta^{\alpha}(t)|_{-0.5}^2\leq \tilde{M}_0 +C\sqrt{T}P(\sup_{t\in[0,T]}\tilde{E}),
\end{equation*}
and then
\begin{equation*}
\sup_{t \in [0,T]}|\bar{\partial}^4\eta^{\alpha}(t)|_{3.5}^2\leq \tilde{M}_0 +C\sqrt{T}P(\sup_{t\in[0,T]}\tilde{E}).
\end{equation*}
\end{proof}
\subsubsection{The $\partial_t^8$-problem}
\begin{proposition}
\label{tprop}
For $\delta>0$ and letting the constant $\tilde{M}_0$ depend on $1/\delta$,
\begin{align}
\nonumber \sup_{t\in[0,T]}\bigg(\|\sqrt{\rho_0}\partial_t^8 v(t)\|_0+\|\rho_0\partial_t^7Dv(t)\|_0+\int_0^t\|\sqrt{\kappa}\rho_0D\partial_t^8v(s)\|_0^2\,ds\bigg)\\
\leq \tilde{M}_0+\delta \sup_{[0,T]}\tilde{E}+C\sqrt{T}P(\sup_{[0,T]}\tilde{E}).
\end{align}
\end{proposition}
\begin{proof}
Letting $\partial_t^8$ act on \eqref{lagappro}, and taking $L^2(\Omega)$-inner product of this with $\partial_t^8 v^i$ yields
\begin{multline*}
\dfrac{1}{2}\dfrac{d}{dt}\int_{\Omega}\rho_0|\partial_t^8v|^2\,dx+\underbrace{\int_{\Omega}\partial_t^8\partial_k(\rho_0^2{F^{-1}}_i^kJ^{-1})\partial_t^8v^i\,dx}_{I_1}
\\+\underbrace{\int_{\Omega}\kappa\partial_t^9\partial_k(\rho_0^2{F^{-1}}_i^kJ^{-1})\partial_t^8v^i\,dx}_{I_2}
=\underbrace{\int_{\Omega}\rho_0\partial_t^8G^i\partial_t^8v^i\,dx}_{I_3}+\underbrace{\kappa\int_{\Omega}\rho_0\partial_t^9G^i\partial_t^8v^i\,dx}_{I_4}.
\end{multline*}
To estimate $I_1$, we use \eqref{pt} and integrate by parts to obtain that
\begin{align*}
I_1=&\int_{\Omega}\rho_0^2\big(J^{-1}{F^{-1}}_r^k[D_{\eta}\partial_t^7v]_r^i+J^{-1}{F^{-1}}_i^k \Div_{\eta}\partial_t^7v\big)\partial_t^8v^i_{,k}\,dx\\&+\underbrace{\int_{\Omega}\rho_0^2J^{-1}{F^{-1}}_r^k[\text{Curl}_{\eta}\partial_t^7v]_i^r\partial_t^8v^i_{,k}\,dx}_{I_{11}}\\
&-\int_{\Omega}\sum_{p=1}^{7}\partial_k\{\rho_0^2\bigg(\partial_t^p[J^{-1}{F^{-1}}_r^k{F^{-1}}_i^s]\partial_t^{8-p}\eta^r_{,s}
\\&\quad\quad+\partial_t^p[J^{-1}{F^{-1}}_i^k{F^{-1}}_r^s]\partial_t^{8-p}\eta^r_{,s}\bigg)\}\partial_t^8v^i\,dx
\end{align*}
\begin{align*}
\\=&\dfrac{1}{2}\dfrac{d}{dt}\int_{\Omega}\rho_0^2J^{-1}|D_{\eta}\partial_t^7v|^2\,dx+\dfrac{1}{2}\dfrac{d}{dt}\int_{\Omega}\rho_0^2J^{-1}|\Div_{\eta}\partial_t^7v|^2\,dx
\\&+\dfrac{1}{2}\int_{\Omega}\rho_0^2J^{-2}J_t|D_{\eta}\partial_t^7v|^2\,dx
-\int_{\Omega}\rho_0^2\partial_t^7v^i_{,k}J^{-1}\partial_t{F^{-1}}_r^k[D_{\eta}\partial_t^7v]_r^i\,dx\\
&+\dfrac{1}{2}\int_{\Omega}\rho_0^2J^{-2}J_t|\Div_{\eta}\partial_t^7v|^2\,dx
-\int_{\Omega}\rho_0^2\partial_t^7v^i_{,k}J^{-1}\partial_t{F^{-1}}_i^k[\Div_{\eta}\partial_t^7v]\,dx\\
&+I_{11}+R_1,
\end{align*}
where
\begin{multline*}
R_1=-\int_{\Omega}\sum_{p=1}^{7}\partial_k\{\rho_0^2\bigg(\partial_t^p[J^{-1}{F^{-1}}_r^k{F^{-1}}_i^s]\partial_t^{8-p}\eta^r_{,s}
\\+\partial_t^p[J^{-1}{F^{-1}}_i^k{F^{-1}}_r^s]\partial_t^{8-p}\eta^r_{,s}\bigg)\}\partial_t^8v^i\,dx.
\end{multline*}
For $I_{11}$, we can use the following anti-symmetrization to obtain the curl structure
\begin{align*}
I_{11}=&(\sum_{i>r}+\sum_{i<r})\int_{\Omega}\rho_0^2J^{-1}{F^{-1}}_r^k\partial_t^8v^i_{,k}({F^{-1}}_i^s\partial_t^7v^r_{,s}-{F^{-1}}_r^s\partial_t^7v^i_{,s})\,dx\\
=&\sum_{i>r}\int_{\Omega}\rho_0^2J^{-1}({F^{-1}}_r^k\partial_t^8v^i_{,k}-{F^{-1}}_i^k\partial_t^8v^r_{,k})({F^{-1}}_i^s\partial_t^7v^r_{,s}-{F^{-1}}_r^s\partial_t^7v^i_{,s})\,dx
\\=&-\dfrac{1}{2}\dfrac{d}{dt}\int_{\Omega}\rho_0^2J^{-1}|\curl_{\eta}\partial_t^7v|^2\,dx-\dfrac{1}{2}\int_{\Omega}\rho_0^2J^{-2}J_t|\curl_{\eta}\partial_t^7v|^2\,dx
\\&+\underbrace{\sum_{i>r}\int_{\Omega}\rho_0^2J^{-1}(\partial_t{F^{-1}}_r^k\partial_t^7v^i_{,k}-\partial_t{F^{-1}}_i^k\partial_t^7v^r_{,k})[\text{Curl}_{\eta}\partial_t^7v]_i^r\,dx}_{R_{12}}
\end{align*}
For $R_1$, we claim that
\begin{equation*}
\int_0^T R_1(t)\,dt \leq \tilde{M}_0+\delta\sup_{t\in[0,T]}\tilde{E}(t)+ C\sqrt{T}P(\sup_{t\in[0,T]}\tilde{E}(t)).
\end{equation*}
By integrating by parts with respect to $x_k$ and then with respect to the time derivative $\partial_t$, we obtain that
\begin{align*}
&\int_0^TR_1\,dt\\=&\sum_{p=1}^7\int_0^T\int_{\Omega}\rho_0^2\partial_t\bigg(\partial_t^p[J^{-1}{F^{-1}}_r^k{F^{-1}}_i^s]\partial_t^{8-p}\eta^r_{,s}
+\partial_t^p[J^{-1}{F^{-1}}_i^k{F^{-1}}_r^s]\partial_t^{8-p}\eta^r_{,s}\bigg)\partial_t^7v^i_{,k}\,dxdt\\
-&\sum_{p=1}^7\int_{\Omega}\rho_0^2\bigg(\partial_t^p[J^{-1}{F^{-1}}_r^k{F^{-1}}_i^s]\partial_t^{8-p}\eta^r_{,s}
+\partial_t^p[J^{-1}{F^{-1}}_i^k{F^{-1}}_r^s]\partial_t^{8-p}\eta^r_{,s}\bigg)\partial_t^7v^i_{,k}\,dx\bigg|_0^T
\end{align*}
Notice that when $p=1$, the space-time integral on the right side hand scales like $l[\rho_0Dv_t\partial_t^7D\eta+\rho_0Dv\partial_t^8D\eta]\rho_0\partial_t^7Dv$ where $l$ denotes $L^{\infty}(\Omega)$ function. Since $\|\rho_0\partial_t^7Dv\|_0^2$ is contained in the energy function $\tilde{E}$, $\|Dv_t\|_{\infty}\leq C\|\partial_t^2\eta\|_2$ with $\|\partial_t^2\eta\|_2$ being contained in $\tilde{E}$, and we can also write $\rho_0\partial_t^7D\eta(t)=\rho_0\partial_t^7D\eta(0)+\int_0^t\rho_0\partial_t^8D\eta(\tau)\,d\tau$, there terms could be estimated by using $L^{\infty}$-$L^2$-$L^2$ H\"{o}lder's inequality and be controlled by a bound which is similar to \eqref{rest}.

\noindent For the case that $p=2$, the integral scales like $$l[\rho_0(Dv_{tt}+Dv_tDv+DvDvDv)\partial_t^6D\eta+\rho_0(Dv_t+DvDv)\partial_t^7D\eta]\rho_0\partial_t^7Dv,$$ which can be estimated in a simliar way as we did for the case $p=1$. For the case that $p=3$, the integral scales like $l[\rho_0(\partial_t^4D\eta+\partial_t^3D\eta Dv+Dv_tDvDv)\partial_t^5D\eta+\rho_0(\partial_t^3D\eta+Dv_tDv+ DvDvDv)\partial_t^6D\eta]\rho_0\partial_t^7Dv$. The first part can be estimated by using $L^3$-$L^6$-$L^2$ H\"{o}lder's inequality, the second part can be estimated by using $L^{\infty}$-$L^2$-$L^2$ H\"{o}lder's inequality. The case $p=4,5$ is treated as the case $p=3$, the case $p=6$ is treated as the case $p=2$, and the case $p=7$ is treated as the case $p=1$.

It is also easy to see that
\begin{equation*}\int_0^TR_{12}(\tau)\,d\tau \leq C \int_0^T \|\rho_0 \partial_t^7Dv\|_0 d\tau \leq CT\sup_{[0,T]}\tilde{E}(t)
\end{equation*}

Now we estimate the integral $I_2$. By integrating by parts with respect to $x_k$, we have
\begin{align*}
I_2=&\kappa\int_{\Omega}\rho_0^2\big(J^{-1}{F^{-1}}_r^k[D_{\eta}\partial_t^8v]_r^i+J^{-1}{F^{-1}}_i^k \Div_{\eta}\partial_t^8v\big)\partial_t^8v^i_{,k}\,dx\\&+\underbrace{\kappa\int_{\Omega}\rho_0^2J^{-1}{F^{-1}}_r^k[\text{Curl}_{\eta}\partial_t^8v]_i^r\partial_t^8v^i_{,k}\,dx}_{I_{21}}\\
&-\int_{\Omega}\sum_{p=1}^{8}\partial_k\{\kappa\rho_0^2\bigg(\partial_t^p[J^{-1}{F^{-1}}_r^k{F^{-1}}_i^s]\partial_t^{9-p}\eta^r_{,s}
\\&\quad\quad+\partial_t^p[J^{-1}{F^{-1}}_i^k{F^{-1}}_r^s]\partial_t^{9-p}\eta^r_{,s}\bigg)\}\partial_t^8v^i\,dx
\\=&\int_{\Omega}\kappa\rho_0^2J^{-1}(|D_{\eta}\partial_t^8v|^2+|\Div_{\eta}\partial_t^8v|^2)\,dx+I_{21}+R_2,
\end{align*}
where
\begin{multline*}
R_2=-\int_{\Omega}\sum_{p=1}^{8}\partial_k\{\kappa\rho_0^2\bigg(\partial_t^p[J^{-1}{F^{-1}}_r^k{F^{-1}}_i^s]\partial_t^{9-p}\eta^r_{,s}
\\+\partial_t^p[J^{-1}{F^{-1}}_i^k{F^{-1}}_r^s]\partial_t^{9-p}\eta^r_{,s}\bigg)\}\partial_t^8v^i\,dx.
\end{multline*}
For $I_{21}$, we can get that

$$I_{21}=(\sum_{i>r}+\sum_{i<r})\int_{\Omega}\kappa\rho_0^2J^{-1}{F^{-1}}_r^k\partial_t^8v^i_{,k}({F^{-1}}_i^s\partial_t^7v^r_{,s}-{F^{-1}}_r^s\partial_t^7v^i_{,s})\,dx$$
\begin{align*}
=&\sum_{i>r}\int_{\Omega}\kappa\rho_0^2J^{-1}({F^{-1}}_r^k\partial_t^8v^i_{,k}-{F^{-1}}_i^k\partial_t^8v^r_{,k})({F^{-1}}_i^s\partial_t^7v^r_{,s}-{F^{-1}}_r^s\partial_t^7v^i_{,s})\,dx\\
=&-\int_{\Omega}\kappa\rho_0^2J^{-1}|\curl_{\eta}\partial_t^8v|^2\,dx-\dfrac{1}{2}\int_{\Omega}\rho_0^2J^{-2}J_t|\curl_{\eta}\partial_t^8v|^2\,dx
\\&+\underbrace{\sum_{i>r}\int_{\Omega}\rho_0^2J^{-1}(\partial_t{F^{-1}}_r^k\partial_t^8v^i_{,k}-\partial_t{F^{-1}}_i^k\partial_t^8v^r_{,k})[\text{Curl}_{\eta}\partial_t^8v]_i^r\,dx}_{R_{22}},
\end{align*}
and for $R_2$, we have
\begin{align}
\nonumber
R_2 &\leq \sum_{p=1}^9(2\|\rho_0\bigg(\partial_t^p[J^{-1}{F^{-1}}_r^k{F^{-1}}_i^s]\partial_t^{9-p}\eta^r_{,s}
+\partial_t^p[J^{-1}{F^{-1}}_i^k{F^{-1}}_r^s]\partial_t^{9-p}\eta^r_{,s}\bigg)\|_0^2\\\nonumber&\quad\quad+\dfrac{1}{2}\|\sqrt{\kappa}\rho_0\partial_t^8Dv\|_0^2)
\\&\leq CP(\tilde{E}(t))+\dfrac{1}{2}\|\sqrt{\kappa}\rho_0\partial_t^8Dv\|_0^2.
\label{restk}
\end{align}
Next, we estimate the term $I_3$ and $I_4$, which is related to the potential force. Since by H\"{o}lder's inequality, we can get $\int_{\Omega}\rho_0\partial_t^8G^i\partial_t^8v\,dx\leq \|\sqrt{\rho_0}\partial_t^8 v\|_0\|\sqrt{\rho_0}\partial_t^8G\|_0$, then with \eqref{gest}, we have
\begin{align*}
\int_0^T I_3+I_4\,dt \leq \tilde M_0+C\sqrt{T}P(\sup_{t\in[0,T]}\tilde E(t)).
\end{align*}
Combining all the estimates above, we see that the proposition is proved.
\end{proof}

\begin{corollary}[Estimates for $\partial_t^8\eta(t)$]
\label{et8}
\begin{equation*}
\sup_{0\in[0,T]}\|\partial_t^8\eta(t)\|_0\leq \sup_{t \in [0,T]}|\bar{\partial}^4\eta^{\alpha}(t)|_{-0.5}^2\leq \tilde{M}_0 +C\sqrt{T}P(\sup_{t\in[0,T]}\tilde{E}).
\end{equation*}
\end{corollary}
\begin{proof}
The weighted embedding estimate \eqref{embd} shows that
\begin{equation*}
\|\partial_t^8\eta(t)\|_0\leq C\int_{\Omega}\rho_0^2(|\partial_t^8\eta|^2+|\partial_t^8D\eta|^2)\,dx.
\end{equation*}
Applying the fundamental calculus theorem, we have
\begin{equation*}
\int_{\Omega}\rho_0^2|\partial_t^8\eta|^2\,dx \leq \tilde{M}_0+T^2\sup_{t\in[0,T]}\|\sqrt{\rho_0}\partial_t^8v\|_0^2
\end{equation*}
Thus, by Proposition \ref{tprop},
\begin{equation*}
\sup_{t\in[0,T]}\|\partial_t^8\eta(t)\|_0^2\leq \sup_{t \in [0,T]}|\bar{\partial}^4\eta^{\alpha}(t)|_{-0.5}^2\leq \tilde{M}_0 +C\sqrt{T}P(\sup_{t\in[0,T]}\tilde{E}).
\end{equation*}
\end{proof}
\subsubsection{The $\partial_t^2\bar{\partial}^3, \partial_t^4\bar{\partial}^2, \partial_t^6\bar{\partial}$ problems}
Since we have provided detailed proofs of the energy estimates for the two end-point cases, the $\bar{\partial}^4$-problem and the $\partial_t^8$-problem, we have covered all of the estimation strategies for all possible error terms in the three remaining intermediate problems. Meanwhile, the energy contributions for the three intermediate are found in the identical fashion as for the $\bar{\partial}^4$ and $\partial_t^8$ problems. As such we have the additional estimate
\begin{proposition}
\label{prop6}
For $\delta>0$ and letting $\tilde{M}_0$ depend on $1/\delta$, for $\alpha=1,2,$
\begin{align*}
\sup_{t\in[0,T]}\sum_{a=1}^3[|\partial_t^{2a}\eta^{\alpha}|^2_{3.5-a}+\|\sqrt{\rho_0}\bar{\partial}^{4-a}\partial_t^{2a}v(t)\|_0^2+\|\rho_0\bar{\partial}^{4-a}\partial_t^{2a}D\eta(t)\|_0^2]\\
+\kappa\int_0^t\|\rho_0\bar{\partial}^{4-a}\partial_t^{2a}Dv(s)\|_0^2\,ds]\leq\tilde{M}_0 +\delta\sup_{[0,T]}\tilde{E}+C\sqrt{T}P(\sup_{t\in[0,T]}\tilde{E}).
\end{align*}
\end{proposition}
\subsection{Additional Elliptic-Type Estimates for Normal Derivatives}
Our energy estimates in Section \ref{energyestimate} provide a priori control of tangential and time derivatives of $\eta$, it remains to gain a priori control of the normal derivatives of $\eta$ to close the argument. This is accomplished via bootstrapping procedure relying on having $\partial_t^8\eta(t)$ bounded in $L^2(\Omega)$.
\begin{proposition}
\label{6tp}
For $t\in[0,T]$, $\partial_t^5v(t)\in H^1(\Omega), \rho_0\Div\partial_t^6\eta\in H^1(\Omega)$, then
\begin{equation}
\sup_{[0,T]}(\|\partial_t^5v\|_1^2+\|\rho_0\partial_t^6J^{-2}\|_1^2) \leq \tilde{M}_0 +\delta\sup_{[0,T]}\tilde{E}+C\sqrt{T}P(\sup_{t\in[0,T]}\tilde{E}).
\end{equation}
\end{proposition}
\begin{proof}
We begin by taking six time derivatives of \eqref{lagc} to obtain that
\begin{equation*}
\kappa\partial_t^7[{F^{-1}}_i^k(2\rho_0J^{-1}-\Phi)_{,k}]+\partial_t[{F^{-1}}_i^k(2\rho_0J^{-1}-\Phi)_{,k}]=-\partial_t^7v^i.
\end{equation*}
With Lemma \ref{elli}, and the bound on $\|\partial_t^7v\|_0^2$ given by Corollary \ref{et8}, we have
\begin{equation}
\label{0o}
\sup_{[0,T]}\|\partial_t^6[{F^{-1}}_i^k(2\rho_0J^{-1}-\Phi)_{,k}]\|_0^2\leq \tilde{M}_0 +\delta\sup_{[0,T]}\tilde{E}+C\sqrt{T}P(\sup_{t\in[0,T]}\tilde{E}).
\end{equation}
For $\beta=1,2,$
\begin{equation}
\label{akl}
2{F^{-1}}_i^k(\rho_0J^{-1})_{,k}=\rho_0{F^{*}}_i^3{J^{-2}}_{,3}+2\rho_{0,3}{F^{*}}_i^3J^{-2}+\rho_0{F^{*}}_i^{\beta}{J^{-2}}_{,\beta}+2\rho_{0,\beta}{F^{*}}_i^{\beta}J^{-2}.
\end{equation}
Acting $\partial_t^6$ on equation \eqref{akl}, we have that
\begin{align*}
&\rho_0{F^{*}}_i^3\partial_t^6{J^{-2}}_{,3}+2\rho_{0,3}{F^{*}}_i^3\partial_t^6J^{-2}\\
=&\underbrace{\partial_t^6[{F^{-1}}_i^k(2\rho_0J^{-1}-\Phi)_{,k}]}_{J_1}+\underbrace{\partial_t^6G}_{J_2}-\underbrace{\rho_0\partial_t^6({F^{*}}_i^{\beta}{J^{-2}}_{,\beta})}_{J_3}-2\underbrace{\rho_{0,\beta}\partial_t^6({F^{*}}_i^{\beta}J^{-2})}_{J_4}
\\&-\underbrace{(\partial_t^6({F^{*}}_i^3)[\rho_0{J^{-2}}_{,3}+2{\rho_0}_{,3}J^{-2}]}_{J_5}+\underbrace{\sum_{a=1}^5c_a\partial_t^a({F^{*}}_i^3)\partial_t^{6-a}[\rho_0{J^{-2}}_{,3}+2{\rho_0}_{,3}J^{-2}]}_{J_6}\end{align*}
Firstly, the $L^2(\Omega)$ bound for $J_1$ is given by \eqref{0o}. For $J_2$, by the same argument we used in the proof of Proposition \ref{ineq:g}, we have
\begin{align*}
\|J_2\|_0^2\leq CP(\sup_{t\in[0,T]}\tilde E(t))\|\rho_0\partial_t^6D\eta\|_0^2\leq \tilde M_0+CTP(\sup_{t\in[0,T]}\tilde E(t)).
\end{align*}

\noindent According to Proposition \ref{prop6}, we have
\begin{equation}
\label{0t}
\sup_{[0,T]}(\|\sqrt{\rho_0}\partial_t^6v\|_0^2+\|\rho_0\bar{\partial}D\partial_t^5v\|_0^2)\leq \tilde{M}_0 +\delta\sup_{[0,T]}\tilde{E}+C\sqrt{T}P(\sup_{t\in[0,T]}\tilde{E}),
\end{equation}
so with \eqref{ini}, we see that
\begin{equation*}
\sup_{t\in[0,T]}\|J_3\|_0^2\leq\tilde{M}_0 +\delta\sup_{[0,T]}\tilde{E}+C\sqrt{T}P(\sup_{t\in[0,T]}\tilde{E}).
\end{equation*}
To estimate $J_4$, we need the following inequality for $\beta=1,2$
\begin{equation*}
\|\dfrac{{\rho_0}_{,\beta}}{\rho_0}\|_{L^{\infty}(\Omega)}\leq C\|\dfrac{{\rho_0}_{,\beta}}{\rho_0}\|_2\leq C\|{\rho_0}_{,\beta}\|_3.
\end{equation*}
where we used the Sobolev embedding theorem and Lemma \ref{hardy} for ${\rho_0}_{,\beta}\in H^3(\Omega)\cap {H}_0^1(\Omega), \beta=1,2$. Thus, we have that
\begin{align}
&\|2{\rho_0}_{,\beta}\partial_t^6({F^{-1}}_i^{\beta}J^{-1})\|_0^2\leq\|2\rho_0\partial_t^6({F^{*}}_i^{\beta}J^{-2})\|_0^2\|\dfrac{{\rho_0}_{,\beta}}{\rho_0}\|_{L^{\infty}(\Omega)}^2.
\label{hav}
\end{align}
Thanks to \eqref{0t} and \eqref{ini}, and the fact that $\|\rho_0\|_4$ is bounded by assumption, from which it follows that
\begin{equation*}
\sup_{t\in[0,T]}\|J_4(t)\|_0^2\leq\tilde{M}_0 +\delta\sup_{[0,T]}\tilde{E}+C\sqrt{T}P(\sup_{t\in[0,T]}\tilde{E}).
\end{equation*}

\noindent To estimate $J_5(t)$, we need to use the identity \eqref{ai3}, which shows that ${F^{*}}_i^3$ is quadratic in $\bar{\partial}\eta$, and in particular, depends only on tangential derivatives. From the estimate \eqref{0t} and the weighted embedding inequality \eqref{embd}, we infer that
\begin{equation*}
\sup_{t\in[0,T]}\|\bar{\partial}\partial_t^6\eta(t)\|_0^2 \leq  \tilde{M}_0 +\delta\sup_{[0,T]}\tilde{E}+C\sqrt{T}P(\sup_{t\in[0,T]}\tilde{E}).
\end{equation*}
Then we have $\sup_{[0,T]}\|J_5\|_0^2\leq  \tilde{M}_0 +\delta\sup_{[0,T]}\tilde{E}+C\sqrt{T}P(\sup_{t\in[0,T]}\tilde{E}).$

\noindent Last, each summand in $J_6$ is a lower-order term, such that the time derivative of each summand is controlled by the energy function $\tilde{E}(t)$. Then by the fundamental theorem of calculus, we have that
\begin{equation*}
\sup_{t\in[0,T]}\|J_6\|_0^2 \leq  \tilde{M}_0 +\delta\sup_{[0,T]}\tilde{E}+C\sqrt{T}P(\sup_{t\in[0,T]}\tilde{E}).
\end{equation*}

Now we have proved that for all $t\in[0,T]$,
\begin{equation*}
\|\rho_0{F^{*}}_i^3\partial_t^6{J^{-2}}_{,3}+2\rho_{0,3}{F^{*}}_i^3\partial_t^6J^{-2}\|_0^2\leq\tilde{M}_0 +\delta\sup_{[0,T]}\tilde{E}+C\sqrt{T}P(\sup_{t\in[0,T]}\tilde{E}),
\end{equation*}
and we will show that the $L^2(\Omega)$-norm of each summand on the left-hand side is uniformly bounded on $[0,T]$. To achieve this goal, we expand the $L^2(\Omega)$-norm to obtain the inequality
\begin{multline}
\label{popppp}
\|\rho_0|{F^{*}}_.^3|\partial_t^6{J^{-2}}_{,3}(t)\|_0^2+4\||{F^{*}}_.^3{\rho_0}_{,3}|\partial_t^6J^{-2}(t)\|_0^2\\
+4\int_{\Omega}\rho_0{\rho_0}_{,3}|{F^{*}}_.^3|^2\partial_t^6J^{-2}\partial_t^6{J^{-2}}_{,3}\,dx\leq\tilde{M}_0 +\delta\sup_{[0,T]}\tilde{E}+C\sqrt{T}P(\sup_{t\in[0,T]}\tilde{E}).
\end{multline}
For each $\kappa>0$, solutions to our degenerate parabolic approximation \eqref{approeq} have sufficient regularity to ensure that $\rho_0{|\partial_t^6J^{-2}|^2}_{,3}$ is integrable. So we can integrate-by-parts with respect to $x_3$ to find that
\begin{multline}
\label{iopi}
4\int_{\Omega}\rho_0{\rho_0}_{,3}|{F^{*}}_.^3|^2\partial_t^6J^{-2}\partial_t^6{J^{-2}}_{,3}\,dx\\
=-2\||{F^{*}}_.^3{\rho_0}_{,3}|\partial_t^6J^{-2}(t)\|_0^2-2\int_{\Omega}\rho_0({\rho_0}_{,3}|{F^{*}}_.^3|^2)_{,3}(\partial_t^6J^{-2})^2\,dx.
\end{multline}
Substituting \eqref{iopi} into \eqref{popppp}, we can get
\begin{multline}
\label{tio2}
\|\rho_0|{F^{*}}_.^3|\partial_t^6{J^{-2}}_{,3}(t)\|_0^2+2\||{F^{*}}_.^3{\rho_0}_{,3}|\partial_t^6J^{-2}(t)\|_0^2\\
\leq\tilde{M}_0 +\delta\sup_{[0,T]}\tilde{E}+C\sqrt{T}P(\sup_{t\in[0,T]}\tilde{E}).
\end{multline}
Using \eqref{ini2}, we see that $|{F^{*}}_.^3|^2$ has a strictly positive lower-bound. By the physical vacuum condition \eqref{vacuum2}, for $\epsilon>0$ taken sufficiently small, there constants $\theta_1, \theta_2>0$ such that $|{\rho_0}_{,3}|\geq \theta_1$ whenever $1-\epsilon\leq x_3\leq 1$ and $0\leq x_3\leq\epsilon$, and $\rho_0(x)>\theta_2$ whenever $\epsilon\leq x_3\leq 1-\epsilon$. Hence, by readjusting the constants on the right-hand side of \eqref{tio2}, we find that
\begin{multline}
\|\rho_0\partial_t^6{J^{-2}}_{,3}\|_0^2+2\|\partial_t^6J^{-2}\|_0^2
\\\leq\tilde{M}_0 +\delta\sup_{[0,T]}\tilde{E}+C\sqrt{T}P(\sup_{t\in[0,T]}\tilde{E})+C\int_{\Omega}\rho_0|\partial_t^6J^{-2}|^2\,dx.
\label{pm}
\end{multline}
By Proposition \ref{prop6}, for $\beta=1,2$,
\begin{equation*}
\sup_{t\in[0,T]}\|\rho_0\partial_t^6{J^{-2}}_{,\beta}\|_0\leq\tilde{M}_0 +\delta\sup_{[0,T]}\tilde{E}+C\sqrt{T}P(\sup_{t\in[0,T]}\tilde{E}).
\end{equation*}
and by the fundamental theorem of calculus and Proposition \ref{tprop},
\begin{equation*}
\sup_{t\in[0,T]}\|\rho_0\partial_t^6J^{-2}\|_0\leq\tilde{M}_0 +\delta\sup_{[0,T]}\tilde{E}+C\sqrt{T}P(\sup_{t\in[0,T]}\tilde{E}).
\end{equation*}
These two inequalities, combined with \eqref{pm}, imply that
\begin{equation*}
\|\rho_0\partial_t^6J^{-2}\|_1^2+\|\partial_t^6J^{-2}\|_0^2\leq \tilde{M}_0 +\delta\sup_{[0,T]}\tilde{E}+C\sqrt{T}P(\sup_{t\in[0,T]}\tilde{E})+C\int_{\Omega}\rho_0|\partial_t^6J^{-2}|^2\,dx.
\end{equation*}
We use Young's inequality and the fundamental theorem of calculus (with respect to $t$) for the last integral, then we find that for $\theta >0$,
\begin{align*}
\int_{\Omega}\rho_0\partial_t^6J^{-2}\partial_t^6J^{-2}dx&\leq \theta\|\partial_t^6J^{-2}(t)\|_0^2+C_{\theta}\|\rho_0\partial_t^6J^{-2}(t)\|_0^2\\
&\leq\theta\|\partial_t^6J^{-2}(t)\|_0^2+C_{\theta}\|\rho_0\partial_t^5Dv\|_0^2\\
&\leq\theta\|\partial_t^6J^{-2}(t)\|_0^2+\tilde{M}_0 +\delta\sup_{[0,T]}\tilde{E}+C\sqrt{T}P(\sup_{t\in[0,T]}\tilde{E}),
\end{align*}
where we have used the fact that $\|\rho_0\partial_t^7Dv(t)\|_0^2$ is contained in the energy function $\tilde{E}(t)$. We choose $\theta\ll1$ and once again readjust the constants, then we see that on $[0,T]$
\begin{equation}
\label{xcv}
\|\rho_0\partial_t^6J^{-2}\|_1^2+\|\partial_t^6J^{-2}\|_0^2\leq \tilde{M}_0 +\delta\sup_{[0,T]}\tilde{E}+C\sqrt{T}P(\sup_{t\in[0,T]}\tilde{E}).
\end{equation}
With $J_t={F^{*}}_i^jv^i_{,j}$, we see that
\begin{equation}
\label{bcz}
{F^{*}}_i^j\partial_t^5v^i_{,j}=\partial_t^tJ-v^i_{j}\partial_t^5({F^{*}}_i^j)-\sum_{a=1}^4c_a\partial_t^a\partial_t^{5-a}v^i_{,j}.
\end{equation}
Then by using \eqref{xcv} and applying the fundamental theorem of calculus to the last two terms on the right-hand side of \eqref{bcz}, we see that
\begin{equation*}
\|{F^{*}}_i^j\partial_t^5v^i_{,j}\|_0^2\leq \tilde{M}_0 +\delta\sup_{[0,T]}\tilde{E}+C\sqrt{T}P(\sup_{t\in[0,T]}\tilde{E}),
\end{equation*}
from which it follows that
\begin{equation*}
\|\Div\partial_t^5v(t)\|_0^2\leq \tilde{M}_0 +\delta\sup_{[0,T]}\tilde{E}+C\sqrt{T}P(\sup_{t\in[0,T]}\tilde{E}).
\end{equation*}
Since Proposition \ref{propc} provides the estimate
\begin{equation*}
\|\curl\partial_t^5v(t)\|_0^2\leq \tilde{M}_0 +\delta\sup_{[0,T]}\tilde{E}+C\sqrt{T}P(\sup_{t\in[0,T]}\tilde{E}),
\end{equation*}
and Proposition \ref{prop6} shows that for $\alpha=1,2$,
\begin{equation*}
|\partial_t^5v^{\alpha}(t)|_{0.5}^2\leq \tilde{M}_0 +\delta\sup_{[0,T]}\tilde{E}+C\sqrt{T}P(\sup_{t\in[0,T]}\tilde{E}).
\end{equation*}
We thus conclude from Proposition \ref{curllemma} that
\begin{equation*}
\sup_{t\in[0,T]}\|\partial_t^5v(t)\|_1^2\leq \tilde{M}_0 +\delta\sup_{[0,T]}\tilde{E}+C\sqrt{T}P(\sup_{t\in[0,T]}\tilde{E}).
\end{equation*}
\end{proof}
After getting a good bound for $\partial_t^5v(t)$ in $H^1(\Omega)$, we proceed with our bootstrapping.
\begin{proposition}
\label{0908}
For $t\in[0,T]$, $\partial_t^3v\in H^2(\Omega)$, $\rho_0\partial_t^4J^{-2}(t)\in H^2(\Omega)$ and
\begin{equation}
\sup_{t\in[0,T]}(\|\partial_t^3v\|_2^2+\|\rho_0\partial_t^4J^{-2}\|_2^2)\leq\tilde{M}_0 +\delta\sup_{[0,T]}\tilde{E}+C\sqrt{T}P(\sup_{t\in[0,T]}\tilde{E}).
\end{equation}
\end{proposition}
\begin{proof}
We take four time-derivatives of \eqref{lagc} to obtain
\begin{equation*}
\kappa\partial_t^5[{F^{-1}}_i^k(2\rho_0J^{-1}-\Phi)_{,k}]+\partial_t^4[{F^{-1}}_i^k(2\rho_0J^{-1}-\Phi)_{,k}]=-\partial_t^5v^i.
\end{equation*}
With Lemma \ref{elli}, and the bound on $\|\partial_t^5v\|_1^2$ given by Proposition \ref{6tp}, we have
\begin{equation}
\label{mbx}
\sup_{t\in[0,T]}\|\partial_t^4[{F^{-1}}_i^k(2\rho_0J^{-1}-\Phi)_{,k}]\|_1^2\leq\tilde{M}_0 +\delta\sup_{[0,T]}\tilde{E}+C\sqrt{T}P(\sup_{t\in[0,T]}\tilde{E}).
\end{equation}
Acting $\partial_t^4$ on equation \eqref{akl}, we have that
\begin{align}
\nonumber&\rho_0{F^{*}}_i^3\partial_t^4{J^{-2}}_{,3}+2\rho_{0,3}{F^{*}}_i^3\partial_t^4J^{-2}\\
\nonumber=&\underbrace{\partial_t^4[{F^{-1}}_i^k(2\rho_0J^{-1}-\Phi)_{,k}]}_{J_1}+\underbrace{\partial_t^4G}_{J_2}-\underbrace{\rho_0\partial_t^4({F^{*}}_i^{\beta}{J^{-2}}_{,\beta})}_{J_3}-2\underbrace{\rho_{0,\beta}\partial_t^4({F^{*}}_i^{\beta}J^{-2})}_{J_4}
\\&-\underbrace{\partial_t^4({F^{*}}_i^3)[\rho_0{J^{-2}}_{,3}+2{\rho_0}_{,3}J^{-2}]}_{J_5}+\underbrace{\sum_{a=1}^3c_a\partial_t^a({F^{*}}_i^3)\partial_t^{4-a}[\rho_0{J^{-2}}_{,3}+2{\rho_0}_{,3}J^{-2}]}_{J_6}
\label{zcvv}
\end{align}
In order to estimate $\partial_t^4J^{-2}(t)$ in $H^1(\Omega)$, we first estimate tangential derivatives of $\partial_t^4J^{-2}(t)$ in $L^2(\Omega)$. We consider for $\alpha=1,2,$
\begin{multline}
\label{vnzl}
\rho_0{F^{*}}_i^3\partial_t^4{J^{-2}}_{,3\alpha}+2{\rho_0}_{,3}{F^{*}}_i^3\partial_t^4{J^{-2}}_{,\alpha}\\
=\sum_{l=1}^6J_{l,\alpha}-(\rho_0{F^{*}}_{,\alpha}\partial_t^4{J^{-2}}_{,3})-2({\rho_0}_{,3}{F^{*}}_i^3)_{,\alpha}\partial_t^4J^{-2}.
\end{multline}
\textbf{Bounds for ${J_1}_{,\alpha}$.} The estimate \eqref{mbx} shows that
\begin{equation*}
\|{J_1}_{,\alpha}\|_0^2\leq \tilde{M}_0 +\delta\sup_{[0,T]}\tilde{E}+C\sqrt{T}P(\sup_{t\in[0,T]}\tilde{E}).
\end{equation*}
\textbf{Bounds for ${J_2}_{,\alpha}$.} From \eqref{gest}, we have that
\begin{equation*}
\|{J_2}_{,\alpha}\|_0^2\leq \tilde{M}_0 +\delta\sup_{[0,T]}\tilde{E}+C\sqrt{T}P(\sup_{t\in[0,T]}\tilde{E}).
\end{equation*}
\textbf{Bounds for ${J_3}_{,\alpha}$.} First, Proposition \ref{prop6} provides the estimate
\begin{equation}
\label{90394}
\sup_{t\in[0,T]}(\|\bar{\partial}^2\partial_t^3v\|_0^2+\|\rho_0\bar{\partial}^2D\partial_t^3v\|_0^2)\leq \tilde{M}_0 +\delta\sup_{[0,T]}\tilde{E}+C\sqrt{T}P(\sup_{t\in[0,T]}\tilde{E}).
\end{equation}
Then we expand ${J_3}_{,\alpha}$ as
\begin{equation*}
{J_3}_{,\alpha}=\rho_0\partial_t^4({{F^{*}}_i^{\beta}}_{,\alpha}{J^{-2}}_{,\beta}+{F^{*}}_i^{\beta}{J^{-2}}_{,\beta\alpha}+{\rho_0}_{,\alpha}\partial_t^4({F^{*}}_i^{\beta}{J^{-2}}_{,\beta}).
\end{equation*}
Using \eqref{ini}, for $\alpha=1,2,$ the highest-order term in $\rho_0\partial_t^4({F^{*}}_i^{\beta}{J^{-2}}_{,\beta\alpha})$ satisfies the inequality
\begin{equation*}
\|\rho_0{F^{*}}_i^{\beta}\partial_t^4{J^{-2}}_{,\beta\alpha}\|_0^2 \leq C\|\rho_0\bar{\partial}^2D\partial_t^3v\|_0^2,
\end{equation*}
which has the bound \eqref{90394}, and the lower-order terms have the same bound by using the fundamental theorem of calculus. For instance,
\begin{multline*}
\|\rho_0{J^{-2}}_{,\beta\alpha}\partial_t^4({F^{*}}_i^{\beta}){J^{-2}}_{,\beta\alpha}\|_0^2\leq\|\rho_0{J^{-2}}_{,\beta\alpha}\|_{L^6{\Omega}}\|\partial_t^4({F^{*}}_i^{\beta})\|_{L^3(\Omega)}\\
\leq C\|\rho_0{J^{-2}}_{,\alpha\beta}\|_1\|\partial_t^4({F^{*}}_i^{\beta})\|_{0.5}\leq \tilde{M}_0,
\end{multline*}
where we have used H\"older's inequality, the Sobolev embedding theorem, and \eqref{ini} for the final inequality. On the other hand, ${\rho_0}_{,\alpha}\partial_t^4({F^{*}}_i^{\beta}{J^{-2}}_{,\beta})$ is estimated in the same manner as \eqref{hav}, which shows that
\begin{equation*}
\|{J_3}_{,\alpha}\|_0^2 \leq\tilde{M}_0 +\delta\sup_{[0,T]}\tilde{E}+C\sqrt{T}P(\sup_{t\in[0,T]}\tilde{E}).
\end{equation*}
\textbf{Bounds for ${J_4}_{,\alpha}$.} By using the fact that $\|\partial_t{J_3}_{,\alpha}\|_0^2$ can be bounded by the energy function and applying the fundamental theorem of calculus, we have that
\begin{equation*}
\|{J_4}_{,\alpha}\|_0^2 \leq\tilde{M}_0 +\delta\sup_{[0,T]}\tilde{E}+C\sqrt{T}P(\sup_{t\in[0,T]}\tilde{E}).
\end{equation*}
\textbf{Bounds for ${J_5}_{,\alpha}$.} Again, by using the fact that from \eqref{ai3}, the vector ${F^{*}}_i^3$ only contains tangential derivatives of $\eta^i$, and the inequalities \eqref{ini}, we have that for $\alpha=1,2$,
\begin{align*}
\|(\partial_t^4({F^{*}}_i^3)[\rho_0{J^{-2}}_{,3}+2{\rho_0}_{,3}J^{-2}])_{,\alpha}\|_0^2&\leq C\|\bar{\partial}^2\partial_t^3v\|_0^2+\tilde{M}_0\\&\leq\tilde{M}_0 +\delta\sup_{[0,T]}\tilde{E}+C\sqrt{T}P(\sup_{t\in[0,T]}\tilde{E}),
\end{align*}
where the last inequality is followed from \eqref{90394}. Then we obtain that
\begin{equation*}
\|{J_5}_{,\alpha}\|_0^2 \leq\tilde{M}_0 +\delta\sup_{[0,T]}\tilde{E}+C\sqrt{T}P(\sup_{t\in[0,T]}\tilde{E}).
\end{equation*}
\textbf{Bounds for ${J_6}_{,\alpha}$.} These are lower-order terms, which can be estimated with the fundamental theorem of calculus and \eqref{ini}, and have the following bound
\begin{equation*}
\|{J_6}_{,\alpha}\|_0^2 \leq\tilde{M}_0 +\delta\sup_{[0,T]}\tilde{E}+C\sqrt{T}P(\sup_{t\in[0,T]}\tilde{E}).
\end{equation*}
\textbf{Bounds for $-(\rho_0{F^{*}}_i^3)_{,\alpha}\partial_t^4{J^{-2}}_{,3}-2({\rho_0}_{,3}{F^{*}}_i^3)_{,\alpha}\partial_t^4J^{-2}$.} The bounds for these terms can be estimated by using the same fashion as we used for ${J_3}_{,\alpha}$ and we have that
\begin{equation*}
\|-(\rho_0{F^{*}}_i^3)_{,\alpha}\partial_t^4{J^{-2}}_{,3}-2({\rho_0}_{,3}{F^{*}}_i^3)_{,\alpha}\partial_t^4J^{-2}\|_0^2 \leq\tilde{M}_0 +\delta\sup_{[0,T]}\tilde{E}+C\sqrt{T}P(\sup_{t\in[0,T]}\tilde{E}).
\end{equation*}
We have hence bounded the $L^2(\Omega)$-norm of the right-hand side of \eqref{vnzl} by $\tilde{M}_0 +\delta\sup_{[0,T]}\tilde{E}+C\sqrt{T}P(\sup_{t\in[0,T]}\tilde{E}).$ Using the same integration-by-parts argument just given in the proof of Proposition \ref{6tp}, we conclude that for $\alpha=1,2,$
\begin{equation}
\sup_{t\in[0,T]}(\|\partial_t^4{J^{-2}}_{,\alpha}\|_0^2+\|\rho_0\partial_t^4{J^{-2}}_{,\alpha}\|_1^2)\leq\tilde{M}_0 +\delta\sup_{[0,T]}\tilde{E}+C\sqrt{T}P(\sup_{t\in[0,T]}\tilde{E}).
\label{znm}
\end{equation}
From the inequality \eqref{znm}, we can infer that for $\alpha=1,2,$
\begin{equation}
\label{m1}
\sup_{t\in[0,T]}\|\Div \partial_t^3v_{,\alpha}\|_0^2\leq \tilde{M}_0 +\delta\sup_{[0,T]}\tilde{E}+C\sqrt{T}P(\sup_{t\in[0,T]}\tilde{E}),
\end{equation}
and according to Proposition \ref{propc}, for $\alpha=1,2,$
\begin{equation}
\label{m2}
\sup_{t\in[0,T]}\|\curl \partial_t^3v_{,\alpha}\|_0^2\leq \tilde{M}_0 +\delta\sup_{[0,T]}\tilde{E}+C\sqrt{T}P(\sup_{t\in[0,T]}\tilde{E}).
\end{equation}
The boundary regularity of $\partial_t^3v_{,\alpha}, \alpha=1,2$, follows from Proposition \ref{prop6}:
\begin{equation}
\label{m3}
\sup_{t\in[0,T]}|\partial_t^3v_{,\alpha}|_{0.5}^2\leq \tilde{M}_0 +\delta\sup_{[0,T]}\tilde{E}+C\sqrt{T}P(\sup_{t\in[0,T]}\tilde{E}),
\end{equation}
Thus, the inequalities \eqref{m1}, \eqref{m2}, and \eqref{m3} together with \eqref{curlineq} and \eqref{znm} show that
\begin{equation}
\sup_{t\in[0,T]}(\|\partial_t^3v_{,\alpha}\|_1^2+\|\rho_0\partial_t^4{J^{-2}}_{,\alpha}\|_1^2)\leq  \tilde{M}_0 +\delta\sup_{[0,T]}\tilde{E}+C\sqrt{T}P(\sup_{t\in[0,T]}\tilde{E}),
\label{dff}
\end{equation}
In order to estimate $\|\partial_t^4{J^{-2}}_{,3}\|_0^2$, we differentiate \eqref{zcvv} in the normal direction $x_3$ to obtain
\begin{multline}
\label{vnzl2}
\rho_0{F^{*}}_i^3\partial_t^4{J^{-2}}_{,33}+2{\rho_0}_{,3}{F^{*}}_i^3\partial_t^4{J^{-2}}_{,3}\\
=\sum_{l=1}^6J_{l,3}-(\rho_0{F^{*}}_{i,3}^3\partial_t^4{J^{-2}}_{,3})-2({\rho_0}_{,3}{F^{*}}_i^3)_{,3}\partial_t^4J^{-2}.
\end{multline}
For ${J_2}_{,3}$, we have
\begin{align*}
&\quad\partial_3\partial_t^4G^i\\&=\sum_{a=0}^{4}\int_{\Omega}\partial_3(\partial_t^a\dfrac{1}{|\eta(x,t)-\eta(z,t)|}\partial_k(\rho_0\partial_t^{4-a}{F^{-1}}_i^k)\,dx
\\&=C\int_{\Omega}\dfrac{(\eta(x,t)-\eta(z,t))\cdot\partial_3\eta(x,t)}{|\eta(x,t)-\eta(z,t)|^3}\partial_k(\rho_0\partial_t^4{F^{-1}}_i^k)\,dz\\&\quad+\int_{\Omega}\dfrac{(\partial_t^3v(x,t)-\partial_t^3v(z,t))\cdot\partial_3\eta(x,t)}{|\eta(x,t)-\eta(z,t)|^3}\partial_k(\rho_0{F^{-1}}_i^k)\,dz\\
&\quad+C\int_{\Omega}\dfrac{(\eta(x,t)-\eta(z,t))\cdot(\partial_t^3v(x,t)-\partial_t^3v(z,t))((\eta(x,t)-\eta(z,t))\cdot\partial_3\eta(x,t))}{|\eta(x,t)-\eta(z,t)|^3}\\&\quad\quad\quad\quad\quad\times\partial_k(\rho_0\partial_t^4{F^{-1}}_i^k)\,dz+R.
\end{align*}
By using a similar argument we used in the proof of Proposition \eqref{ineq:g}, for instance, from \eqref{calpha} to \eqref{tio}, and combining with \eqref{ini2}, we can bound $\|{J_2}_{,3}\|_0^2$ by $\tilde{M}_0+CTP(\sup_{t\in[0,T]}\tilde{E})$.

Then following our estimates for the tangential derivatives, inequality \eqref{dff} together with Proposition \ref{prop6} and \ref{6tp} show that the other terms on the right-hand side of \eqref{vnzl2} are bounded in $L^2(\Omega)$ by $\tilde{M}_0 +\delta\sup\limits_{[0,T]}\tilde{E}+C\sqrt{T}P(\sup\limits_{t\in[0,T]}\tilde{E})$.

\noindent It follows that for $k=1,2,3,$
\begin{equation*}
\|\rho_0{F^{*}}_i^3\partial_t^4{J^{-2}}_{,k3}+3{\rho_0}_{,3}{F^{*}}_i^3\partial_t^4{J^{-2}}_{,k}\|\leq \tilde{M}_0 +\delta\sup\limits_{[0,T]}\tilde{E}+C\sqrt{T}P(\sup\limits_{t\in[0,T]}\tilde{E}).
\end{equation*}
Note that the coefficient in front of ${\rho_0}_{,3}\rho_0{F^{*}}_i^3\partial_t^4{J^{-2}}_{,k}$ has changed from 2 to 3, but the identical integration-by-parts argument which we used in the proof of Proposition \ref{6tp} can be employed again, and shows that
\begin{equation*}
\|\rho_0\partial_t^4J^{-2}\|_2^2+\|\partial_t^4J^{-2}\|_1^2\leq \tilde{M}_0 +\delta\sup_{[0,T]}\tilde{E}+C\sqrt{T}P(\sup_{t\in[0,T]}\tilde{E}).
\end{equation*}
Thus, $\|\Div\partial_t^3v\|_1^2\leq\tilde{M}_0 +\delta\sup_{[0,T]}\tilde{E}+C\sqrt{T}P(\sup_{t\in[0,T]}\tilde{E})$. From Proposition \ref{propc}, we have $\|\curl\partial_t^3v\|_1^2\leq\tilde{M}_0 +\delta\sup_{[0,T]}\tilde{E}+C\sqrt{T}P(\sup_{t\in[0,T]}\tilde{E})$. Thus, combining these two estimates with the bound on $\partial_t^3v^{\alpha}$ given by Proposition \ref{prop6}, we can get the following estimate by Proposition \ref{curllemma}:
\begin{equation*}
\|\partial_t^3v\|_2^2\leq\tilde{M}_0 +\delta\sup_{[0,T]}\tilde{E}+C\sqrt{T}P(\sup_{t\in[0,T]}\tilde{E}).
\end{equation*}
\end{proof}
\begin{proposition}
For $t\in[0,T]$, $\partial_tv\in H^3(\Omega)$, $\rho_0\partial_t^2J^{-2}(t)\in H^3(\Omega)$ and
\begin{equation}
\sup_{t\in[0,T]}(\|\partial_tv\|_3^2+\|\rho_0\partial_t^2J^{-2}\|_3^2)\leq\tilde{M}_0 +\delta\sup_{[0,T]}\tilde{E}+C\sqrt{T}P(\sup_{t\in[0,T]}\tilde{E}).
\end{equation}
\end{proposition}
\begin{proof}
By taking two time-derivatives of \eqref{lagc} and applying the same argument that we used in the proof of Proposition \ref{0908}, the proposition is proved.
\end{proof}

\begin{proposition}
\label{notp}
For $t\in[0,T]$, $\eta\in H^4(\Omega)$, $\rho_0J^{-2}(t)\in H^4(\Omega)$ and
\begin{equation}
\sup_{t\in[0,T]}(\|\eta\|_4^2+\|\rho_0J^{-2}\|_4^2)\leq\tilde{M}_0 +\delta\sup_{[0,T]}\tilde{E}+C\sqrt{T}P(\sup_{t\in[0,T]}\tilde{E}).
\end{equation}
\end{proposition}
\begin{proof}
From \eqref{lagc}, we can use the same argument that we used in the proof of Proposition \ref{0908} again to conclude the proof.
\end{proof}

Now we only have the last two terms of $\tilde{E}$ to estimate.
\begin{proposition}
\label{lastp}
\begin{equation*}
\sup_{t\in[0,T]}(\|\curl_{\eta}v\|_3^2+\|\rho_0\bar{\partial}^4\curl_{\eta}v\|_0^2)\leq \tilde{M}_0 +\delta\sup_{[0,T]}\tilde{E}+C\sqrt{T}P(\sup_{t\in[0,T]}\tilde{E}).
\end{equation*}
\end{proposition}
\begin{proof}
Acting $D^3$ on the identity \eqref{c2} for $\curl_{\eta}v$, we see that the highest-order term scales like
\begin{equation*}
D^3\curl u_0+\int_0^t D^4vDv F^{-1}F^{-1}\,dt'.
\end{equation*}
Integration-by-parts with $t$ shows that the highest-order contribution to the term $D^3\curl_{\eta}v$ can be written as
\begin{equation*}
D^3\curl u_0+\int_0^t D^4\eta [Dv F^{-1}F^{-1}]_t\,dt'+ D^4\eta(t)Dv(t)F^{-1}(t)F^{-1}(t),
\end{equation*}
which, according to Proposition \ref{notp}, has $L^2(\Omega)$-norm bounded by $\tilde{M}_0 +\delta\sup\limits_{[0,T]}\tilde{E}+C\sqrt{T}P(\sup\limits_{t\in[0,T]}\tilde{E})$, after readjusting the constants; thus the inequality for the $H^3(\Omega)$-norm of $\curl_{\eta}v$ is proved.

The same type of analysis works for the weighted estimate. After integration by parts with $t$, the highest-order term in the expression for $\rho_0\bar{\partial}^4\curl_{\eta}v$ scales like
\begin{equation*}
\rho_0\bar{\partial}^4\curl u_0+\int_0^t \rho_0\bar{\partial}^4D\eta [Dv F^{-1}F^{-1}]_t\,dt'+ \rho_0\bar{\partial}^4D\eta(t)Dv(t)F^{-1}(t)F^{-1}(t).
\end{equation*}
Hence, the inequality \eqref{p4prop} shows that the weighted estimate holds as well.

\end{proof}

\section{Proof of Theorem \ref{theorem}}
In this section, we prove the Theorem \ref{theorem}, the arguments are similar to \cite{DS_2010}, we include them for self-contained presentation and completeness.
\label{s7}
\subsection{Time internal of existence and bounds independent of $\kappa$.}
Combining the estimates from Proposition \ref{propc} to \ref{lastp} and Corollary \ref{et8}, we obtain the following inequality on $(0,T_{\kappa})$
\begin{equation*}
\sup_{t\in[0,T]}\tilde{E}(t)\leq \tilde{M}_0 +\delta\sup_{[0,T]}\tilde{E}+C\sqrt{T}P(\sup_{t\in[0,T]}\tilde{E}).
\end{equation*}
By choosing $\delta$ sufficiently small, we have
\begin{equation*}
\sup_{t\in[0,T]}\tilde{E}(t)\leq \tilde{M}_0 +C\sqrt{T}P(\sup_{t\in[0,T]}\tilde{E}).
\end{equation*}
By using the continuation argument, this inequality provides us with a time of existence $T_1$ independent of $\kappa$ and an estimate on $(0,T_1)$ independent of $\kappa$ of the type:
\begin{equation}
\label{bound}
\sup_{t\in[0,T_1]}\tilde{E}(t)\leq 2\tilde{M}_0,
\end{equation}
as long as conditions \eqref{ini2} and \eqref{ini} hold. These conditions now can be verified by using the fundamental theorem of calculus and further shrinking the time interval, if necessary. In particular, our sequence of solutions $\{\eta^{\kappa}\}_{\kappa>0}$ to our approximate $\kappa$-problem \eqref{approeq} satisfy the $\kappa$-independent bound \eqref{bound} on the $\kappa$-independent time interval $(0,T_1)$.
\subsection{The limit as $\kappa\rightarrow 0$ and Existence of solutions to equations (\ref{eq:lag})}
By the $\kappa$-independent bound \eqref{bound}, standard compactness arguments provide the existence of a strongly convergent subsequences for $\kappa'>0$
\begin{align*}
&\eta^{\kappa'}\rightarrow \eta \quad\,\,\text{in}\quad L^2((0,T_1);H^3(\Omega)),\\
&v_t^{\kappa'}\rightarrow v_t \quad\text{in}\quad L^2((0,T_1);H^2(\Omega)).
\end{align*}
Consider the variational form of \eqref{lagappro}: for all $\psi\in L^2(0,T_1;H^1(\Omega))$,
\begin{multline*}
\int_0^{T_1}\bigg[\int_{\Omega}\rho_0(v^{\kappa'})_t^i\psi^i\,dx-\int_{\Omega}\rho_0^2(J^{\kappa'})^{-2}({{F^{-1}}^{\kappa'}}_i^k)\psi^i_{,k}\,dx\\
-\kappa\int_{\Omega}\rho_0^2\partial_t[(J^{\kappa'})^{-2}({{F^{-1}}^{\kappa'}}_i^k)]\psi^i_{,k}\,dx-\int_{\Omega}\rho_0[(G^{\kappa'})^i+\kappa\partial_t(G^{\kappa'})^i]\psi^i\,dx\bigg]=0.
\end{multline*}
The strong convergence of the sequences $(\eta^{\kappa'},v_t^{\kappa'})$ shows that the limit $(\eta,v_t)$ satisfies
\begin{equation*}
\int_0^{T_1}\bigg[\int_{\Omega}\rho_0v_t^i\psi^i\,dx-\int_{\Omega}\rho_0^2J^{-2}{F^{-1}}_i^k\psi^i_{,k}\,dx
-\int_{\Omega}\rho_0G^i\psi^i\,dx\bigg]=0.
\end{equation*}
This means that $\eta$ is a solution to \eqref{lag} on the $\kappa$-independent time interval $(0,T_1)$. A standard arguments also shows that $v(0)=u_0$ and $\eta(0)=e$.

For the uniqueness of solutions,  we can prove it by a similar fashion as we did in Section \ref{ee}. The reader can see \cite{DS_2010} for details. Moreover, in \cite{Xin_2013}, the authors proved a novelty uniqueness result for $W^{1,\infty}$ class solutions to both Euler and Euler--Poisson equations under the Eulerian coordinates. They used the relative entropy method to avoid complicated high order energy estimate. However, by using this method, the uniqueness for the Euler--Poisson equations was obtained only for the case $1<\gamma\leq 2$.

\subsection{Optimal Regularity for initial data}
\label{s6}
For the purposes of constructing solutions to our degenerate parabolic $\kappa$-problem \eqref{approeq}, we smoothed our initial data such that both our initial velocity field $u_0^{\kappa}$ is smooth, and our initial density $\rho_0^{\kappa}$ is smooth, positive in the interior, and vanishing on the boundary with \eqref{vacuum2}. Then our a prior estimates allow us to pass the limit $\lim\limits_{\kappa\rightarrow 0}u_0^{\kappa}$ and $\lim\limits_{\kappa\rightarrow 0}\rho_0^{\kappa}=\rho_0$. Recall that our smooth parameter is small compared to viscosity, then by our smoothing construction, $\rho_0\in H^4(\Omega)$, satisfies $\rho_0>0$ in $\Omega$ and the physical vacuum condition \eqref{vacuum2} on the boundary $\Gamma$. Similarly, the initial velocity field need only satisfy $E(0)<\infty$.

\section{The general case for $1 <\gamma< 3$}
\label{ga3}
The general Euler--Poisson equations can be written in the Lagrangian coordinates similarly as \eqref{eq:lag}:
\begin{subequations}
\label{eq:laggamma}
\begin{align}
  \label{laggamma}
  \rho_0 v_t + {F^{*}}_i^k\partial_k(\rho_0^{\gamma}J^{-\gamma})  &= \rho_0 {F^{-1}}_i^k\partial_k(\Phi) &\text{in}&\ \  \Omega\times (0,T],\\ \label{cffffffgamma}
  - {F^{*}}_i^j\partial_j({F^{-1}}_i^k\partial_k(\Phi))  &=  \rho_0 &\text{in}&\ \  \Omega \times (0,T], \\ \label{fieldgamma}
(v, \eta) &= (u_0, e) &\text{in}&\ \ \Omega\times \{t=0\},\\
  \rho_0 &= 0 &\text{on}&\ \  \Gamma,
\end{align}
\end{subequations}
If $\gamma \neq 2$, we set $\omega_0=\rho_0^{\gamma-1}$, then physical vacuum condition shows that
\begin{equation}
  \omega_0 \geq C \text{dist}(x, \partial \Omega)
\label{w0000}
\end{equation}
when $x \in \Omega$ is close to the vacuum boundary $\Gamma$, and we also have
\begin{align}
  \bigg |\dfrac{\partial\omega_0}{\partial N}(x)\bigg |\geq C\ \ \text{when}  \ d(x,\partial \Omega) \leq \alpha,
  \label{cond_2}\\
  \omega_0 \geq C_{\alpha} > 0\ \ \text{when}  \ d(x,\partial \Omega) \geq \alpha.
  \label{cond_3}
\end{align}
Now it is reasonable to suppose that $\omega_0$ is smooth and define $\omega_0J^{1-\gamma}\Div_{\eta} v$ as intermediate variable $X$. In this case, the approximate parabolic equations can be written as
\begin{equation*}
\rho_0 v_t + {F^{*}}_i^k\partial_k(\rho_0\omega_0J^{-\gamma})+\kappa\partial_t[{F^{*}}_i^k\partial_k(\rho_0\omega_0J^{-\gamma}) ]= \rho_0 G^i+\kappa\rho_0\partial_tG^i \,\,\text{in}\,\,  \Omega\times (0,T],
\end{equation*}
and the corresponding equation for $X$ is
\begin{align*}
&\dfrac{J^{\gamma}X_t}{\omega_0}-\gamma\kappa[{F^{*}}_i^j{F^{-1}}_i^k X_{,k}]_{,j}+\kappa \dfrac{\rho_0}{\omega_0}J^{\gamma-1} X\\=-&\dfrac{\gamma}{\gamma-1}\kappa[{F^{*}}_i^j\partial_t{F^{-1}}_i^k(\omega_0J^{1-\gamma})_{,k}]_{,j}-\gamma J^{-1}(J_t)^2\\
 &+\partial_t{F^{*}}_i^j{v^i}_{,j}-\dfrac{\gamma}{\gamma-1}[{F^{*}}_i^j{F^{-1}}_i^k(\omega_0J^{1-\gamma})_{,k}]_{,j}+\rho_0\\&-\kappa J\partial_t({F^{-1}}_i^j)G^i_{,j}.
\end{align*}
then we can construct approximated solutions to this degenerate parabolic regularization just in a similar way to that used in Section \ref{s5}.
Noticing that $\rho_0=\omega_0^{\frac{1}{\gamma-1}}$ will not be smooth now, as a result, the regularity of the potential force term $G$ will also be different. Thus, we need to require a certain high space regularity for $G$ to guarantee the construction of approximated solutions in Section \ref{s5} and a prior estimates in Section \ref{ee} still workable.

First, since in our fixed-point framework, we want to get the regularity that $X\in H^4(\Omega)$, then from \eqref{appro1} and \eqref{initialdata}, we will need the regularity that $J\partial_t({F^{-1}}_i^j)G^i_{,j} \in H^2(\Omega)$. From \eqref{ineq:est0}, $D^2G$ contains the following kind integral,
\begin{align*}
\int_{\Omega}\dfrac{1}{|\eta(x,t)-\eta(z,t)|}J(D_{\eta})^a(\rho_0J^{-1})\,dz,\,\, a=1,2,3,
\end{align*}
then the highest order term of $D^3G$ scales like
\begin{align*}\int_{\Omega}\dfrac{(\eta(x,t)-\eta(z,t))\cdot D\eta(x,t)}{|\eta(x,t)-\eta(z,t)|^3}[D^3\eta D^2(\omega_0^{\frac{1}{\gamma-1}}J^{-1})+D^3(\omega_0^{\frac{1}{\gamma-1}}J^{-1})]\,dz.
\end{align*}
With Young's inequality, we have that, for $\frac{1}{p}+\frac{1}{q}-1=\frac{1}{2}$,
\begin{multline*}
\|\int_{\Omega}\dfrac{(\eta(x,t)-\eta(z,t))\cdot D\eta(x,t)}{|\eta(x,t)-\eta(z,t)|^3}D^3(\omega_0^{\frac{1}{\gamma-1}}J^{-1})\,dz\|_0\\\leq C\|\dfrac{1}{|x|^2}\|_{L^p(\Omega)}\|D^3(\omega_0^{\frac{1}{\gamma-1}}J^{-1})\|_{L^q(\Omega)}.
\end{multline*}
We also have that
\begin{align*}
D^3(\omega_0^{\frac{1}{\gamma-1}}J^{-1})=\sum_{b=0}^3c_b\omega_0^{\frac{1}{\gamma-1}-b}D^{3-b}J^{-1}.
\end{align*}
Then we choose suitable $p,q$ for every $a$ such that
\begin{equation*}\|\dfrac{(\eta(x,t)-\eta(z,t))\cdot D\eta(x,t)}{|\eta(x,t)-\eta(z,t)|^3}D^3(\omega_0^{\frac{1}{\gamma-1}}J^{-1})\,dz\|_0
\end{equation*} can be bounded. Recall \eqref{w0000}, $\omega_0^{3(\frac{1}{\gamma-1}-1)}$ is integrable in $\Omega$ when $\gamma >1$, we can take $p=1,q=2$ for $b=0,1$. On the other hand, if $b=3$, we will require that $\omega_0^{q(\frac{1}{\gamma-1}-3)}$ is integrable in $\Omega$, which implies that
\begin{equation}
\label{opoperoe}
q(\frac{1}{\gamma-1}-3)>-3.
\end{equation}
However, we also need
\begin{equation}
\label{pcondition}
1\leq p <\frac{3}{2}
\end{equation}
to guarantee that $\dfrac{1}{|x|^{2p}}$ is integrable in $\Omega$. So we require that $1<\gamma<3$, which will enable us to find suitable $p,q$ satisfying the condition \eqref{opoperoe} and \eqref{pcondition}.

\noindent
Following the same analysis we used above, we can have that $\|D^3G\|_0$ is bounded when $1<\gamma<3$, which implies that we can construct our approximated solutions by a similar fixed-point methodology we used for $\gamma=2$.

Furthermore, the a prior estimates in Section \ref{ee} are correct under the condition $1<\gamma<3$. For instance, since we have
\begin{equation*}
\bar\partial\rho_0=\rho_0\dfrac{\bar\partial\omega_0}{\omega_0},
\end{equation*}
then with the higher-order Hardy's inequality, $\kappa$-independent energy estimates for time and tangential derivatives still hold true.

So for $1 < \gamma < 3$, we can get the local well-posedness by doing the similar proof as $\gamma=2$ in Section \ref{s5} -- \ref{s7}. Details can be seen in \cite{DS_2010}.

\section*{Acknowledgments}The author would like to thank Prof. Zhouping Xin for suggesting us the project
and many stimulating discussions. The author was in part supported by NSFC (grant No.11171072,
11121101 and 11222107), NCET-12-0120, Innovation Program of Shanghai Municipal Education Commission (grant
No.12ZZ012), Shanghai Talent Development Fund and SGST 09DZ2272900.


\begin{thebibliography}{99}
\bibitem{AM}
\newblock D. Ambrose, N. Masmoudi,
\newblock Well-posedness of 3D vortex sheets with surface tension,
\newblock Comm. Math. Sci.,5:391-430, 2007.

\bibitem{AD_13}
\newblock T. Alazard, J. Delort,
\newblock Global solutions and asymptotic behavior for two dimensional gravity water
waves,
\newblock Preprint, 2013.

\bibitem{C}
\newblock S. Chandrasekhar,
\newblock An Introduction to the Study of Stellar Structure,
\newblock Univ. of Chicago Press, 1938.

\bibitem{cheng_08}
\newblock A. Cheng, D. Coutand, S. Shkoller,
\newblock On the motion of vortex sheets with surface tension in the 3D Euler equations with vorticity,
 \newblock Commun. Pure Appl. Math. 61, 2008, 1715--1752.

\bibitem{DS_06}
\newblock D. Coutand, S. Shkoller,
\newblock On the interaction between quasilinear elastodynamics
and the Navier--Stokes equations,
\newblock Arch. Rational Mech. Anal. 179, 2006, 303--352.
\bibitem{DS}
\newblock D. Coutand, S. Shkoller,
\newblock Well-posedness of the free-surface incompressible Euler equations with or without surfce tension,
\newblock J.Amer.Math.Soc., 20(3), 2007, 829--930.
\bibitem{DS_09}
\newblock D. Coutand, S. Shkoller,
\newblock A Priori Estimates for the Free-Boundary 3D Compressible Euler Equations in Physical Vacuum,
\newblock  Comm. Math. Phys., Vol.296, No.12, 2010, 559--587.
\bibitem{DS_2009}
\newblock D. Coutand, S. Shkoller,
\newblock Well-posedness in smooth function spaces for the moving-boundary 1-D compressible Euler equations in physical vacuum,
\newblock Comm. Pure Appl. Math, Vol.64, No.3, 2011, 328--366.

\bibitem{DS_2010}
\newblock D. Coutand, S. Shkoller
\newblock \emph{Well-posedness in smooth function spaces for the moving boundary 3-D compressible Euler equations in physical vacuum}
\newblock Arch. Rational Mech. Anal.,Vol. 206, (2012), 515--616



\bibitem{Cox_68}
\newblock J. Cox, R. Giuli,
\newblock Principles of stellar structure, I.,II,
\newblock New York: Gordon and
Breach, 1968

\bibitem{TY_2002}
\newblock Y. Deng, T. Liu, T. Yang, Z. Yao,
\newblock Solutions of Euler--Poisson Equations for Gaseous Stars,
\newblock Arch. Rational Mech. Anal. 164, 2002, 261--285.

\bibitem{Di_83}
\newblock R. DiPerna,
\newblock Convergence of the viscosity method for isentropic gas dynamics.
\newblock Comm. Math. Phys., Vol. 91, No.1, 1983, 1–-30.

\bibitem{Sh_2001}
\newblock S. Engelberg, H. Liu, E. Tadmor
\newblock Critical Thresholds in Euler--Poisson Equations,
\newblock Indiana University Mathematics Journal, Vol.50, No.1, 2001.

\bibitem{Gu_2011}
\newblock X. Gu, Z. Lei,
\newblock Well-posedness of 1-D compressible Euler-Poisson equations with physical vacuum,
\newblock Journal of Differential Equations, 252 (2012) pp. 2160--2188.

\bibitem{Ion_13}
\newblock A. Ionescu and F. Pusateri,
\newblock Global solutions for the gravity water waves system in 2D,
\newblock Preprint,2013.

\bibitem{J_08}
\newblock J. Jang,
\newblock Nonlinear Instability in Gravitational Euler-Poisson system for $\gamma$=6/5,
\newblock Arch. Rational Mech. Anal., 188, 2008, no. 2, 265--307.

\bibitem{J_09}
\newblock J. Jang,
\newblock Local Well-Posedness of Dynamics of Viscous Gaseous Stars,
\newblock Arch. Rational Mech. Anal., 195, 2009, 797--863.

\bibitem{J_13}
\newblock J. Jang,
\newblock Nonlinear Instability Theory of Lane-Emden stars,
\newblock to appear in Comm. Pure Appl. Math.

\bibitem{MJ_2009}
\newblock J. Jang, N. Masmoudi,
\newblock Well-posedness for compressible Euler equations with physical vacuum singularity,
\newblock Comm. Pure Appl. Math. 62, 2009, 1327--1385.

\bibitem{MJ_2010}
    \newblock J. Jang, N. Masmoudi,
    \newblock Well-posedness of compressible Euler equations in a physical vacuum,
    \newblock to appear in Comm. Pure Appl. Math.

\bibitem{MJ_11}
\newblock J. Jang, N. Masmoudi,
\newblock Vacuum in Gas and Fluid dynamics,
\newblock The IMA Volumes in Mathematics and its Applications, 153 Springer, 2011, 315--329.
\bibitem{MJ_12}
\newblock J. Jang, N. Masmoudi,
\newblock Well and ill-posedness for compressible Euler equations with vacuum,
\newblock Journal of Mathematical Physics 53, 2012, no.11, 115625.
\bibitem{JI_2013}
\newblock J. Jang, I. Tice,
\newblock Instability theory of the Navier-Stokes-Poisson equations,
\newblock Anal. PDE 6, 2013, no. 5, 1121–-1181.

\bibitem{Kufner_1985}
    \newblock A. Kufner,
    \newblock Weighted Sobolev Spaces,
    \newblock Wiley-Interscience, 1985.
\bibitem{Lannes}
\newblock D. Lannes,
\newblock Well-posedness of the water-waves equations,
\newblock J. Amer. Math. Soc., 18, 2005, 605--654.


\bibitem{L_14}
\newblock Z. Lei,
\newblock Global Well-posedness of Incompressible Elastodynamics in 2D,
\newblock preprint 2014.

\bibitem{LLZ_08}
\newblock Z. Lei, C. Liu, Y. Zhou,
\newblock Global Solutions for Incompressible
Viscoelastic Fluids,
\newblock Arch. Rational Mech. Anal. 188 (2008) 371–-398

\bibitem{LSZ_12}
\newblock Z. Lei, T. Sideris, Y. Zhou,
\newblock Almost global existence of solutions to the incompressible elasticity in 2D,
\newblock To appear in Tran. AMS.

\bibitem{LZ_05}
\newblock Z. Lei, Y. Zhou,
\newblock Global existence of classical solutions for 2D Oldroyd model via the
incompressible limit,
\newblock SIAM J. Math. Anal. 37(3), 797–-814 (2005).



\bibitem{HL_2010}
    \newblock H. Li, A. Matsumura and G. Zhang,
    \newblock Optimal decay rate of the compressible Navier-Stokes-Poisson System in $R^3$,
    \newblock Arch. Ration. Mech. Anal.,196, 2010, 681--713

\bibitem{HL_2008}
\newblock H. Li, J. Li, Z. Xin,
\newblock Vanishing of Vacuum States and Blow-up Phenomena of the Compressible Navier-Stokes Equations,
\newblock Comm. Math. Phys., 281, 2008, 401--444.
\bibitem{Lin_87}
\newblock L. Lin,
\newblock On the vacuum state for the equations of isentropic gas dynamics.
\newblock J. Math. Anal. Appl. 121, 1987, 406–-425.

\bibitem{Lin_97}
\newblock S. Lin,
\newblock Stability of gaseous stars in spherically symmetric motions.,
\newblock SIAM J. Math. Anal. 28, 1997, no. 3, 539–-569.


\bibitem{Lindblad_2005}
    \newblock H. Lindblad,
    \newblock Well-posedness for the motion of an incompressible liquid with free surface boundary,
    \newblock Ann.of Math. (2) 162, no.1, 2005, 109--194.

\bibitem{Lind_2005}
    \newblock H. Lindblad,
    \newblock Well-posedness for the motion of a compressible liquid with free surface boundary,
    \newblock Comm. Math. Phys.,(2) 260, 2005, 319--392.

\bibitem{L_1996}
    \newblock T. Liu,
    \newblock Compressible Flow with Damping and Vacuum,
    \newblock Japan J. Indust. Apph Math., 13, 1996, 25--32.

\bibitem{Liu_80}
\newblock T. Liu, J. Smoller,
\newblock On the vacuum state for isentropic gas dynamics equations,
\newblock Adv. Math. 1, 1980, 345–-359.
\bibitem{LT_1997}
    \newblock T. Liu, T. Yang,
    \newblock Compressible Euler equations with vacuum,
    \newblock J. Differential Equations 140, no.2, 1997, 223--237.

\bibitem{LT_2000}
    \newblock T. Liu, T. Yang,
    \newblock Compressible flow with vacuum and physical singularity,
    \newblock Methods Appl.Anal. 7, no.3, 2000, 495--510.

\bibitem{Lions_96}
\newblock P. Lions, B. Perthame, P. Souganidis,
\newblock Existence and stability of entropy solutions
for the hyperbolic systems of isentropic gas dynamics in Eulerian and Lagrangian
coordinates.
\newblock Comm. Pure Appl. Math.,49(6), 1996, 599–-638.

\bibitem{XinL}
\newblock T. Luo, Z. Xin, T. Yang,
\newblock Interface behavior of compressible Navier-Stokes equations with vacuum,
\newblock SIAM J.Math.Anal. 31, 2000, 1175--1191.

\bibitem{Xin_2013}
\newblock T. Luo, Z. Xin, H. Zeng
\newblock Well-Posedness for the Motion of Physical Vacuum of the
Three-dimensional Compressible Euler Equations with or without
Self-Gravitation,
\newblock preprint



\bibitem{Makino_1986}
    \newblock T. Makino,
    \newblock On a local existence theorem for the evolution equation of gaseous stars,
    \newblock In: T.Nishida, M.Mimura, H.Fujii (eds) Patterns and waves,
    \newblock North-Hooland, Amesterdam, 1986.

\bibitem{Masmoudi_2012}
\newblock N. Masmoudi, F. Rousset,
\newblock Uniform regularity and vanishing viscosity limit for the free surface Navier-Stokes equations,
\newblock preprint, 2012

\bibitem{Schweizer_2005}
   \newblock B. Schweizer,
   \newblock On the three-dimensinal euler equations with a free boundary subject to surface tension,
   \newblock Ann.Inst.H.Poincare Anal. Non Lineaire 22, no.6, 2005, 753--781.

\bibitem{Taylor}
\newblock M. Taylor,
\newblock Partial Differential Equations, Vol. I-III,
\newblock Berlin-Heidelberg-New York: Springer, 1996

\bibitem{Tra_09}
\newblock Y. Trakhinin,
\newblock Local existence for the free boundary problem for the non-relativistic
and relativistic compressible Euler equations with a vacuum boundary condition.
\newblock Comm. Pure Appl. Math. 62, 2009, 1551-–1594.

\bibitem{Wu_1997}
    \newblock S. Wu,
    \newblock Well-posedness in Sobolev spaces of the full water wave problem in 2-D,
    \newblock Invent.Math. 130, no.1, 1997, 39--72.

\bibitem{Wu_1999}
    \newblock S. Wu,
    \newblock Well-posedness in Sobolev spaces of the full water wave problem in 3-D,
    \newblock J.Amer.Math.Soc. 12, no.2, 1999, 445--495.

\bibitem{Wu_09}
\newblock S. Wu,
\newblock Almost global well-posedness of the 2-D full water wave problem,
\newblock Invent. Math., 177, 2009, 45–-135.

\bibitem{Wu_11}
\newblock S. Wu,
\newblock Global well-posedness of the 3-D full water wave problem,
\newblock Invent. Math., 184, 2011, 125–-220.



\bibitem{Xin}
\newblock Z. Xin,
\newblock Blowup of smooth solutions to the compressible Navier-Stokes equation with compact density,
\newblock Comm. Pure Appl.Math. 51, no.3, 1998, 229--240.
\bibitem{Xu_05}
\newblock C. Xu, T. Yang,
\newblock Local existence with physical vacuum boundary condition to Euler
equations with damping,
\newblock J. Differ. Equ. 210, 2005, 217–-231.

\bibitem{Yang_06}
\newblock T. Yang,
\newblock Singular behavior of vacuum states for compressible fluids,
\newblock J. Comput. Appl.Math. 190, 2006, 211–-231.
\bibitem{ZZ}
   \newblock P. Zhang, Z. Zhang,
   \newblock On the free boundary problem of 3-D incompressible Euler equaitons,
   \newblock Comm. Pure Appl. Math, vol.61, no.7, 2008, 877--940.




\end{thebibliography}
\end{document}